\documentclass[11pt, letterpaper]{article}
\pdfoutput=1
\usepackage[authoryear]{natbib}
\usepackage{amsmath,amssymb,amsfonts,amsthm,bbm}
\usepackage{epic,eepic,epsfig,longtable}
\usepackage{multirow,verbatim}
\usepackage{array}
\usepackage{epsfig}
\usepackage{setspace}
\usepackage{CJK}
\usepackage{color}
\usepackage{graphicx}

\makeatletter
\def\singlespace{\def\baselinestretch{1}\@normalsize}

\makeatletter
\def\singlespace{\def\baselinestretch{1}\@normalsize}

\numberwithin{equation}{section}

\renewcommand{\hat}{\widehat}

\renewcommand{\hat}{\widehat}

\newcommand{\bfm}[1]{\ensuremath{\mathbf{#1}}}

\def\ba{\bfm a}   \def\bA{\bfm A}  
   \def\bB{\bfm B}  
\def\bc{\bfm c}     
     
\def\be{\bfm e}     \def\EE{\mathbb{E}}
    
\def\bg{\bfm g}     
   \def\bH{\bfm H}  
   \def\bI{\bfm I}

     \def\PP{\mathbb{P}}
     
     \def\RR{\mathbb{R}}

\def\bu{\bfm u}     
\def\bv{\bfm v}     
\def\bw{\bfm w}     
\def\bx{\bfm x}   \def\bX{\bfm X}  
\def\by{\bfm y}     
\def\bz{\bfm z}     

\def\cH{\mathcal{H}}
\def\cH{\mathcal{U}}
\def\cH{\mathcal{V}}
\def\cH{\mathcal{R}}
\def\cH{\mathcal{T}}
\def\cH{\mathcal{A}}

 \def\cA{{\cal  A}}
 \def\cB{{\cal  B}}
 \def\cC{{\cal  C}}
 \def\cD{{\cal  D}}

 \def\cH{{\cal  H}}
 \def\cI{{\cal  I}}
 \def\cJ{{\cal  J}}
 \def\cK{{\cal  K}}

 \def\cQ{{\cal  Q}}
 
 \def\cS{{\cal  S}}
 \def\cT{{\cal  T}}

\def\bzero{\bfm 0}

\newcommand{\bfsym}[1]{\ensuremath{\boldsymbol{#1}}}

\def\bzero{\bfsym 0}
\def\bone{\bfsym 1}
 \def\balpha{\bfsym \alpha}
 \def\bbeta{\bfsym \beta}
              
            \def\bDelta {\bfsym {\Delta}}
               
 \def\bmu{\bfsym {\mu}}

              \def\bSigma{\bfsym \Sigma}
         \def\bLambda {\bfsym {\Lambda}}

 \def\bxi{\bfsym {\xi}}
 
  \def\bUpsilon{\bfsym \Upsilon}
  \def\bvarpi{\bfsym \varpi}

\DeclareMathOperator{\argmin}{argmin}

\DeclareMathOperator{\cov}{cov}

\DeclareMathOperator{\diag}{diag}

\DeclareMathOperator{\var}{var}

\DeclareMathOperator{\tr}{tr}

\def\newpage{\vfill\eject}

\def\var{\mbox{var}}

\def\today{\ifcase\month\or
  January\or February\or March\or April\or May\or June\or
  July\or August\or September\or October\or November\or December\fi
  \space\number\day, \number\year}

\newdimen\biblioindent    \biblioindent=30pt

 at 8truept

\newcommand{\beq}{\begin{equation}}
  \newcommand{\eeq}{\end{equation}}
\newcommand{\beqn}{\begin{eqnarray}}
  \newcommand{\eeqn}{\end{eqnarray}}
\newcommand{\beqnn}{\begin{eqnarray*}}
  \newcommand{\eeqnn}{\end{eqnarray*}}

\allowdisplaybreaks
\usepackage{verbatim}
\renewcommand{\baselinestretch}{1.66}
\newtheorem{lem}{Lemma}

\newtheorem{thm}{Theorem}
\newtheorem{cor}{Corollary}
\newtheorem{thmx}{Theorem}
\newtheorem{prop}{Proposition}

\newcounter{CondCounter}

\newcommand{\subgnorm}[1]{\lVert#1\rVert_{\psi_2}}

\newcommand{\linfnorm}[1]{\lVert#1\rVert_{\infty}}
\newcommand{\lzeronorm}[1]{\lVert#1\rVert_0}
\newcommand{\lqnorm}[1]{\lVert#1\rVert_q}
\newcommand{\lonenorm}[1]{\lVert#1\rVert_1}
\newcommand{\ltwonorm}[1]{\lVert#1\rVert_2}

\begin{document}

\title{\bf Optimal estimation of functionals of \\ high-dimensional mean and covariance matrix\footnote{Supported by NSF grants DMS-1712591 and DMS-1947097 and NIH grant R01-GM072611.} }

 \author{
  Jianqing Fan \vspace{0.01in}\\
  Department of Operations Research and Financial Engineering,\\
  Princeton University, Princeton, NJ 08544 \vspace{0.1in} \\
  Haolei Weng \vspace{0.01in} \\
  Department of Statistics and Probability, \\
  Michigan State University, East Lansing, MI 48824 \vspace{0.1in} \\
  Yifeng Zhou  \vspace{0.01in} \\
  Department of Operations Research and Financial Engineering,\\
  Princeton University, Princeton, NJ 08544
  }

\date{}

\maketitle

\begin{spacing}{1.2}
\begin{abstract}
Motivated by portfolio allocation and linear discriminant analysis, we consider estimating a functional $\bmu^T \bSigma^{-1} \bmu$ involving both the mean vector $\bmu$ and covariance matrix $\bSigma$. We study the minimax estimation of the
functional in the high-dimensional  setting where $\bSigma^{-1} \bmu$ is sparse. Akin to past works on functional estimation, we show that the optimal rate for estimating the functional undergoes a phase transition between regular parametric rate and some form of high-dimensional estimation rate.
We further show that the optimal rate is attained by a carefully designed plug-in estimator
based on de-biasing, while a family of naive plug-in estimators are proved to fall short. We further generalize the estimation problem and techniques that allow robust inputs of mean and covariance matrix estimators.
Extensive numerical experiments lend further supports to our theoretical results.
\end{abstract}


\end{spacing}

\newpage

\section{Introduction}
\label{sec:introduction}

In multivariate statistics, the mean vector $\bmu$ and covariance matrix $\bSigma$ play a critical role in a variety of statistical procedures such as linear discriminant analysis (LDA), multivariate analysis of variance and principle component analysis (PCA). See \cite{anderson03} for a comprehensive mathematical treatment of classical multivariate analysis. 
Modern multivariate statistics confronts new statistical challenges due to arrival of big high-dimensional data \citep{jt09, fhl14}. For example, when the dimension $p$ is comparable to or much larger than the sample size $n$, sample covariance matrix is a poor estimate for $\bSigma$ \citep{bsy88, by93}, and classical PCA becomes inconsistent \citep{p07, jl09, wf17} when the eigenvalues are not sufficiently spiked.

To the challenges arising from high dimensionality, many new theory and methods have been proposed. For example, various regularization techniques have been proposed to estimate large covariance matrix under different matrix structural assumptions such as sparsity \citep{ka08, bl08b, lf09, cl11a}, conditional sparsity \citep{ffl08, flm11, flm13}, and smoothness \citep{fb07, bl08a, czz10, cy12}. We refer to the two review papers \citep{fll16, crz16} and references therein for many other important related works. There have also been active researches on the inferential theory of the high-dimensional mean vector. A large body of work focuses on developing powerful one-sample or two-sample tests where conventional approaches like Hotelling's $T^2$ test fail \citep{bs96, sd08, sri09, cq10, zcx13, clx14, wpl15}. 

Let $\bx_1, \ldots, \bx_n \in \RR^p$ be independently and identically distributed (i.i.d.) random vectors with $\EE(\bx_i)=\bmu$ and $\cov(\bx_i)=\bSigma$.  Motivated by the sparse portfolio allocations and sparse linear discriminants (see Section \ref{sec:example} for details), the primary goal of this paper is to estimate the functional $\bmu^T\bSigma^{-1}\bmu$ based on the observations $\{\bx_i\}_{i=1}^n$, under the assumption that $\bSigma^{-1}\bmu$ is (approximately) sparse.
Estimation of functionals has been studied in great generality in nonparametric statistics \citep{nemir00}. Minimax and adaptive theory for the estimation of linear functionals \citep{ik84, dl91a, dl91b, kt01, cl04, cl05a, bc09}, and quadratic functionals \citep{br88, dn90, fan91, el96, cl05b, but07, cct17}, have been originally established for Gaussian white noise model and then extended to the convolution model, among others. Along this line of works, an elbow phenomenon has been recurrently discovered: the optimal rate of convergence for some functionals exhibits a phase transition between the regular parametric rate and certain forms of nonparametric rate. 

We will investigate the optimal rate of convergence for estimating the functional $\bmu^T\bSigma^{-1}\bmu$. Akin to the functional estimation in nonparametric statistics, we will reveal that the minimax estimation rate of $\bmu^T\bSigma^{-1}\bmu$, in the high-dimensional multivariate problem, undergoes a transition between parametric rate and some type of high-dimensional estimation rate. Moreover, we show that the optimal rate is achieved by a carefully designed plug-in estimator based on a de-biased $\ell_1$-regularized estimator of $\bSigma^{-1}\bmu$. On the  contrary, a family of naive plug-in estimators are proved to fall short. A similar phase transition phenomenon was uncovered on the  estimation for quadratic functional of sparse covariance matrices $\bSigma$ \citep{frw15}. We also refer to \cite{gwcl18} regarding applying de-biasing to obtain optimal estimators for some one-dimensional functionals in high-dimensional linear models; and \cite{cg17} on construction of confidence interval for linear functionals in the sparse high-dimensional linear regression model. 

The remainder of the paper is organized as follows. Section \ref{sec:example} introduces two motivating examples and the basic setup. Section \ref{sec2} describes our estimator and studies in detail the minimax estimation property. Section \ref{sec6} presents numerical performance of the proposed estimator on both synthetic and real datasets. To improve readability, all the proof is relegated to Section \ref{sec:proofs}.

\noindent
\textbf{Notation}. For $\ba \in \RR^p$, denote $\lqnorm{\ba}=(\sum_{i=1}^p|a_i|^q)^{\frac{1}{q}}$ for $q \in (0,\infty), \lzeronorm{\ba}=\sum_{i=1}^p|a_i|^0$ with convention $0^0=0$, and $\linfnorm{\ba}=\max_{1\leq i\leq p}|a_i|$. Given a symmetric matrix $\bA=(a_{ij}) \in \RR^{p\times p}$, $\lambda_{\max}(\bA)$ and $\lambda_{\min}(\bA)$ represent its largest and smallest eigenvalues respectively, and $\|\bA\|_{\max}=\max_{ij}|a_{ij}|, \delta_{\bA}=\max_{i}|a_{ii}|$. For $a,b\in \RR$, $a\wedge b=\min(a,b), a\vee b=\max(a,b)$. For $c>0$, $\lceil c\rceil$ denotes the smallest integer greater than or equal to $c$, and $\lfloor c \rfloor$ is the largest integer less than or equal to $c$. Moreover, $f(n) \lesssim g(n)~(f(n) \gtrsim g(n))$ means there exists some constant $C>0$ such that $ f(n) \leq C g(n)~(f(n)\geq Cg(n))$ for all $n$; $f(n) \asymp g(n)$ if and only if $f(n) \lesssim g(n)$ and $f(n) \gtrsim g(n)$; $f(n) \gg g(n)$ is equivalent to $g(n)=o(f(n))$. We use $\bI_p$ to denote the $p\times p$ identity matrix, and $B_q(r)=\{\bu \in \RR^p: \lqnorm{\bu}\leq r\}$ for the $\ell_q$ ball with radius $r$.

\section{Preliminaries and examples}
\label{sec:example}

Suppose that $\bx_1, \ldots, \bx_n$ are independent copies of $\bx  \in \RR^p$ with $\EE(\bx)=\bmu$ and $\cov(\bx)=\bSigma$. We consider the problem of estimating the functional $\bmu^T\bSigma^{-1}\bmu$ using the data $\{\bx_i\}_{i=1}^n$. Throughout the paper we assume, unless otherwise stated, that $\bx$ is sub-gaussian. That is, $\bx=\bSigma^{1/2}\by+\bmu$ and the zero-mean isotropic random vector $\by$ satisfies
\begin{align}
\label{subgaussian:def}
\PP(|\bc^T\by|\geq t)\leq 2\exp(-t^2/\nu^2), \quad \mbox{for all~~} t\geq 0, ~\|\bc\|_2=1,
\end{align}
with $\nu>0$ being a constant. We study the estimation problem under minimax framework. The central goal is to characterize the minimax rate of the estimation error given by
\begin{align*}
\inf_{\hat{\theta}}\sup_{(\bmu, \bSigma) \in \cH}\EE|\hat{\theta}-\bmu^T\bSigma^{-1}\bmu|,
\end{align*}
where the infimum is taken over all measurable estimators, and $\cH$ is some parameter space under consideration. We first derive a lower bound for the error to reveal the effect of high dimension.
\begin{prop}
\label{prop:one}
Consider $\cH=\{(\bmu, \bSigma): \bmu^T\bSigma^{-1}\bmu \leq c\}$, where $c>0$ is a fixed constant. If $p\geq n^2$, it holds that
\begin{align*}
\inf_{\hat{\theta}}\sup_{(\bmu, \bSigma) \in \cH}\EE|\hat{\theta}-\bmu^T\bSigma^{-1}\bmu| \geq \tilde{c},
\end{align*}
where $\tilde{c}>0$ is a constant that depends on $c$.
\end{prop}
Proposition \ref{prop:one} demonstrates that it is impossible to consistently estimate the functional $\bmu^T\bSigma^{-1}\bmu$, under the scaling $p\geq n^2$ which is not uncommon in high-dimensional problems. To overcome the difficulty, we need a more structured parameter space. However, it is not clear what kind of simple constraints would make the problem solvable and practical.  As a part of the contribution, we find the following parameter spaces with sparsity constraint 
$$
\cH(s, \tau) =\Big\{(\bmu,\bSigma): \lzeronorm{\bSigma^{-1}\bmu} \leq s, ~\bmu^T\bSigma^{-1}\bmu \leq \tau,~ c_L\leq \lambda_{\min}(\bSigma)\leq \delta_{\bSigma}\leq c_U \Big \},  
$$
or more generally the approximate sparsity constraint
$$
\cH_q(R, \tau) =\Big\{(\bmu,\bSigma): \lqnorm{\bSigma^{-1}\bmu}^q \leq R,~\bmu^T\bSigma^{-1}\bmu \leq \tau, c_L\leq \lambda_{\min}(\bSigma)\leq \delta_{\bSigma} \leq c_U \Big \},\quad q\in (0,1],
$$
suffice for our problem, where  $\delta_{\bSigma} = \max_{i \in [p]} \sigma_{ii}$ (recall the notation $\delta_{\bA}=\max_{i}|a_{ii}|$), $0<c_L<c_U$ are fixed constants; $s, R$ and $\tau$ can scale with $n$ and $p$. For notational simplicity, we have suppressed the dependence of $\cH(s,\tau)$ and $\cH_q(R, \tau)$ on $ c_L,c_U$. The sparsity assumption on $\bSigma^{-1}\bmu$ can be well justified in several multivariate statistics problems, of which two will be introduced shortly in Sections \ref{portfolio} and \ref{hlda}. Moreover, by setting $\bSigma=\sigma^2\bI_p$ and assuming normality for $\bx$, our problem is reduced to quadratic functional estimation over sparsity classes in the Gaussian sequence model \citep{fan91, cct17}. However, with the covariance matrix $\bSigma$ being unknown, results in \cite{cct17} can not be directly generalized. Delicate analyses are required to establish minimax optimality results, as will be shown in Section \ref{sec2}.

\subsection{The mean-variance portfolio optimization}
\label{portfolio}
The mean-variance portfolio optimization method has been widely adopted by both institutional and retail investors ever since it was proposed by \cite{mark52}. Markowitz’s theory was highly influential and can be regarded as one of the foundations in modern finance. It can be expressed as an optimization problem with the solution determining proportion of each asset in a portfolio by maximizing the expected return under risk constraint, where the risk is measured by the variance of the portfolio. Specifically, let $\bx \in \RR^p$ be the excess return of $p$ risky assets, with $\EE(\bx)=\bmu, \cov(\bx)=\bSigma$. Markowitz’s portfolio optimization problem is
\begin{align*}
\max_{\bw}~ \EE(\bw^T\bx)=\bw^T\bmu \quad \mbox{subject to}~~\var(\bw^T\bx)=\bw^T\bSigma\bw \leq \sigma^2,
\end{align*}
where $\sigma$ is the prescribed risk level. The optimal portfolio $\bw^*$ for the risky assets admits (the remaining invests in cash, including short positions) the explicit expression,
\begin{equation} \label{eq:fan0}
\bw^*=\frac{\sigma}{\sqrt{\bmu^T\bSigma^{-1}\bmu}}\bSigma^{-1}\bmu.
\end{equation}
The functional $\bmu^T\bSigma^{-1}\bmu$ is the square of the maximum Sharpe ratio, which measures the risk-adjusted performance of the optimal portfolio. We need both $\bmu^T\bSigma^{-1}\bmu$ and $\bSigma^{-1}\bmu$ in order to construct the optimal portfolio allocation $\bw^*$ (See Section~\ref{sec4.2}), which forms the focus of our studies.  Since the number of assets $p$ can be large compared to $n$, the number of observed return vectors, the estimation of mean-variance efficient portfolios is faced with great challenges in the high-dimensional regime \citep{kz07, blw09, ka10}. One stream of research has been focused on the construction of sparse portfolios via regularizations \citep{bddgl09, dgnu09, fzy12}. We refer to \cite{alz17} for a detailed list of references. Our sparsity assumption on the optimal portfolio weight, is well aligned with this line of works.


\subsection{High-dimensional linear discriminant analysis}
\label{hlda}
Linear discriminant analysis (LDA) is one of the most classical classification techniques in statistics and machine learning. Consider the binary classification problem where $\bx$ is a $p$-dimensional normal vector drawn with equal probability from one of the two distributions $N(\bmu_1,\bSigma)$ and $N(\bmu_2,\bSigma)$. It is well known that, Fisher's linear discriminant rule: classify $\bx$ to class 1 if and only if $(\bmu_1-\bmu_2)^T\bSigma^{-1}\big(\bx-\frac{\bmu_1+\bmu_2}{2}\big)\geq 0$, achieves the optimal classification error given by
\begin{align*}
R_{opt}=\Phi(-\Delta/2), ~~\Delta=\sqrt{(\bmu_1-\bmu_2)^T\bSigma^{-1}(\bmu_1-\bmu_2)},
\end{align*}
where $\Phi(\cdot)$ is the standard normal distribution function \citep{anderson03}. The functional $(\bmu_1-\bmu_2)^T\bSigma^{-1}(\bmu_1-\bmu_2)$ is the square of the signal-to-noise ratio $\Delta$, measuring the fundamental difficulty of the classification problem. The classical LDA procedure approximates Fisher's rule by replacing the unknown parameters $\bmu_1,\bmu_2,\bSigma$ by their sample versions. Consistency results under the classical asymptotic framework when $p$ is fixed have been well established \citep{anderson03}. However, in the high-dimensional settings, the standard LDA can be no better than random guess \citep{bl04}. Various high-dimensional LDA approaches have been proposed under the sparsity assumption on $\bmu_1-\bmu_2$ or $\bSigma$ \citep{ffy08, swdw11, mai13}. An alternative approach to sparse linear discriminant analysis imposes sparsity directly on $\bSigma^{-1}(\bmu_1-\bmu_2)$, based on the key observation that Fisher's rule depends on $\bmu_1-\bmu_2$ and $\bSigma$ only through the product $\bSigma^{-1}(\bmu_1-\bmu_2)$ \citep{cl11b, mzy12, cz18}. Such sparsity assumption is precisely what we have made in the paper. We should emphasize that different from the aforementioned works with the main focus on excess misclassification risk, our analysis centers on the estimation of $(\bmu_1-\bmu_2)\bSigma^{-1}(\bmu_1-\bmu_2)$, the quantity that characterizes the intrinsic difficulty of the classification problem.

It is important to observe that the functional estimation in the LDA problem, at first glance, looks different from our problem formulated at the beginning of Section \ref{sec:example}, because it involves two sets of samples. However, it is possible to extend our results to the LDA setting by a simple adaptation. Towards that goal, let $\{\bx_i\}_{i=1}^{n_i}$ and $\{\by_i\}_{i=1}^{n_2}$ be two sets of i.i.d. random samples from $N(\bmu_1,\bSigma)$ and $N(\bmu_2,\bSigma)$ respectively. Define the parameter space
\begin{align*}
\cI(s,\tau)=\Big\{(\bmu_1, \bmu_2,\bSigma):~ &\lzeronorm{\bSigma^{-1}(\bmu_1-\bmu_2)} \leq s, ~(\bmu_1-\bmu_2)^T\bSigma^{-1}(\bmu_1-\bmu_2) \leq \tau, \\
&c_L\leq \lambda_{\min}(\bSigma)\leq \delta_{\bSigma} \leq c_U \Big \}.
\end{align*}
For the moment, we write $\hat{\theta}_{\{\bx_i\}_{i=1}^{n_1},\{\by_i\}_{i=1}^{n_2}}$ to clarify that the estimator is a function of two sets of samples. We are able to lower bound the minimax error in the two-sample problem by the error from the one-sample problem,
\begin{align*}
&\inf_{\hat{\theta}} \sup_{(\bmu_1,\bmu_2,\bSigma)\in \cI(s,\tau)}\EE|\hat{\theta}_{\{\bx_i\}_{i=1}^{n_1},\{\by_i\}_{i=1}^{n_2}}-(\bmu_1-\bmu_2)^T\bSigma^{-1}(\bmu_1-\bmu_2)| \\
\geq &\inf_{\hat{\theta}} \sup_{(\bmu_1,\bSigma)\in \cH(s,\tau)}\EE|\hat{\theta}_{\{\bx_i\}_{i=1}^{n_1},\{\by_i\}_{i=1}^{n_2}}-\bmu_1\bSigma^{-1}\bmu_1| \\
\geq & \inf_{\hat{\theta}} \sup_{(\bmu_1,\bSigma)\in \cH(s,\tau)}\EE|\hat{\theta}_{\{\bx_i\}_{i=1}^{n_1+2n_2}}-\bmu_1\bSigma^{-1}\bmu_1|.
\end{align*}
Here, the first inequality is obtained by setting $\bmu_2=\bzero$; the second inequality holds because each $\by_i~(1\leq i\leq n_2)$ can be replaced by $\frac{\bx_{n_1+2i-1}-\bx_{n_1+2i}}{\sqrt{2}}$, where $\{\bx_{n_1+j}\}_{j=1}^{2n_2}$ are additional independent samples from $N(\bmu_1,\bSigma)$. As will be shown in Section \ref{general:estimator:sec}, a matching upper bound can be derived when $n_1 \asymp n_2$. Similar arguments hold for the approximate sparsity class.

\section{Minimax estimation of the functional}

\label{sec2}

To fix ideas, we first present a detailed discussion for the exact sparsity class $\cH(s,\tau)$ in Sections \ref{rate-optimal:sparse}--\ref{general:estimator:sec}. Generalization of the main results to the approximate sparsity class $\cH_q(R,\tau)$  will be given in Section \ref{approximate:sparse:class}. We further discuss the functional estimation problem in the dense regime in Section \ref{dense:regime}. For notational simplicity, we set $\balpha=\bSigma^{-1}\bmu$ and $\theta=\bmu^T\bSigma^{-1}\bmu$, and denote the sample mean and sample covariance matrix by
\[
\hat{\bmu}=\frac{1}{n}\sum_{i=1}^n\bx_i, \quad \hat{\bSigma}=\frac{1}{n}\sum_{i=1}^n(\bx_i-\hat{\bmu})(\bx_i-\hat{\bmu})^T,
\]
respectively.

\subsection{Optimal estimation over exact sparsity classes}
\label{rate-optimal:sparse}


We first consider the estimation of $\balpha$, which will pave our way to the estimation of the functional $\theta$. Since 
$
\balpha=\argmin_{\bbeta\in \RR^p} \frac{1}{2}\bbeta^T \bSigma \bbeta-\bbeta^T\bmu$ and  $\balpha$ is sparse, a natural estimator for the vector $\balpha$ is the $\ell_1$-regularized M-estimator:
\begin{equation}\label{lasso:alpha}
\tilde{\balpha} \in \argmin_{\ltwonorm{\bbeta} \leq \gamma} \frac{1}{2}\bbeta^T \hat{\bSigma}\bbeta-\bbeta^T \hat{\bmu}+\lambda \lonenorm{\bbeta}.
\end{equation}
The constraint $\ltwonorm{\bbeta}\leq \gamma$ is necessary to ensure the existence of $\tilde{\balpha}$. Otherwise,  when $\hat {\bSigma}$ is degenerate,  it may hold with positive probability that no finite solution exists in \eqref{lasso:alpha}. 
Moreover, as will be seen in the proof of Theorem \ref{thm:one}, without that constraint the probability of nonexistence of the solution vanishes asymptotically. Nevertheless, we should rule out the rare event for finite samples since the minimax error considered in this paper is measured in expectation.

Given the estimator $\tilde{\balpha}$, we propose to estimate the functional $\theta$ by
\begin{equation}
\label{optimal:estimator}
\tilde{\theta}=2\hat{\bmu}^T\tilde{\balpha}-\tilde{\balpha}^T\hat{\bSigma}\tilde{\balpha}.
\end{equation}
The above estimator is motivated by the de-biasing ideas for statistical inference in high-dimensional linear models \citep{zz14, jm14a, jm14b, vbrd14, jm18}. We now give a detailed explanation in our case. Since $\theta=\bmu^T\balpha$, we would like to construct a plug-in estimator
$
(\bmu^*)^T\balpha^*,
$
for some estimators $\bmu^*$ and $\balpha^*$ of $\bmu$ and $\balpha$ respectively. We set $\bmu^*=\hat{\bmu}$ and choose a de-biased version of $\tilde{\balpha}$ for $\balpha^*$. In particular, it is known that $\tilde{\balpha}$ defined in \eqref{lasso:alpha} satisfies the Karush-Kuhn-Tucker (KKT) conditions\footnote{For simplicity, we assume $\ltwonorm{\tilde{\balpha}}< \gamma$. In fact that holds with high probability as seen in the proof of Theorem \ref{thm:one}.}:
\begin{align}
\label{kkt:condition}
\hat{\bmu}-\hat{\bSigma}\tilde{\balpha}=\lambda \hat{\bg},
\end{align}
where $\hat{\bg}$ is a subgradient of $\lonenorm{\bbeta}$ at $\bbeta=\tilde{\balpha}$. Multiplying both sides of \eqref{kkt:condition} by $\bSigma^{-1}$ and rearranging the terms yields
\begin{align}
\label{kkt:variant}
\tilde{\balpha}=\balpha+\underbrace{\bSigma^{-1}(\hat{\bmu}-\hat{\bSigma}\balpha)}_{:=\bDelta_1}+\underbrace{(\bI_p-\bSigma^{-1}\hat{\bSigma})(\tilde{\balpha}-\balpha)}_{:=\bDelta_2}-\lambda \bSigma^{-1} \hat{\bg}.
\end{align}
Observe that for $(\bmu,\bSigma)\in \mathcal{H}(s,\tau)$, $\linfnorm{\mathbb{E}\bDelta_1}=n^{-1}\linfnorm{\balpha}=O(n^{-1}\sqrt{\tau})$; and $\linfnorm{\bDelta_2}\leq \|\bI_p-\bSigma^{-1}\hat{\bSigma}\|_{\max}\cdot \lonenorm{\tilde{\balpha}-\balpha}=O_p((1+\sqrt{\tau})n^{-1}s\log p)$ as implied by the error bound $\lonenorm{\tilde{\balpha}-\balpha}$ in the proof of Theorem \ref{thm:one} and Lemma \ref{rare:events}(iii) from Section \ref{tech:lemma}. Given that $\lambda \asymp \sqrt{(1+\tau)n^{-1}\log p}$ as shown in Theorem \ref{thm:one}, the major bias term of $\tilde{\balpha}$ is $-\lambda \bSigma^{-1} \hat{\bg}$. Therefore, the decomposition \eqref{kkt:variant} suggests a bias-corrected estimator
\begin{align*}
\balpha^*=\tilde{\balpha}+\lambda \bSigma^{-1} \hat{\bg}=\tilde{\balpha}+\bSigma^{-1}(\hat{\bmu}-\hat{\bSigma}\tilde{\balpha}),
\end{align*}
where the second equality is due to \eqref{kkt:condition}. Consequently,
\begin{align}
\label{almost:theta}
(\bmu^*)^T\balpha^*=\hat{\bmu}^T[\tilde{\balpha}+\bSigma^{-1}(\hat{\bmu}-\hat{\bSigma}\tilde{\balpha})].
\end{align}
Since $\bSigma^{-1} \hat{\bmu} \approx \bSigma^{-1} \bmu=\balpha$, we replace $\bSigma^{-1}\hat{\bmu}$  (two such terms) in \eqref{almost:theta} by $\tilde{\balpha}$  to make $(\bmu^*)^T\balpha^*$ a legitimate estimator. This leads to the proposed estimator in \eqref{optimal:estimator}.

The bias correction demonstrated in the above paragraph turns out to be crucial for the plug-in estimator $\tilde{\theta}$ to achieve the optimal rate. Without the de-biasing step, vanilla plug-in estimators will fall short. These two statements are formally stated in this and next subsections, whose proofs are relegated to Section~\ref{sec:proofs}. As a by-product, the minimax estimation error of $\tilde{\balpha}$ will be derived to shed more light on the functional estimation problem.

\begin{thm}
\label{thm:one}
Set $\lambda=t\nu\sqrt{\frac{(1+\tau)\log p}{n}}, \gamma = 2\sqrt{\frac{\tau}{c_L}}$ in \eqref{lasso:alpha}. If $\frac{s\log p}{n} <\tilde{c}$, then for all $t> 1 \vee \frac{2}{c_2}$,
\begin{align*}
&\sup_{(\bmu,\bSigma)\in \cH(s,\tau)} \EE\ltwonorm{\tilde{\balpha}-\balpha}^2 \leq c_1\cdot \Big(\frac{t^2(1+\tau)s\log p}{n}+\tau p^{-(c_2t-1) \wedge c} \Big), \\
& \sup_{(\bmu,\bSigma)\in \cH(s,\tau)} \EE|\tilde{\theta}-\theta| \leq c_3 \cdot \Big(\frac{t^2(1+\tau)s\log p}{n}+ (\tau+ \sqrt{\tau}) \cdot (n^{-\frac{1}{2}}+p^{-\frac{(c_2t-2) \wedge (c-1)}{2}}) \Big).
\end{align*}
Here, $c_1,c_2,c_3>0$ are constants possibly depending on $\nu,c_L,c_U$; $c$ is an arbitrary positive constant; $\tilde{c}\in (0,1)$ is a constant dependent on $c,v,c_L,c_U$ and $\tilde{c}\rightarrow 0$ as $c\rightarrow \infty$.
\end{thm}


The two terms $\tau p^{-(c_2t-1) \wedge c}$ and $p^{-\frac{(c_2t-2) \wedge (c-1)}{2}}$ appearing in the above bounds might not be optimally derived. However, they are both negligible by choosing a sufficiently large constant. In typical applications $\tau$ is constant.  Therefore, the results of Theorem~\ref{thm:one} are simplified to
\begin{equation}\label{eq:fan1}
\sup_{(\bmu,\bSigma)\in \cH(s,\tau)} \EE\ltwonorm{\tilde{\balpha}-\balpha}^2 = O\left ( \frac{s \log p}{n} \right), \quad \sup_{(\bmu,\bSigma)\in \cH(s,\tau)} \EE|\tilde{\theta}-\theta| =  O\left ( \frac{s \log p}{n} + \frac{1}{\sqrt{n}} \right).
\end{equation}
See Corollary~\ref{minimax:optimal} for more general results. Before we discuss in details the upper bounds, we present complementary lower bounds in the next theorem, which show that the rates in \eqref{eq:fan1} are optimal.

\begin{thm}
\label{lower:bound}

There exist positive constants $\{c_i\}_{i=1}^6$ possibly depending on $c_L, c_U, \nu$ such that,
\begin{itemize}
\item[(a)]if $\frac{s\log p/s}{n} < c_1, p/s>c_2$, then
\begin{align*}
 \inf_{\hat{\balpha}} \sup_{(\bmu,\bSigma)\in \mathcal{H}(s,\tau)}  \mathbb{E}\|\hat{\balpha}-\balpha\|_2^2 \geq  c_3 \cdot \Big [\tau \wedge \frac{(1+\tau)s\log(p/s)}{n} \Big],
\end{align*}
\item[(b)] if $\frac{s\log p}{n}<1$, then
\begin{align*}
\inf_{\hat{\theta}}\sup_{(\bmu,\bSigma)\in \mathcal{H}(s,\tau)}\mathbb{E}|\hat{\theta}-\theta| \geq c_4\cdot \Big[\tau \wedge  \frac{\tau+\sqrt{\tau}}{\sqrt{n}}\Big]+c_5 \cdot  \Big[\tau \wedge \frac{(1+\tau)s\log p}{n}\Big] c_0 \exp(-e^{2s^2p^{c_6 c_0-1}}),
\end{align*}
where $c_0$ can be any constant in $[0,1]$.
\end{itemize}
\end{thm}


According to Theorems \ref{thm:one} and \ref{lower:bound}, we conclude several important points as follows.
\begin{itemize}
\item[(1)] \emph{Estimation of $\balpha$.} Consider the scaling $\frac{s\log p}{n}=o(1)$. Suppose $p^{-\delta}\lesssim \frac{s\log p}{n}$ for some $\delta>0$. It is clear that we may choose the constants  $t$ and $c$ in Theorem \ref{thm:one} large enough so that
\begin{align}
\label{clean:bound:alpha:version}
\sup_{(\bmu,\bSigma)\in \cH(s,\tau)}\EE\ltwonorm{\tilde{\balpha}-\balpha}^2 \lesssim \frac{(1+\tau)s\log p}{n}.
\end{align}
On the other hand, if $s\lesssim p^{1-\tilde{\delta}}$ for some $\tilde{\delta}>0$, Theorem \ref{lower:bound} implies
\begin{equation}
\label{clean:lower:bound}
\inf_{\hat{\balpha}}\sup_{(\bmu,\bSigma)\in \cH(s,\tau)}\EE\ltwonorm{\hat{\balpha}-\balpha}^2 \gtrsim \tau \wedge \frac{(1+\tau)s\log p}{n}.
\end{equation}
Hence, as long as $\tau \gtrsim \frac{s\log p}{n}$, the estimator $\tilde{\balpha}$ is rate-optimal. Moreover, when $\tau \lesssim \frac{s\log p}{n}$, the trivial estimator $\bzero$ is optimal, because its maximum error
\[
\sup_{(\bmu,\bSigma)\in \cH(s,\tau)}\EE\ltwonorm{\bzero-\balpha}^2 \leq \sup_{(\bmu,\bSigma)\in \cH(s,\tau)}c_L^{-1}\balpha^T\bSigma\balpha=c_L^{-1}\sup_{(\bmu,\bSigma)\in \cH(s,\tau)}\bmu^T\bSigma^{-1}\bmu  \lesssim \tau,
\]
matches the lower bound in \eqref{clean:lower:bound}.

\item[(2)] \emph{Estimation of $\theta$.} Consider the same scaling $\frac{s\log p}{n}=o(1)$, and $p^{-\delta}\lesssim \frac{s\log p}{n}$ for some $\delta>0$. It is straightforward to confirm that choosing large enough $t$ and $c$ in Theorem \ref{thm:one} yields
\[
(\tau+\sqrt{\tau})p^{-\frac{(c_2t-2)\wedge (c-1)}{2}} \lesssim \frac{(1+\tau)s\log p}{n},
\]
thus
\begin{equation}
\label{clean:upper:bound}
\sup_{(\bmu,\bSigma)\in \cH(s,\tau)}\EE|\tilde{\theta}-\theta|\lesssim \frac{(1+\tau)s\log p}{n} + \frac{\tau+\sqrt{\tau}}{\sqrt{n}}.
\end{equation}
On the other hand, since $c_0 \in [0, 1]$ can be any constant  in Theorem \ref{lower:bound}(b), the second term in the lower bound is not negligible only when $s^2\lesssim p^{1-\tilde{\delta}}$ for some $\tilde{\delta}>0$. In such a case, we can choose sufficiently small $c_0 > 0$ so that $e^{2s^2p^{c_6 c_0-1}} = 1+o(1)$ and obtain
\begin{equation}
\label{clean:lower:bound2}
\inf_{\hat{\theta}}\sup_{(\bmu,\bSigma)\in \mathcal{H}(s,\tau)}\mathbb{E}|\hat{\theta}-\theta| \gtrsim \Big[\tau \wedge \frac{(1+\tau)s\log p}{n}\Big] + \Big[\tau \wedge  \frac{\tau+\sqrt{\tau}}{\sqrt{n}}\Big].
\end{equation}
It can be directly verified that the above lower bound will match the upper bound in \eqref{clean:upper:bound} when $\tau \gtrsim \frac{s\log p}{n}$. Hence $\tilde{\theta}$ is rate-optimal in the regime $\tau \gtrsim \frac{s\log p}{n}$.  Furthermore, in the other regime $\tau \lesssim \frac{s\log p}{n}$,  the lower bound in \eqref{clean:lower:bound2} is simplified to be of order $\tau$. So the trivial estimator $0$ attains the optimal rate since its error
\[
\sup_{(\bmu,\bSigma)\in \cH(s,\tau)}\EE|0-\theta| \leq  \tau,
\]
matches the lower bound.
\end{itemize}
We summarize the preceding discussions in the corollary below.
\begin{cor}
\label{minimax:optimal}
Consider the scaling $\frac{s\log p}{n}=o(1)$, $p^{-\delta}\lesssim \frac{s\log p}{n}$ for some $\delta >0$. Set $\lambda \asymp \sqrt{\frac{(1+\tau)\log p}{n}}, \gamma \asymp \sqrt{\tau}$ in \eqref{lasso:alpha}.
\begin{enumerate}
\item[(a)] Suppose $s \lesssim p^{1-\tilde{\delta}}$ for some $\tilde{\delta}>0$, then
\begin{align*}
 \inf_{\hat{\balpha}} \sup_{(\bmu,\bSigma)\in \mathcal{H}(s,\tau)}  \mathbb{E}\|\hat{\balpha}-\balpha\|_2^2  \asymp \tau \wedge \frac{(1+\tau)s\log p}{n} \asymp
\begin{cases}
\frac{(1+\tau)s\log p}{n} & {\rm ~if~} \tau \gtrsim \frac{s\log p}{n} \\
\tau & {\rm ~if~} \tau \lesssim \frac{s\log p}{n}
\end{cases}
\end{align*}
The estimator $\tilde{\balpha}$ is minimax rate-optimal in the regime $\tau \gtrsim \frac{s\log p}{n}$, and the trivial estimator $\bzero$ attains the optimal rate when $\tau \lesssim \frac{s\log p}{n}$.
\item[(b)] Suppose $s^2 \lesssim p^{1-\tilde{\delta}}$ for some $\tilde{\delta}>0$, then
\begin{align}
\label{refer:optimal:rate}
 \inf_{\hat{\theta}} \sup_{(\bmu,\bSigma)\in \mathcal{H}(s,\tau)} \EE|\hat{\theta}-\theta| \asymp &\Big[\tau \wedge  \frac{\tau+\sqrt{\tau}}{\sqrt{n}}\Big]+  \Big[\tau \wedge \frac{(1+\tau)s\log p}{n}\Big] \nonumber \\
\asymp &\begin{cases}
\frac{\tau+\sqrt{\tau}}{\sqrt{n}}+\frac{(1+\tau)s\log p}{n}& {\rm~if~} \tau \gtrsim \frac{s\log p}{n} \\
\tau &  {\rm ~if~} \tau \lesssim \frac{s\log p}{n}
\end{cases}
\end{align}
The estimator $\tilde{\theta}$ is minimax rate-optimal in the regime $\tau \gtrsim \frac{s\log p}{n}$, and the trivial estimator $0$ obtains the optimal rate when $\tau \lesssim \frac{s\log p}{n}$.
\end{enumerate}
\end{cor}

There are a few remarks we should make about the results in Corollary \ref{minimax:optimal}.

\noindent {\bf Remark 1.~} The scaling $\frac{s\log p}{n}=o(1)$ considered in Corollary \ref{minimax:optimal} is standard for high-dimensional sparse models. The condition $p^{-\delta}\lesssim \frac{s\log p}{n}$ is very mild. For instance, it holds when $p\geq n^{\epsilon}$ with $\epsilon>0$ being any positive constant.

\noindent {\bf Remark 2.~} Regarding the estimation for $\balpha$, since $\ltwonorm{\balpha}^2=O(\tau)$ in the parameter space $\cH(s,\tau)$, we may consider $\tau$ as the signal strength. Under the additional assumption on the sparsity $s\lesssim p^{1-\tilde{\delta}}$, the minimax estimation rate for $\balpha$ is $\tau \wedge \frac{(1+\tau)s\log p}{n}$. It is interesting to observe that the optimal rate depends on the signal strength  in a non-linear fashion. Moreover, in the regime $\tau \lesssim \frac{s\log p}{n}$ where the signal is sufficiently weak, a trivial estimator $\bzero$ achieves the optimal rate. Such a result can be well explained by the bias-variance tradeoff: for estimating very weak signals, the variance of an estimator plays the dominant role in the resulting error. Once $\tau$ is above the threshold $\frac{s\log p}{n}$, our estimator $\tilde{\balpha}$ becomes rate-optimal. Note that our analysis is focused  on the absolute error. The study of the relative error is an interesting problem and left for a future research.

\noindent {\bf Remark 3.~} Estimation of the functional $\balpha=\bSigma^{-1}\bmu$ has been studied in various contexts such as portfolio selection \citep{alz17}, time series \citep{cxw16}, and linear discriminant analysis \citep{cz18}. Different Lasso \citep{t96} or Dantzig-selector \citep{ct07} type estimators have been proposed and analyzed. However, our result characterizes the optimality of the proposed estimator $\tilde{\balpha}$ over a wide range of the signal strength $\tau$, which is not available in the existing works.

\noindent {\bf Remark 4.~} For the estimation of the functional $\theta=\bmu^T\bSigma^{-1}\bmu$, given the fact that $\theta \leq \tau$ in $\cH(s,\tau)$, we can regard $\tau$ as the signal strength when estimating $\theta$. Under the sparsity condition $s^2\lesssim p^{1-\tilde{\delta}}$, the minimax rate takes the form $\Big[\tau \wedge  \frac{\tau+\sqrt{\tau}}{\sqrt{n}}\Big]+  \Big[\tau \wedge \frac{(1+\tau)s\log p}{n}\Big]$. The signal strength $\tau$ appears in the rate in a rather complicated way. Consider the common case $\tau \asymp 1$, the rate is reduced to $\frac{1}{\sqrt{n}} \vee \frac{s\log p}{n}$. It is clear that the rate of convergence undergoes a transition between the parametric rate $\frac{1}{\sqrt{n}}$ and the high-dimensional rate $\frac{s\log p}{n}$. We refer to \cite{frw15, gwcl18} and reference therein for similar phenomenon in other high-dimensional problems. As in the estimation of $\balpha$, the trivial estimator $0$ attains the optimal rate when the signal is weak $\tau \lesssim \frac{s\log p}{n}$, while our proposed estimator $\tilde{\theta}$ is rate-optimal in the other regime $\tau \gtrsim \frac{s\log p}{n}$.

\noindent {\bf Remark 5.~} When the covariance matrix $\bSigma$ is known to be proportional to $\bI_p$, the minimax estimation of the functional $\theta$ has been thoroughly studied in \cite{cct17}. The authors derived non-asymptotic minimax rates for $\theta$ without any assumption on the sparsity $s$. The obtained rate shares some similarity with the rate we have derived in the present setting. Hence, our analysis might be considered as an extension of \cite{cct17} to unknown $\bSigma$.

\subsection{Sub-optimality of a class of plug-in estimators}
\label{suboptimality:plugin}

This section demonstrates the sub-optimality of the naive plug-in estimator. Consider more generally
\begin{equation}
\label{suboptimal:theta}
\tilde{\theta}_c=c\hat{\bmu}^T\tilde{\balpha}+(1-c)
\tilde{\balpha}^T\hat{\bSigma}\tilde{\balpha}
\end{equation}
for a given constant $c$.  The naive plug-in estimator mentioned in Section \ref{rate-optimal:sparse} corresponds to $c=1$ and our proposed debias-based estimator $\tilde{\theta}$ in \eqref{optimal:estimator} corresponds to $c=2$. The question is then whether the optimality of $\tilde{\theta}$ carries over to the other plug-in estimators $\tilde{\theta}_c$ with $c\neq 2$. The answer turns out to be negative. We give a formal statement in the next proposition.

\begin{prop}
\label{prop:two}
Consider the scaling $\frac{s\log p}{n}=o(1), p^{-\delta}\lesssim \frac{s\log p}{n}$ for some $\delta>0$. Choose $\lambda \asymp \sqrt{\frac{(1+\tau)\log p}{n}}, \gamma \asymp \sqrt{\tau}$ in \eqref{lasso:alpha}. If $\tau \gtrsim \frac{s\log p}{n}$, then for any given constant $c\neq 2$,
\begin{align*}
\sup_{(\bmu,\bSigma)\in \cH(s,\tau)}\EE|\tilde{\theta}_c-\theta| \gtrsim \sqrt{\frac{\tau(1+\tau)s\log p}{n}}.
\end{align*}
\end{prop}

When $\tau \gg \frac{s\log p}{n}$, it is straightforward to verify that
\[
\sqrt{\frac{\tau(1+\tau)s\log p}{n}} \gg \frac{(1+\tau)s\log p}{n}+\frac{\tau+\sqrt{\tau}}{\sqrt{n}}.
\]
According to Corollary \ref{minimax:optimal}, the lower bound in the above inequality is precisely the minimax rate when $\tau \gtrsim \frac{s\log p}{n}$. Therefore, Proposition~\ref{prop:two} shows that in the regime $\tau \gg \frac{s\log p}{n}$, the plug-in estimator $\tilde{\theta}_c$ can not achieve the optimal rate for any constant $c\neq 2$.

The sub-optimality of $\tilde{\theta}_c~(c\neq 2)$ is essentially due to the bias of $\tilde{\balpha}$ induced by the $\ell_1$ regularization. Specifically, since $\tilde{\balpha}$ is the global solution of the convex optimization problem \eqref{lasso:alpha}, the first-order optimality condition shows that
\begin{align}
\label{first:order}
\langle \hat{\bSigma}\tilde{\balpha}-\hat{\bmu}+\lambda \hat{\bg}, \bbeta-\tilde{\balpha} \rangle \geq 0,
\end{align}
where $\hat{\bg}$ is a subgradient of $\lonenorm{\bbeta}$ at $\bbeta=\tilde{\balpha}$, and $\bbeta$ can be any vector with $\ltwonorm{\bbeta}\leq \gamma$. Setting $\bbeta=\frac{1}{2}\tilde{\balpha}$ in \eqref{first:order} gives
\[
\hat{\bmu}^T\tilde{\balpha}-\tilde{\balpha}^T\hat{\bSigma}\tilde{\balpha} \geq \lambda \lonenorm{\balpha},
\]
thus yielding
\[
|\tilde{\theta}_c-\tilde{\theta}|=|c-2|\cdot |\hat{\bmu}^T\tilde{\balpha}-\tilde{\balpha}^T\hat{\bSigma}\tilde{\balpha} | \geq |c-2|\cdot \lambda\lonenorm{\tilde{\balpha}}.
\]
The regularization term $\lambda \lonenorm{\tilde{\balpha}}$ is so large that the gap between $\tilde{\theta}_c~(c\neq 2)$ and the optimal estimator $\tilde{\theta}$  exceeds the optimal rate. We refer to the proof in Section \ref{proof:corollary2} for detailed calculations.

The preceding arguments suggest a remedy to address the sub-optimality of  $\tilde{\theta}_c$ for $c\neq 2$. If the bias arising from $\ell_1$ regularization can be attenuated, the resulting plug-in estimator might be possibly improved. Indeed, once the ``biased" estimator $\tilde{\balpha}$ in \eqref{suboptimal:theta} is replaced by the less-biased $\ell_0$-estimator
\begin{equation}
\label{alpha:lzero}
\check{\balpha} \in \argmin_{\lzeronorm{\bbeta}\leq s, \ltwonorm{\bbeta}\leq \gamma} \frac{1}{2}\bbeta^T\hat{\bSigma}\bbeta-\bbeta^T\hat{\bmu},
\end{equation}
the optimal estimation rate will be ultimately obtained. We show it formally in Theorem \ref{corollary:three}. Define the new class of plug-in estimators as
\[
\check{\theta}_c=c\hat{\bmu}^T\check{\balpha}+(1-c)\check{\balpha}^T\hat{\bSigma}\check{\balpha}, \quad \forall c \in \RR.
\]

\begin{thm}
\label{corollary:three}
Consider the scaling $\frac{s\log p}{n}=o(1), p^{-\delta}\lesssim \frac{s\log p}{n}$ for some $\delta>0$. Set $\gamma \asymp \sqrt{\tau}$ in \eqref{alpha:lzero}. In the regime $\tau \gtrsim \frac{s\log p}{n}$,
\begin{align*}
&\sup_{(\bmu,\bSigma)\in \cH(s,\tau)} \EE\ltwonorm{\check{\balpha}-\balpha}^2 \lesssim \frac{(1+\tau)s\log p}{n}, \\
&\sup_{(\bmu,\bSigma)\in \cH(s,\tau)} \EE|\check{\theta}_c-\theta| \lesssim \frac{\tau+\sqrt{\tau}}{\sqrt{n}}+\frac{(1+\tau)s\log p}{n}, \quad \forall c \in \mathbb{R}.
\end{align*}
\end{thm}


We emphasize a few points  about Proposition \ref{prop:two} and Theorem \ref{corollary:three}.

\noindent {\bf Remark 6.~}
The upper bound for $\check{\balpha}$ in Theorem \ref{corollary:three} matches the minimax rate in Corollary \ref{minimax:optimal} when $\tau \gtrsim \frac{s\log p}{n}$. As expected, the $\ell_0$-estimator $\check{\balpha}$ attains the optimal rate as the $\ell_1$-regularized M-estimator $\tilde{\balpha}$ does in the same regime. However, despite $\tilde{\balpha}$ is rate-optimal for $\balpha$, it does not guarantee that the projection of $\tilde{\balpha}$ along every direction is optimal for estimating the same projection of $\balpha$. This is the main reason why the plug-in estimator $\tilde{\theta}_c ~(c\neq 2)$ using $\tilde{\balpha}$ turns out to be sub-optimal.

\noindent {\bf Remark 7.~}
Combined with Corollary \ref{minimax:optimal}, Theorem \ref{corollary:three} reveals that the new plug-in estimator $\check{\theta}_c$ is rate-optimal in the same regime as the de-biased estimator $\tilde{\theta}$ given in \eqref{optimal:estimator}, and this optimality holds for every choice of $c\in \RR$. On the other hand, we should emphasize that $\check{\theta}_c$ is computationally infeasible, while $\tilde{\theta}$ can be computed in polynomial time. Hence, our proposed estimator $\tilde{\theta}$ achieves the best of both worlds in terms of computational feasibility and statistical efficiency.

\subsection{A general result for the de-biased estimator}
\label{general:estimator:sec}
In this section, we discuss some generalization of our method to estimate the functional that takes the form
\[
\quad \vartheta=\bxi^T\bUpsilon^{-1}\bxi,
\]
where $\bxi \in \RR^p, \bUpsilon \in \RR^{p\times p}$ are unknown parameters. When $\bxi=\bmu, \bUpsilon=\bSigma$, $\vartheta$ equals to the functional $\theta$ that we have studied in the previous two sections. It is direct to observe that the de-biased estimator $\tilde{\theta}$ for $\theta$ proposed in \eqref{optimal:estimator} is simply a function of the sample mean $\hat{\bmu}$ and sample covariance matrix $\hat{\bSigma}$. More broadly, we may estimate $\vartheta$ in the same way by replacing $\hat{\bmu}$ and $\hat{\bSigma}$ with some other good estimators for $\bxi$ and $\bUpsilon$ respectively. Thus motivated, our generalization is to take two estimators $\hat{\bxi} \in \RR^p$ and $\hat{\bUpsilon} \succeq 0$ as inputs, and construct the estimator $\hat{\vartheta}$ for $\vartheta$ in the following way,
\begin{align}
\hat{\bvarpi} &\in \argmin_{\ltwonorm{\bbeta} \leq \gamma} \frac{1}{2}\bbeta^T \hat{\bUpsilon}\bbeta-\bbeta^T \hat{\bxi}+\lambda \lonenorm{\bbeta}, \label{new:alpha:estimator} \\
 \hat{\vartheta}&=2\hat{\bxi}^T\hat{\bvarpi}-\hat{\bvarpi}^T\hat{\bUpsilon}\hat{\bvarpi}. \label{new:theta:estimator}
\end{align}
This framework includes estimating $\theta = \bmu^T \bSigma^{-1} \bmu$ using different inputs of estimators for $\bmu$ and $\bSigma$ such as robustified estimators \citep{fwz16,fwzz18,kmrsz19}.  It also includes the two-sample problem as to be elaborated below.

The estimation risk of $\hat{\vartheta}$ will critically depend on the approximation accuracy of the inputs $\hat{\bxi}$ and $\hat{\bUpsilon}$. We give a general upper bound for $\hat{\vartheta}$ in the next proposition. As a by-product, we include an upper bound for $\hat{\bvarpi}$ as an estimator of $\bvarpi=\bUpsilon^{-1}\bxi$. Towards that goal, define
\begin{align*}
&\cK(s)=\Big \{\bu  \in \mathbb{R}^p: \lonenorm{\bu_{S^c}} \leq 3 \lonenorm{\bu_{S}}, |S| \leq s, S\subseteq [p] \Big \},  \\
&\cA(\lambda)=\Big \{  \linfnorm{\hat{\bUpsilon}\bvarpi-\hat{\bxi}} \leq \lambda/2  \Big \}, \\
&\cB(s, \kappa)=\Big\{\max_{\bu\in \mathcal{K}(s)\cap B_2(1) } \big| \sqrt{\bu^T\bUpsilon\bu}-\sqrt{\bu^T\hat{\bUpsilon}\bu}\big | \leq \kappa \Big\}.
\end{align*}
We drop the sub-gaussian assumption on $\bx$.

\begin{prop}
\label{general:debiasing}
 Set $\gamma \asymp  \sqrt{\tau}$ in \eqref{new:alpha:estimator}. If there exist some constants $c_1,c_2>0$ such that
\begin{align}
\label{concentration:mean:vector}
\mathbb{P}(\|\hat{\bxi}-\bxi\|_{\infty}>t)\leq c_1pe^{-c_2m t^2}, \quad \quad \forall t>0,
\end{align}
then the followings hold
\begin{align*}
&\sup_{(\bxi,\bUpsilon)\in \cH(s,\tau)} \EE\ltwonorm{\hat{\bvarpi}-\bvarpi}^2 \lesssim  \lambda^2 s + \tau\big[\PP(\cA^c(\lambda))+\PP(\cB^c(s, \sqrt{c_L}/2))\big], \\
&\sup_{(\bxi,\bUpsilon)\in \cH(s,\tau)}  \EE| \hat{\vartheta}-\vartheta| \lesssim \lambda s(\lambda+\sqrt{\log p/m}) + (\sqrt{\tau}+\tau)p\sqrt{\PP(\cA^c(\lambda))+\PP(\cB^c(s, \sqrt{c_L}/2))}\\
&\hspace{3.cm}  +\sup_{(\bxi,\bUpsilon)\in \cH(s,\tau)}\mathbb{E}|\bvarpi^T(\hat{\bUpsilon}-\bUpsilon)\bvarpi|+\sup_{(\bxi,\bUpsilon)\in \cH(s,\tau)} \sqrt{\EE|\bvarpi^T(\hat{\bxi}-\bxi)|^2}.
\end{align*}
\end{prop}


Proposition \ref{general:debiasing} has a few implications we should discuss.
\begin{itemize}
\item[(1)] \emph{Upper bound for two-sample problems.} Recall the high-dimensional LDA problem discussed in Section \ref{hlda}.  Since two sets of samples are present, the interested functional is $\vartheta=(\bmu_1-\bmu_2)^T\bSigma^{-1}(\bmu_1-\bmu_2)$. We have shown in  Section \ref{hlda} that the minimax error for $\vartheta$ can be lower bounded by the minimax error for $\theta$ in the one-sample problem. We now use Proposition \ref{general:debiasing} to derive a matching upper bound. Consider the estimator $\hat{\vartheta}$  in \eqref{new:theta:estimator} by setting
\begin{align*}
&\hat{\bxi}=\hat{\bmu}_1-\hat{\bmu}_2, \quad \hat{\bmu}_1=\frac{1}{n_1}\sum_{i=1}^{n_1}\bx_i,  \quad \hat{\bmu}_2=\frac{1}{n_2}\sum_{i=1}^{n_2}\by_i, \\
&\hat{\bUpsilon}=\frac{1}{n_1+n_2}\Big[\sum_{i=1}^{n_1}(\bx_i-\hat{\bmu}_1)(\bx_i-\hat{\bmu}_1)^T+\sum_{i=1}^{n_2}(\by_i-\hat{\bmu}_2)(\by_i-\hat{\bmu}_2)^T\Big].
\end{align*}
Suppose the two samples $\{\bx_i\}_{i=1}^{n_1}$ and $\{\by_i\}_{i=1}^{n_2}$ follow sub-gaussian distributions with $n_1 \asymp n_2$. Choose $\lambda \asymp \sqrt{\frac{(1+\tau)\log p}{n_1+n_2}}$ in \eqref{new:alpha:estimator}. Under the scaling $p^{-\delta}\lesssim \frac{s\log p}{n_1+n_2}$ for some $\delta>0$, a minor modification of the arguments in the proof of Theorem \ref{thm:one} enables us to simplify the upper bound in Proposition \ref{general:debiasing} to obtain
\begin{align*}
\sup_{(\bxi,\bUpsilon)\in \cH(s,\tau)}  \EE| \hat{\vartheta}-\vartheta| \lesssim \frac{\tau+\sqrt{\tau}}{n_1+n_2}+\frac{(1+\tau)s\log p}{n_1+n_2}.
\end{align*}
In light of the discussion preceding Corollary \ref{minimax:optimal}, we thus can conclude that the minimax rate in the one-sample problem continues to hold in the two-sample case.
\item[(2)] \emph{Robust estimation of the functionals.} The main results regarding the functional $\theta=\bmu^T\bSigma^{-1}\bmu$ in Sections \ref{rate-optimal:sparse} and \ref{suboptimality:plugin} are established for sub-gaussian distributions. When the data possesses heavier tails, it might be necessary to substitute the sample mean and sample covariance matrix used in $\tilde{\balpha}$ and $\tilde{\theta}$ by some robustified versions, in order to achieve a better bias-variance tradeoff. Motivated by recent advances in nonasymptotic deviation analyses of tail-robust estimators for the mean vector and covariance matrix (see \cite{fwz16}, \cite{kmrsz19} and references therein), to estimate $\balpha$ and $\theta$ under heavier-tailed distributions, we consider the estimators $\hat{\bvarpi}$ and  $\hat{\vartheta}$ in \eqref{new:alpha:estimator} and \eqref{new:theta:estimator}, with element-wise truncated mean and covariance matrix estimators defined as follows:
\begin{align*}
&\hat{\bxi}=(\hat{\xi}_1,\ldots, \hat{\xi}_p), ~~\hat{\xi}_j=\frac{1}{n}\sum_{i=1}^n\varphi_{\tau_j}(x_{ij}),~~ j=1,\ldots, p. \\
&\hat{\bUpsilon}=(\hat{\Upsilon}_{k\ell})_{1\leq k,\ell\leq p},~~\hat{\Upsilon}_{k\ell}=\frac{1}{N}\sum_{i=1}^N\varphi_{\tau_{k\ell}}(y_{ik}y_{i\ell}/2), ~~1\leq k, \ell\leq p.
\end{align*}
Here, $N=n(n-1)/2$; the $\by_i$'s are the paired data formed by $\bx_i$'s: $\{\by_1,\by_2,\ldots, \by_N\}=\{\bx_1-\bx_2,\bx_1-\bx_3,\ldots, \bx_{n-1}-\bx_n\}$; and $\varphi_{\tau}(u)=(|u| \wedge \tau)\cdot \mbox{sign}(u)$.  According to Theorem 3.1 in \cite{kmrsz19}, if $\max_{k\ell}\mathbb{E}(y^2_{1k}y^2_{1\ell})\leq c_1$, then for any $0<\delta <1$, under appropriate choice of the thresholds $\{\tau_{k\ell}\}_{1\leq k,\ell\leq p}$ it holds that
\begin{align}
\label{infinity:bound:matrix}
\PP\Big(\|\hat{\bUpsilon}-\bUpsilon\|_{\max}\geq c_2\sqrt{\frac{\log p+\log \delta^{-1}}{n}}\Big)\leq \delta.
\end{align}
The deviation analyses in \cite{kmrsz19} lead to a concentration result for $\hat{\bxi}$ as well. Specifically, if $\max_{j}\EE(x^2_{1j})\leq c_3$, then for any $t> 0$, setting $\tau_{j}=2\EE(x^2_{1j})/t$ gives that 
\begin{align}
\label{infinity:bound:vector}
\PP(\linfnorm{\hat{\bxi}-\bxi}>t)\leq c_4pe^{-c_5nt^2}.
\end{align}
Note that the condition \eqref{concentration:mean:vector} in Proposition \ref{general:debiasing} is stronger than \eqref{infinity:bound:vector}, because the truncated estimator $\hat{\bxi}$ in \eqref{infinity:bound:vector} depends on $t$. We can further use the confidence interval method in \cite{dllo16} to turn the $t$-dependent estimators into an estimator $\hat{\bxi}$ invariant of $t$ such that
\begin{align}
\PP(\linfnorm{\hat{\bxi}-\bxi}>t)\leq c_6pe^{-c_7nt^2}, \quad \forall t \geq 0.
\end{align}
In the above results, all the $c_i$'s are positive constants. Therefore, under bounded fourth moments conditions, we can apply Proposition \ref{general:debiasing} to derive upper bounds for $\hat{\bvarpi}$ and  $\hat{\vartheta}$. In particular, set $\lambda \asymp \sqrt{\frac{(1+\tau)s\log p}{n}}$ in \eqref{new:alpha:estimator}, $\log \delta^{-1} \asymp \log p$ in \eqref{infinity:bound:matrix}, and assume the scaling $\frac{s^2\log p}{n}=o(1)$ and $p^{-\delta}\lesssim \frac{s^2\log p}{n}$ for some $\delta>0$. It is straightforward to simplify the upper bounds in Proposition \ref{general:debiasing} to obtain
\begin{align}
&\sup_{(\bxi,\bUpsilon)\in \cH(s,\tau)} \EE\ltwonorm{\hat{\bvarpi}-\bvarpi}^2 \lesssim \frac{(1+\tau)s^2\log p}{n}, \label{robust:bound:alpha}\\
&\sup_{(\bxi,\bUpsilon)\in \cH(s,\tau)}  \EE| \hat{\vartheta}-\theta| \lesssim \frac{(1+\tau)s^2\log p}{n}+\sup_{(\bxi,\bUpsilon)\in \cH(s,\tau)}\mathbb{E}|\bvarpi^T(\hat{\bUpsilon}-\bUpsilon)\bvarpi| \nonumber \\
&\hspace{3.cm}+\sup_{(\bxi,\bUpsilon)\in \cH(s,\tau)} \sqrt{\EE|\bvarpi^T(\hat{\bxi}-\bxi)|^2}. \label{robust:bound:theta}
\end{align}
\end{itemize}
Compared to the one in the sub-gaussian case, the bound for $\hat{\bvarpi}$ in \eqref{robust:bound:alpha} has an additional multiplicative factor $s$. A similar rate was derived in the problem of high-dimensional Huber regression with heavy-tailed designs \citep{szf18}. Whether such a rate is minimax optimal for heavy-tailed distributions seems unknown. Regarding the estimator $\hat{\vartheta}$, the upper bound in \eqref{robust:bound:theta} has one term identical to the rate for $\hat{\bvarpi}$, as in the sub-gaussian scenario. The other two terms are more subtle. When $\hat{\bxi}=\hat{\bmu}, \hat{\bUpsilon}=\hat{\bSigma}$, it is direct to compute the sum of these two terms to be of the order $\frac{\sqrt{\tau}+\tau}{\sqrt{n}}$. Hence both terms contribute to the parametric rate part in the minimax rate of $\theta$ (cf. Corollary \ref{minimax:optimal}). However, in the present heavier-tailed setting, the estimators $\hat{\bxi}$ and $\hat{\bUpsilon}$ are constructed by truncation operations to trade bias for robustness. It becomes difficult to characterize the accurate dependence of the two terms on $n$ and $\tau$. More fundamentally, what is the minimax rate for estimating the functional $\theta$ under heavy-tailed distributions? Does the rate undergo a transition between parametric rate and high-dimensional rate as we revealed in the sub-gaussian situation? We leave a thorough minimax analysis of the functional estimation under heavy-tailed distributions for a future research.

\subsection{Generalization to approximate sparsity classes}
\label{approximate:sparse:class}
In Section \ref{rate-optimal:sparse}, we have derived the minimax rate for the functional $\theta=\bmu^T\bSigma^{-1}\bmu$ 
when $\balpha=\bSigma^{-1} \bmu$ is sparse. 
We investigate the performance of $\tilde{\theta}$ over the approximate sparsity classes
\begin{equation*}
\cH_q(R, \tau)=\Big\{(\bmu,\bSigma): \lqnorm{\balpha}^q \leq R, \theta \leq \tau, c_L\leq \lambda_{\min}(\bSigma)\leq \delta_{\bSigma}\leq c_U\Big\}, \quad q\in (0,1].
\end{equation*}
Observe that in $\cH_q(R,\tau)$, in addition to the $\ell_q$ ball constraint $\lqnorm{\balpha}^q\leq R$, the vector $\balpha$ has to satisfy the quadratic inequality $\theta=\balpha^T\bSigma\balpha\leq \tau$. The equality may not hold simultaneously in the preceding two constraints. In fact, if we define
\[
\tilde{\tau}=(c_UR^{\frac{2}{q}})\wedge \tau, \quad \tilde{R}=(p^{1-\frac{q}{2}}c_L^{-\frac{q}{2}}\tau^{\frac{q}{2}})\wedge R,
\]
Lemma \ref{signal:strength} from Section \ref{proof:of:thm3} shows that $\cH_q(R,\tau)=\cH_q(\tilde{R},\tilde{\tau})$. Therefore, the effective scaling parameters under $\cH_q(R,\tau)$ are $(\tilde{R}, \tilde{\tau})$ rather than $(R, \tau)$. We should expect $(\tilde{R},\tilde{\tau})$ to play the role in the minimax results. We derive the upper and lower bounds for $\tilde{\balpha}$ and $\tilde{\theta}$ in the next two theorems.


\begin{thm}
\label{approximate:sparse:thm1}
Consider the scaling $\tilde{R}(1+\tilde{\tau})^{-\frac{q}{2}}(\frac{\log p}{n})^{1-\frac{q}{2}}=o(1)$, and $p^{-\delta} \lesssim \tilde{R}(\frac{\log p}{n})^{1-\frac{q}{2}}(\tilde{\tau}^{-\frac{q}{2}}\vee \tilde{\tau}^{-1})$ for some $\delta >0$. Set $\lambda \asymp \sqrt{\frac{(1+ \tilde{\tau})\log p}{n}}, \gamma \asymp \sqrt{\tilde{\tau}}$ in \eqref{lasso:alpha}. Then, it holds the followings:
\begin{align*}
&\sup_{(\bmu,\bSigma)\in \cH_q(R,\tau)}\EE\ltwonorm{\tilde{\balpha}-\balpha}^2 \lesssim (1+\tilde{\tau})^{1-\frac{q}{2}} \tilde{R} \bigg(\frac{\log p}{n}\bigg)^{1-\frac{q}{2}}, \\
&\sup_{(\bmu,\bSigma)\in \cH_q(R,\tau)} \EE |\tilde{\theta}-\theta| \lesssim  (1+ \tilde{\tau})^{1-\frac{q}{2}} \tilde{R} \bigg(\frac{\log p}{n}\bigg)^{1-\frac{q}{2}}+\frac{\tilde{\tau}+\sqrt{\tilde{\tau}}}{\sqrt{n}}.
\end{align*}
\end{thm}


\begin{thm}
\label{approxiamte:sparse:thm2}
Consider the scaling $\tilde{R}(1+\tilde{\tau})^{-\frac{q}{2}}(\frac{\log p}{n})^{1-\frac{q}{2}}=o(1)$.
\begin{itemize}
\item[(a)] If $1\lesssim \tilde{R}(1+\tilde{\tau})^{-\frac{q}{2}}(\frac{\log p}{n})^{-\frac{q}{2}} \lesssim p^{1-\delta}$ for some $\delta >0$, then
\[
\inf_{\hat{\balpha}}\sup_{(\bmu,\bSigma) \in \cH_q(R,\tau)} \mathbb{E}\ltwonorm{\hat{\balpha}-\balpha}^2 \gtrsim \tilde{\tau} \wedge\Big [(1+\tilde{\tau})^{1-\frac{q}{2}}\tilde{R}\Big(\frac{\log p}{n}\Big)^{1-\frac{q}{2}}\Big].
\]
\item[(b)] If $1\lesssim \tilde{R}^2(1+\tilde{\tau})^{-q}(\frac{\log p}{n})^{-q} \lesssim p^{1-\delta}$ for some $\delta>0$, then
\begin{align*}
\inf_{\hat{\theta}}\sup_{(\bmu,\bSigma) \in \cH_q(R,\tau)} \mathbb{E}|\hat{\theta}-\theta| \gtrsim \tilde{\tau} \wedge\Big[ (1+ \tilde{\tau})^{1-\frac{q}{2}} \tilde{R} \bigg(\frac{\log p}{n}\bigg)^{1-\frac{q}{2}} \Big]+ \Big[\tilde{\tau} \wedge \frac{\tilde{\tau}+\sqrt{\tilde{\tau}}}{\sqrt{n}}\Big].
\end{align*}
\end{itemize}
\end{thm}

Theorems \ref{approximate:sparse:thm1} and \ref{approxiamte:sparse:thm2} can be seen as a generalization of Theorems \ref{thm:one} and \ref{lower:bound}, respectively. The quantity $\tilde{R}(1+\tilde{\tau})^{-\frac{q}{2}}(\frac{\log p}{n})^{-\frac{q}{2}}$ plays the same role as the sparsity level $s$ in the exactly sparse case. Setting $q=0$ in the two theorems, we fully recover the upper bounds \eqref{clean:bound:alpha:version} and \eqref{clean:upper:bound} and the lower bounds \eqref{clean:lower:bound} and \eqref{clean:lower:bound2}, for the exact sparsity classes under the same scaling conditions. Moreover, observe that
\begin{align*}
\sup_{(\bmu,\bSigma)\in \cH_q(R,\tau)}\EE\ltwonorm{\bzero-\balpha}^2&=\sup_{(\bmu,\bSigma)\in \cH_q(\tilde{R},\tilde{\tau})}\EE\ltwonorm{\bzero-\balpha}^2 \lesssim  \tilde{\tau} \\
\sup_{(\bmu,\bSigma)\in \cH_q(R,\tau)}\EE| 0-\theta|&=\sup_{(\bmu,\bSigma)\in \cH_q(\tilde{R},\tilde{\tau})}\EE|0-\theta| \lesssim \tilde{\tau}.
\end{align*}
The above, combined with Theorems \ref{approximate:sparse:thm1} and \ref{approxiamte:sparse:thm2}, gives us the generalization of Corollary \ref{minimax:optimal}.

\begin{cor}
Consider the scaling $\tilde{R}(1+\tilde{\tau})^{-\frac{q}{2}}(\frac{\log p}{n})^{1-\frac{q}{2}}=o(1)$, and $p^{-\delta} \lesssim \tilde{R}(\frac{\log p}{n})^{1-\frac{q}{2}}(\tilde{\tau}^{-\frac{q}{2}}\vee \tilde{\tau}^{-1})$ for some $\delta >0$. Set $\lambda \asymp \sqrt{\frac{(1+ \tilde{\tau})\log p}{n}}, \gamma \asymp \sqrt{\tilde{\tau}}$ in \eqref{lasso:alpha}.
\begin{enumerate}
\item[(a)] Suppose $1\lesssim \tilde{R}(1+\tilde{\tau})^{-\frac{q}{2}}(\frac{\log p}{n})^{-\frac{q}{2}} \lesssim p^{1-\delta}$ for some $\delta >0$, then
\begin{align*}
 \inf_{\hat{\balpha}} \sup_{(\bmu,\bSigma)\in \mathcal{H}_q(R,\tau)}  \mathbb{E}\|\hat{\balpha}-\balpha\|_2^2  \asymp  \tilde{\tau} \wedge\Big [(1+\tilde{\tau})^{1-\frac{q}{2}}\tilde{R}\Big(\frac{\log p}{n}\Big)^{1-\frac{q}{2}}\Big].
\end{align*}
The estimator $\tilde{\balpha}$ is minimax rate-optimal in the regime $\tilde{\tau} \gtrsim \tilde{R}(1+\tilde{\tau})^{-\frac{q}{2}}(\frac{\log p}{n})^{1-\frac{q}{2}}$, and the trivial estimator $\bzero$ attains the optimal rate when $\tilde{\tau} \lesssim \tilde{R}(1+\tilde{\tau})^{-\frac{q}{2}}(\frac{\log p}{n})^{1-\frac{q}{2}}$.
\item[(b)] Suppose $1\lesssim \tilde{R}^2(1+\tilde{\tau})^{-q}(\frac{\log p}{n})^{-q} \lesssim p^{1-\delta}$ for some $\delta>0$, then
\begin{align*}
 \inf_{\hat{\theta}} \sup_{(\bmu,\bSigma)\in \mathcal{H}_q(R,\tau)} \EE|\hat{\theta}-\theta| \asymp \tilde{\tau} \wedge\Big[ (1+ \tilde{\tau})^{1-\frac{q}{2}} \tilde{R} \bigg(\frac{\log p}{n}\bigg)^{1-\frac{q}{2}} \Big]+ \Big[\tilde{\tau} \wedge \frac{\tilde{\tau}+\sqrt{\tilde{\tau}}}{\sqrt{n}}\Big].
\end{align*}
The estimator $\tilde{\theta}$ is minimax rate-optimal in the regime $\tilde{\tau} \gtrsim \tilde{R}(1+\tilde{\tau})^{-\frac{q}{2}}(\frac{\log p}{n})^{1-\frac{q}{2}}$, and the trivial estimator $0$ obtains the optimal rate when $\tilde{\tau} \lesssim \tilde{R}(1+\tilde{\tau})^{-\frac{q}{2}}(\frac{\log p}{n})^{1-\frac{q}{2}}$.
\end{enumerate}
\end{cor}

\subsection{Functional estimation in the dense regime}
\label{dense:regime}

As shown in Sections \ref{rate-optimal:sparse} and \ref{approximate:sparse:class}, the minimax optimality of our proposed functional estimator $\tilde{\theta}$ in \eqref{optimal:estimator} relies on the condition $s^2\lesssim p^{1-\tilde{\delta}}$ for exact sparsity classes and $\tilde{R}^2(1+\tilde{\tau})^{-q}(\frac{\log p}{n})^{-q} \lesssim p^{1-\delta}$ for approximate sparsity ones. This sparsity condition is required in a chi-squared distance calculation on Gaussian mixtures to obtain the matching minimax lower bound. Similar conditions and chi-squared distance calculations have appeared in the minimax analysis of other high-dimensional sparse problems \citep{cl04, cz12, frw15}. Such condition can be rather mild and hold for many high-dimensional data applications when the signal is reasonably sparse. On the other hand, it is interesting to investigate the complementary dense regime. This section presents some results and discussions along this line to shed more lights on the functional estimation problem. To simplify the presentation, we will focus on the estimation over exact sparsity classes. A straightforward modification of Theorem \ref{lower:bound}(b) yields the following minimax lower bound in the dense regime where $s \gtrsim \sqrt{p}$. 

\begin{prop}
\label{lowerbound:dense}
In the regime $s \gtrsim \sqrt{p}$, it holds that
\begin{align*}
\inf_{\hat{\theta}}\sup_{(\bmu,\bSigma)\in \mathcal{H}(s,\tau)}\mathbb{E}|\hat{\theta}-\theta| \gtrsim  \Big[\tau \wedge \frac{(1+\tau)\sqrt{p}}{n}\Big]+\Big[\tau \wedge  \frac{\tau+\sqrt{\tau}}{\sqrt{n}}\Big].
\end{align*}
\end{prop}

While the lower bound in Proposition \ref{lowerbound:dense} does not depend on the sparsity level $s$, we will propose estimators to obtain the matching upper bound under some scaling conditions. We first consider the Gaussian case.

\begin{prop}
\label{gaussian:case}
Suppose $\bx_1, \ldots, \bx_n\overset{i.i.d.}{\sim} N(\bmu,\bSigma)$ and $\limsup_{n\rightarrow \infty}\frac{p}{n}<1$. Consider the functional estimator $\hat{\theta}_{g}=\frac{n-p-2}{n}\hat{\bmu}^T\hat{\bSigma}^{-1}\hat{\bmu}-\frac{p}{n}$. Then,
\begin{align*}
\sup_{(\bmu,\bSigma)\in \cH(s,\tau)}\EE|\hat{\theta}_g-\theta|\lesssim \frac{\sqrt{p}}{n} + \frac{\tau+\sqrt{\tau}}{\sqrt{n}}.
\end{align*}
\end{prop}

On a related note, we refer to \cite{ka10} for a comprehensive study of the high-dimensionality effects in the Markowitz problem when $s=p\asymp n$. It is clear from the proof of Proposition \ref{gaussian:case} that the estimator $\hat{\theta}_g$ is unbiased for $\theta$. Hence $\hat{\theta}_g$ is a de-biased version of the naive plug-in estimator $\hat{\bmu}^T\hat{\bSigma}^{-1}\hat{\bmu}$. Before discussing the optimality of $\hat{\theta}_g$, we generalize the upper bound to sub-Gaussian distributions. 

We split the data into $m+1$ roughly equal-sized parts with sizes $n-m\lfloor \frac{n}{m+1}\rfloor, \lfloor \frac{n}{m+1}\rfloor, \ldots, \lfloor \frac{n}{m+1}\rfloor$. Let $\hat{\bmu}_{(0)}$ and $\hat{\bSigma}_{(0)}$ be the sample mean and sample covariance matrix based on the first part, respectively. Denote the unbiased sample covariance matrix from the $(j+1)$th part by $\hat{\bSigma}_{(j)}$, for $j=1,2,\ldots, m$. Consider the following functional estimator:
\begin{align}
\label{estimator:sg:new}
\hat{\theta}_{sg}=\sum_{k=0}^m\Big(\hat{\bmu}_{(0)}^T\tilde{\bSigma}^{-1/2}_{(0)}\cdot \prod_{j=1}^k(\bI_p-\tilde{\bSigma}^{-1/2}_{(0)}\hat{\bSigma}_{(j)}\tilde{\bSigma}_{(0)}^{-1/2})\cdot \tilde{\bSigma}^{-1/2}_{(0)}\hat{\bmu}_{(0)}\Big)-\frac{p(m+1)}{n},
\end{align}
where $\tilde{\bSigma}_{(0)}=\hat{\bSigma}_{(0)}+\epsilon_n \bI_p$ with some constant $\epsilon_n>0$. We add the small regularization term $\epsilon_n\bI_p$ to the sample covariance matrix $\hat{\bSigma}_{(0)}$ to ensure valid matrix inverse operation. The data splitting employed in $\hat{\theta}_{sg}$ is amenable to the bias and variance analysis.

\begin{thm}
\label{subgaussian:case}
Consider the scaling $p=O(n^{\alpha})$ for some $\alpha<1$. Choose $m=\lceil \frac{\alpha}{1-\alpha}\rceil, \epsilon_n=\sqrt{\frac{p}{n}}$ in $\hat{\theta}_{sg}$. For sub-Gaussian distributions as described in \eqref{subgaussian:def}, it holds that
\begin{align*}
\sup_{(\bmu,\bSigma)\in \cH(s,\tau)}\EE|\hat{\theta}_{sg}-\theta|\lesssim \frac{\sqrt{p}}{n} + \frac{\tau+\sqrt{\tau}}{\sqrt{n}}.
\end{align*}
\end{thm}

Recall that in the Gaussian case, correcting the bias of the naive plug-in estimator $\hat{\bmu}^T\hat{\bSigma}^{-1}\hat{\bmu}$ enables us to obtain the upper bound in Proposition \ref{gaussian:case}. Such bias correction is possible largely due to the two results that $\hat{\bSigma}^{-1}$ is inverse Wishart distributed (up to a scaling) and it is independent from $\hat{\bmu}$. However, for general sub-Gaussian distributions, neither of them is correct. We therefore have to correct the bias in a more ``nonparametric" way. Without delving into details (which will be made clear in the proof), the terms in the summation of $\eqref{estimator:sg:new}$ can be understood as correcting different orders of bias so as to make remaining bias after correction negligible compared to the upper bound $\frac{\sqrt{p}}{n} + \frac{\tau+\sqrt{\tau}}{\sqrt{n}}$. The number of terms needed depends on the dimension $p$ in a monotonically increasing way as specified in Theorem \ref{subgaussian:case}. Our bias correction has some similarity with the idea of the iterated bootstrap that is developed in a series of papers \citep{k19, kz19, k20} for bias reduction in the estimation of smooth functionals under high-dimensional models. 
 
Based on Propositions \ref{lowerbound:dense}, \ref{gaussian:case} and Theorem \ref{subgaussian:case}, an argument like the one leading to Corollary \ref{minimax:optimal}(b) yields the following minimax rate results.

\begin{cor}
\label{minimax:optimal:dense}
Consider the dense regime where $s \gtrsim \sqrt{p}$.
\begin{itemize}
\item[(a)] For Gaussian distributions, suppose $\limsup_{n\rightarrow \infty}p/n<1$ or $p\gtrsim n^2$, then
\begin{align*}
 \inf_{\hat{\theta}} \sup_{(\bmu,\bSigma)\in \mathcal{H}(s,\tau)} \EE|\hat{\theta}-\theta| \asymp \Big[\tau \wedge \frac{(1+\tau)\sqrt{p}}{n}\Big]+\Big[\tau \wedge  \frac{\tau+\sqrt{\tau}}{\sqrt{n}}\Big].
\end{align*}
The estimator $\hat{\theta}_g$ is minimax rate-optimal when $\limsup_{n\rightarrow \infty}p/n<1$ and $\tau \gtrsim \frac{\sqrt{p}}{n}$, while the trivial estimator $\bzero$ is optimal when $p\gtrsim n^2$ or $\tau \lesssim \frac{\sqrt{p}}{n}$.
\item[(b)] For sub-Gaussian distributions, suppose $p=O(n^{\alpha})$ for some $\alpha<1$ or $p\gtrsim n^2$, then
\begin{align*}
 \inf_{\hat{\theta}} \sup_{(\bmu,\bSigma)\in \mathcal{H}(s,\tau)} \EE|\hat{\theta}-\theta| \asymp \Big[\tau \wedge \frac{(1+\tau)\sqrt{p}}{n}\Big]+\Big[\tau \wedge  \frac{\tau+\sqrt{\tau}}{\sqrt{n}}\Big].
\end{align*}
The estimator $\hat{\theta}_{sg}$ is minimax rate-optimal when $p=O(n^{\alpha})$ and $\tau \gtrsim \frac{\sqrt{p}}{n}$, and the trivial estimator $\bzero$ is optimal when $p\gtrsim n^2$ or $\tau \lesssim \frac{\sqrt{p}}{n}$.
\end{itemize}
\end{cor}

\noindent {\bf Remark 8.~} Corollaries \ref{minimax:optimal} and \ref{minimax:optimal:dense} together show that the minimax rate of estimation has an elbow at $s=\sqrt{p}$ (ignore the logarithmic term). The same elbow effect has been revealed for quadratic functional estimation in the Gaussian sequence model \citep{cct17}. We should point out that the minimax result under the scaling $p\lesssim s^2, n \lesssim p\lesssim n^2$ is not covered in Corollaries \ref{minimax:optimal} and \ref{minimax:optimal:dense}. The essential difficulty lies in obtaining accurate estimation of $\bSigma^{-1}$ while $n\lesssim p$ and no direct sparsity structure exists in $\bSigma$ or $\bSigma^{-1}$. We leave a thorough investigation of functional estimation under the aforementioned scaling for future research. That said, in Section \ref{dense:combine:all} we discuss some possible solutions when additional assumptions are imposed on the model. If $\bSigma$ is known to be proportional to $\bI_p$, correcting the bias in the dense regime becomes rather manageable. \cite{cct17} provides a complete minimax rate characterization in this reduced case.

\noindent {\bf Remark 9.~} The rate-optimal estimators $\tilde{\theta}$ in the sparse regime and $\{\hat{\theta}_g, \hat{\theta}_{sg}\}$ in the dense regime are obtained via different de-biasing methods. For the former, $\ell_1$ regularization exploits the sparsity structure and also induces substantial bias that needs be corrected to achieve optimal rate; while for the latter, the dominating bias arises from error accumulation in the high dimension.



\section{Numerical experiments}
\label{sec6}

\subsection{Simulation}
\label{simu:setting}

In this section, we perform simulation studies to validate our theoretical results. In particular, we compute the empirical convergence rates over some instances and compare them with the theoretical forms we have derived in Section \ref{sec2}. We further evaluate and compare our proposed estimators with some alternative methods.

To empirically verify the convergence rate, we set $\bmu = \xi\cdot (\bone_s^T, \bzero^T)^T\in \RR^{p}$, where $\bone_s$ is a $s$-dimensional vector with all entries equal to $1$ and $\bzero$ is a vector with all entries equal to $0$, and $\bSigma = \eta\cdot \bI_p$. Under these parameters, we have $\balpha =  \xi\eta^{-1} \cdot(\bone_s^T, \bzero^T)^T$ and $\theta = s \xi^2\eta^{-1}$. We generate data from normal distribution with the sample size $n = \lfloor 2^k \rfloor$ for $k = 10.5, 11, 11.5,\cdots, 15$. Now we take  $p = \lfloor 0.5\cdot n^{0.5}\rfloor + 8, \eta = 2$ and consider the following three settings:

\begin{itemize}
	\item[(1)] $ s = 2, \xi = 1$
	\item[(2)] $ s =  \lfloor n^{0.24} \rfloor, \xi = \lfloor n^{0.24} \rfloor^{-0.5}$
	\item[(3)] $s = \lfloor n^{0.24} \rfloor , \xi = 3\cdot n^{-0.45}$
\end{itemize}

\noindent
For each setting we repeat the experiment for $200$ times, and for each specific $n$ we use formulas \eqref{lasso:alpha} and \eqref{optimal:estimator} to obtain estimators $\tilde{\balpha}^{(1)}, \tilde{\theta}^{(1)}$ in setting (1), $\tilde{\balpha}^{(2)}, \tilde{\theta}^{(2)}$ in setting (2) and $\tilde{\balpha}^{(3)}, \tilde{\theta}^{(3)}$ in setting (3). They are indeed the same estimator but applied to three different simulation settings. The tuning parameters are picked optimally to minimize the estimation error of $\balpha$.

\begin{figure}[htb!]
	\begin{center}
		\includegraphics[width=0.75\linewidth]{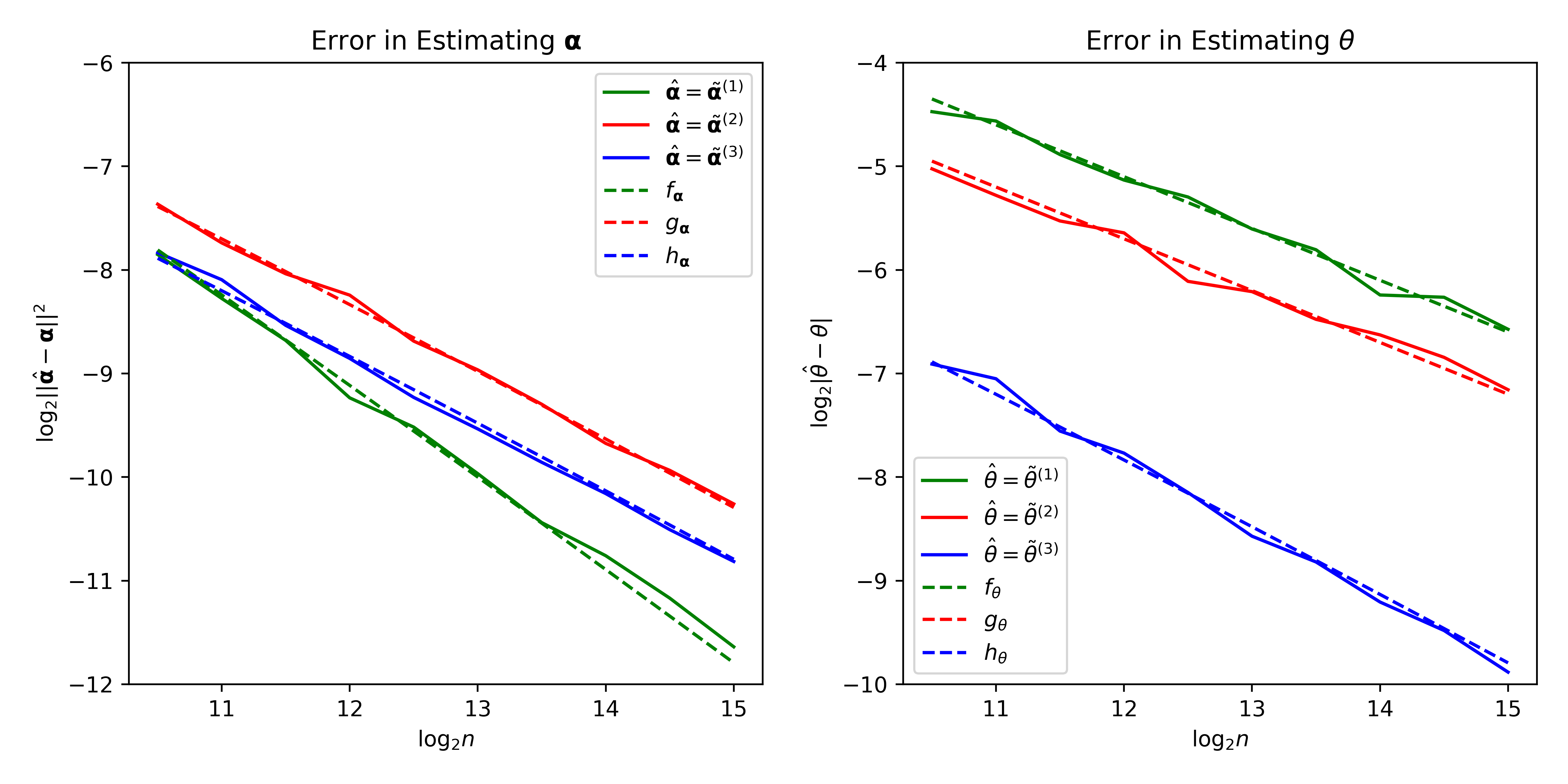}
		\caption{Averaged error v.s. sample size on logarithmic scale for the three settings (solid curves) of Section \ref{simu:setting} along with their corresponding theoretical rates of convergence (dashed lines) for setting 1 (green), 2 (red) and 3 (blue).} \label{fig:one}
	\end{center}
\end{figure}

Figure \ref{fig:one} depicts the averaged error ($\log_2 \|\hat \balpha - \balpha \|^2_2$ for estimating $\balpha$ and $\log_2 |\hat \theta - \theta|$ for estimating $\theta$) versus sample size in the logarithmic scale. To verify our theoretical results, we also show the theoretical rate of convergence in the figure. Specifically, for setting (1), it's easy to check that the form of theoretical convergence rate for estimating $\tilde{\balpha}$ on logarithmic scale is $C + \log_2 \log_2 n - \log_2 n$, where $C$ is a constant. Since the plots are on log-scale, we include the following function in the left plot of Figure \ref{fig:one}:
\[
f_{\balpha}(x) = C_1 + \log_2 x - x
\]
with an  appropriately calibrated  constant $C_1$. In the same way, we add the other two functions for setting (2) and setting (3):
\begin{align*}
&g_{\balpha}(x) = C_2 + \log_2 x - 0.76 \cdot x,   \\
&h_{\balpha}(x) = C_3 + \log_2 x - 0.76 \cdot x.
\end{align*}
For the estimation of $\theta$, similarly, we attach the three functions below to the right plot of Figure \ref{fig:one}:
\begin{align*}
&f_{\theta}(x)=C_4-0.5\cdot x,  \\
&g_{\theta}(x)=C_5-0.5\cdot x, \\
&h_{\theta}(x)=C_6 + \log_2 x - 0.76 \cdot x.
\end{align*}
From Figure \ref{fig:one}, we can see that all the empirical convergence rates (solid curves) are well matched with the theoretical ones (dashed curves).

We next compare the performance of our method with some alternative estimators. The benchmark would be the plug-in estimators:
\[
\hat{\balpha}_P = \hat{\bSigma}^{-1} \hat{\bmu}, \quad \hat{\theta}_P = \frac{|\hat{\bmu}^T \hat{\bSigma}^{-1} \hat{\bmu} - a|}{b},
\]
where $a = \frac{p}{n - p}, b = \frac{n}{n - p}$. The estimator $\hat{\theta}_P$ is a bias-corrected plug-in estimator that was proposed in \cite{ka10} for normal distributions under the scaling $p<n$. We also consider the following Dantzig-type estimator for $\balpha$ \citep{cxw16}:
\begin{align*}
\hat{\balpha}_D=\argmin_{\balpha\in \mathbb{R}^p: \|\hat{\bSigma} \balpha - \hat{\bmu}\|_{\infty} \leqslant \lambda}  \|\balpha\|_1.
\end{align*}
We then construct a family of plug-in estimators for $\theta$:
\[
\hat{\theta}_{D, c} = c\hat{\bmu}^T \hat{\balpha}_D + (1 - c) \hat{\balpha}_D^T \hat{\bSigma} \hat{\balpha}_D.
\]
We perform the comparison under two different scenarios.

In the first scenario, we follow the same model setup from the part that verifies the convergence rate, with parameters $p = \lfloor 0.5\cdot n \rfloor, s = 5, \xi = 2, \eta = 1$ where $n = 60, 80, \cdots, 200$. For a fair comparison, all the estimators are obtained under optimal tuning. Each experiment is repeated $300$ times and the results are shown in Figure \ref{fig:comparison:result:optimal}. Regarding the estimation for $\balpha$, we see that both $\hat{\balpha}_D$ and $\tilde{\balpha}$ have much better performance compared with $\hat{\balpha}_P$. This is expected because the estimator $\hat{\balpha}_P$ does not exploit the sparsity structure in the data. For the estimation of the functional $\theta$, we also observe that $\hat{\theta}_{D, 2}$ and $\tilde{\theta}_2$ outperform $\hat{\theta}_P$ by a large margin. Moreover, it is clear that $\tilde{\theta}_2$ has a better convergence rate compared with the naive plug-in estimators $\tilde{\theta}_0$ and $\tilde{\theta}_1$. This is consistent with our theoretical conclusions in Section \ref{sec2}. As expected, the estimators $\hat{\balpha}_D$ and $\tilde{\balpha}$ have similar performance, and estimators $\tilde{\theta}_{D, c}$ and $\tilde{\theta}_c$ for $c=0,1,2$ have very similar performance too. Such a phenomenon that Lasso and Dantzig type estimators exhibit similar behavior has strong theoretical support in high-dimensional sparse regression \citep{brt09}. We leave the theoretical analysis for estimating functional using Danzig estimator as a future research.

\begin{figure}[h!]
	\begin{center}
		\includegraphics[width=0.75\linewidth]{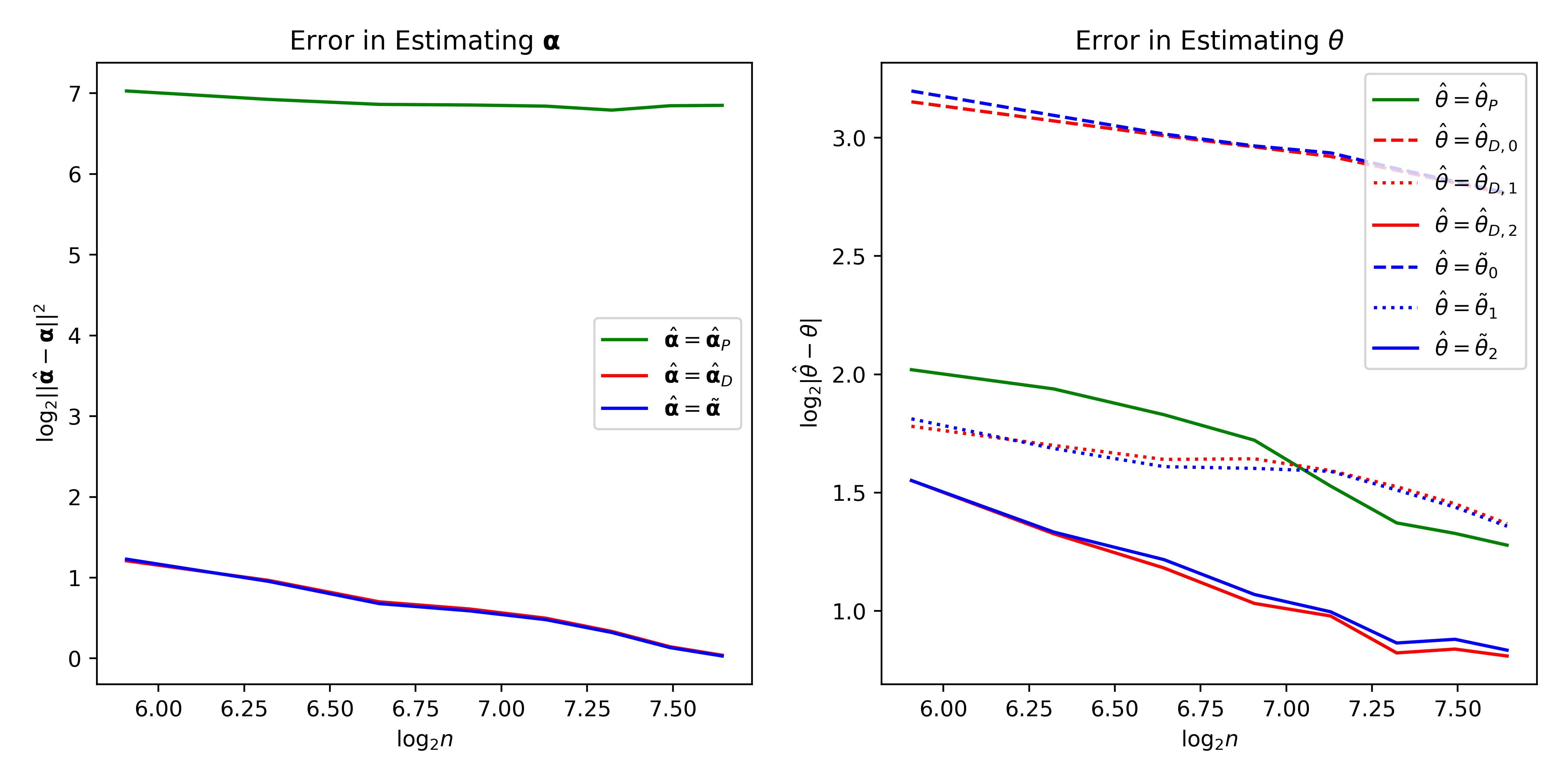}
		\caption{Comparison of different optimally tuned estimators for $\balpha$ and $\theta$ in the first scenario.} \label{fig:comparison:result:optimal}
	\end{center}
\end{figure}

In addition to the comparison of optimally tuned estimators, we also perform the comparison of the aforementioned estimators under 5-fold cross-validation (cv). The implementation of cv for $\hat{\balpha}_D$ can be found in \cite{cxw16}. We now describe cv for our estimator $\tilde{\balpha}$. Suppose we split the data into $m$ folds. For each $j=1,\ldots, m$, we construct the sample estimates $\hat{\bmu}_j$ and $\hat{\bSigma}_j$ from the data in $j$-th fold, and obtain estimator $\tilde{\balpha}_{-j}$ from the rest of the data. We then compute the cv error as
\[
\bar{l} = \frac{1}{m}\sum_{j=1}^m\Big(\frac{1}{2} \tilde{\balpha}_{-j}^T \hat{\bSigma}_j \tilde{\balpha}_{-j} - \hat{\bmu}_j^T \tilde{\balpha}_{-j}\Big).
\]
We choose the parameters that minimize $\bar{l}$ and obtain the corresponding estimator $\tilde{\balpha}$. The comparison is shown in Figure \ref{fig:comparison:result:data}. As is clear from the figure, similar comparison results to the ones under optimal tuning are found in the case of cross-validation.

\begin{figure}[h!]
	\begin{center}
		\includegraphics[width=0.75\linewidth]{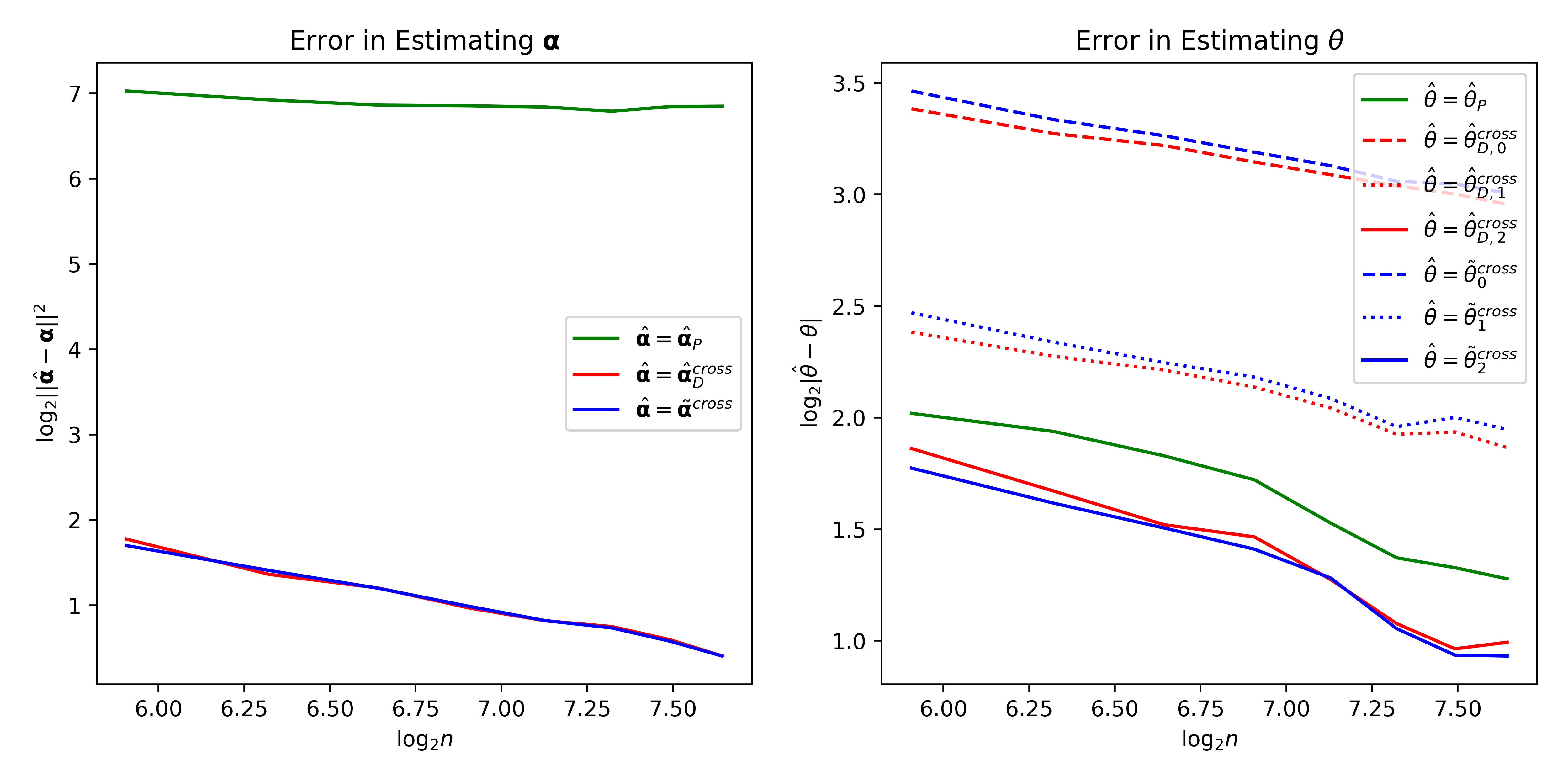}
		\caption{Comparison of different cross-validation tuned estimators for $\balpha$ and $\theta$ in the first scenario.} \label{fig:comparison:result:data}
	\end{center}
\end{figure}

In the second scenario, as a preliminary investigation for our empirical study, we use financial data to calibrate for a low signal-to-noise ratio regime. Specifically, we randomly select $p = 100$ stocks from S\&P500 to compute the sample mean $\hat{\bmu}$ and the sample covariance matrix $\hat{\bSigma}$, using daily data in 2017 and 2018. We hard threshold the vector $\hat{\bSigma}^{-1}\hat{\bmu}$ to keep the top $10$ entries with largest absolute values to obtain a sparse vector as the choice for $\balpha$. We then use $\hat{\bSigma}\balpha$ and $\hat{\bSigma}$ as the values for the mean and covariance matrix parameters to generate multivariate Gaussian data. The average signal-to-noise ratio, defined as $\sum_{j=1}^p |\mu_j|/ (\tr(\bSigma))^{1/2}$, is $0.096$, which is similar to the one from real data, despite the thresholding. Let $n = 260, 280, \cdots, 500$. Due to the low signal-to-noise ratio, we repeat the experiment for $600$ times. The comparisons under optimal tuning and cross-validation are shown in Figures \ref{fig:comparison:2:optimal} and \ref{fig:comparison:2:data}, respectively. Our estimators $\tilde{\balpha}$ and $\tilde{\theta}_2$ perform much better than benchmark estimators and naive plug-in estimators. The functional estimator $\tilde{\theta}_2$ also outperforms Dantzig-type estimators. Also $\tilde{\theta}_0, \tilde{\theta}_1$ have worse performance compared with $\tilde{\theta}_2$.

\begin{figure}[h!]
	\begin{center}
		\includegraphics[width=0.75\linewidth]{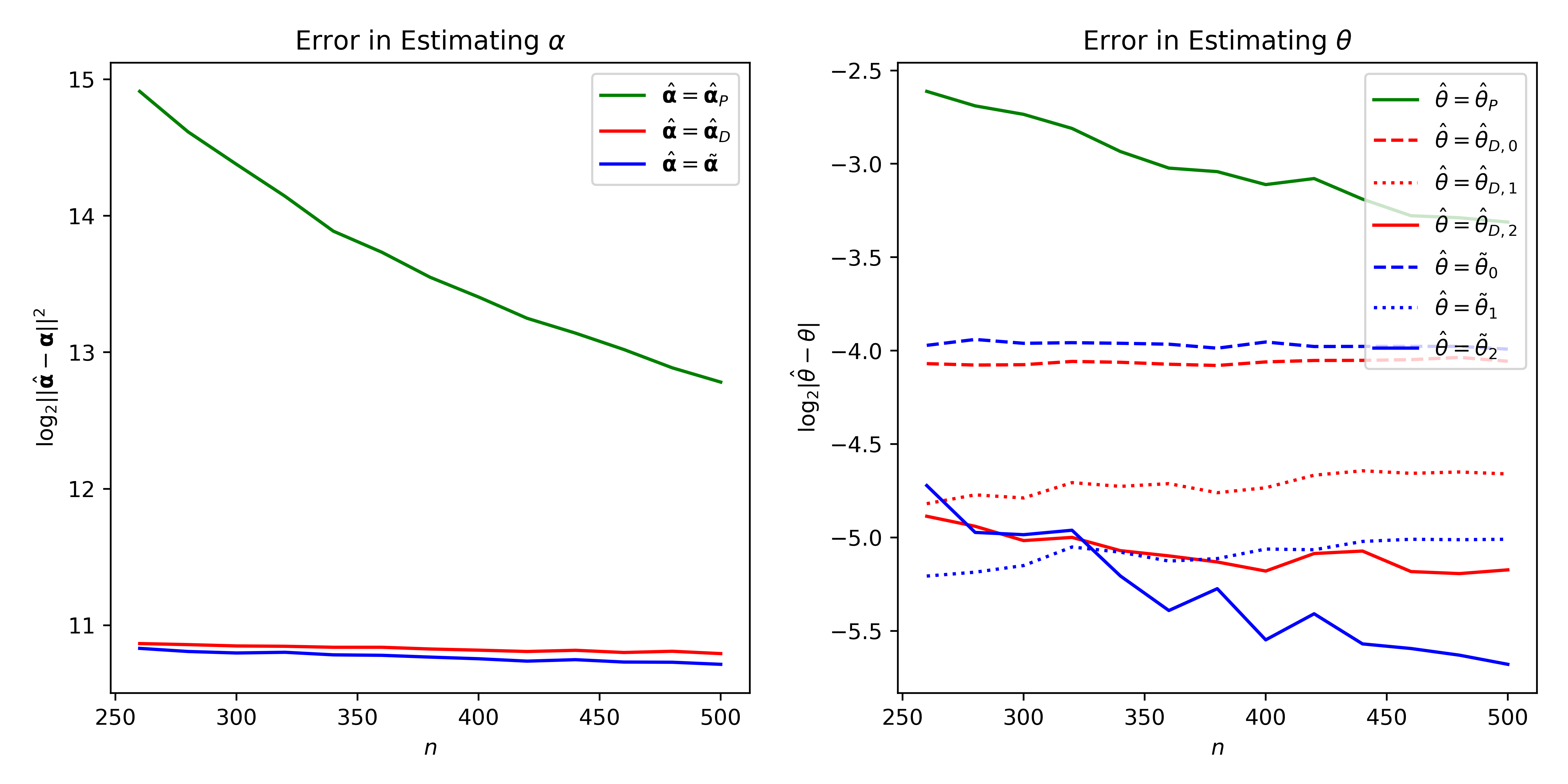}
		\caption{Comparison of different optimally tuned estimators for $\balpha$ and $\theta$ in the second scenario.} \label{fig:comparison:2:optimal}
	\end{center}
\end{figure}

\begin{figure}[h!]
	\begin{center}
		\includegraphics[width=0.75\linewidth]{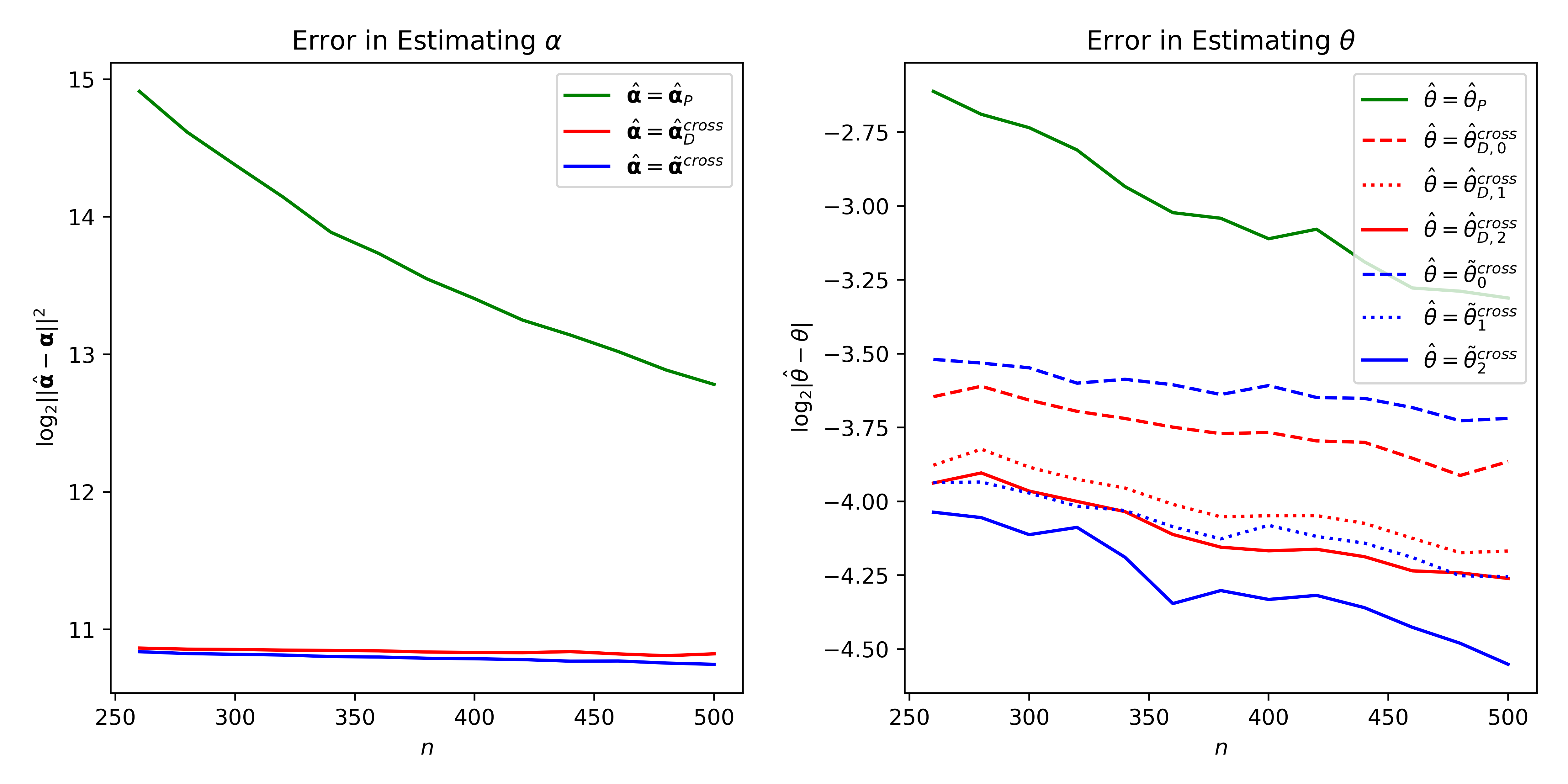}
		\caption{Comparison of different cross-validation tuned estimators for $\balpha$ and $\theta$ in the second scenario.} \label{fig:comparison:2:data}
	\end{center}
\end{figure}

\subsection{Empirical Study\label{sec4.2}}

In this section, we use the daily data for the constituents in the S\&P500 from 2012 to 2018 to construct portfolios and test the portfolio performance in a $5$-year horizon starting from 2014 to 2018. 
Since returns of stocks are highly correlated, we use a single factor, the market portfolio, to adjust their dependence and construct the portfolios based on the factor-adjusted returns (residuals after regressed on the returns of S\&P 500). A $2$-year training window is employed: we use the first 7 quarters of data to estimate $\hat{\bmu}$ and $\hat{\bSigma}$ and tune the parameters in estimating $\balpha$ and $\theta$ to yield the highest Sharpe ratio in the following quarter, which serves as a validation window. Note that we obtain and hedge the market beta in the training window and validation window combined. The testing window is $1$ month where we will hedge the beta previously obtained and test the portfolio trained with data in past $2$-year, and then we roll the window forward after the testing period, i.e. we re-balance our portfolio monthly. For comparison, we provide benchmarks the market and equal-weighted portfolios, and also the minimum-variance portfolio with gross-exposure constraint \citep{fzy12}. In addition, we impose the constraint on short-selling such that a short position for each individual stock  cannot exceed $25\%$ of the principal and the total short position cannot exceed $50\%$. Whenever this constraint is violated, we would rescale all the positions on risky assets to satisfy this constraint (the rest invested in cash).

\begin{figure}[h!]
	\begin{center}
		\includegraphics[width=0.55\linewidth]{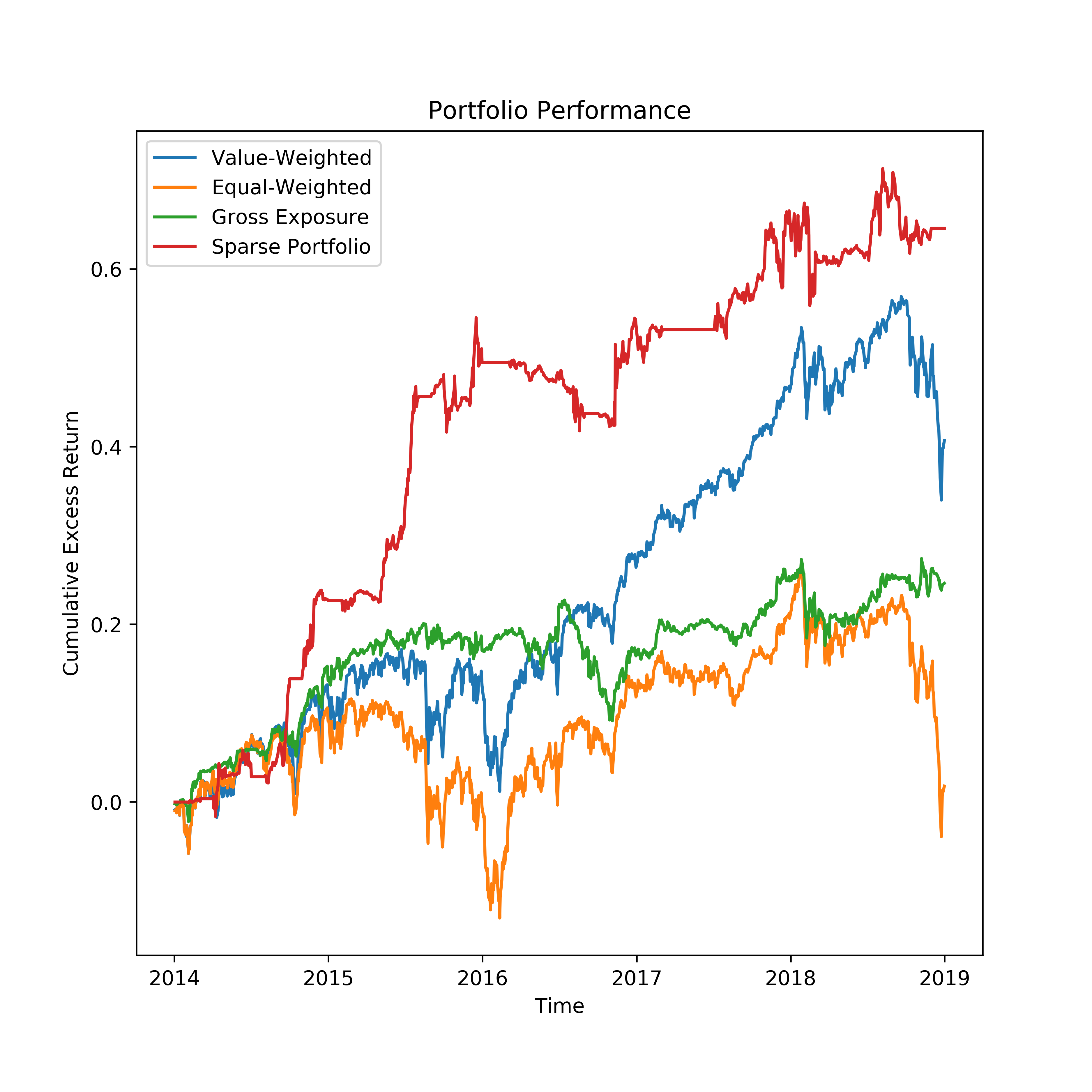}
		\caption{Portfolio performance.} \label{fig:final:one}
	\end{center}
\end{figure}

Figure \ref{fig:final:one} depicts the cumulative excess returns of the aforementioned strategies for constructing portfolios. Obviously, the sparse portfolio constructed by our estimators $\tilde{\balpha}, \tilde{\theta}_2$ (see \eqref{eq:fan0}) outperforms other portfolios, in terms of the annualized return and Shape ratio. Moreover, during two correction periods in these $5$ years, the sparse portfolio has smaller pullback compared with the market, and the same phenomenon can be observed for the minimum variance portfolio with gross exposure constraint, which is expected to performance in terms of stability and maximum draw-down.\\

\begin{table}[!h]
	\small
	\begin{center}
		\begin{tabular}{c|ccccccc}
			\hline
			\hline
			 & Value-Weighted & & Equal-Weighted & & Gross Exposure & & Sparse Portfolio\\
			\hline
			Annual Return & $8.15\%$ & & $0.36\%$ & & $4.93\%$ & & $12.94\%$\\
			Volatility & $13.53\%$ & & $13.73\%$ & & $7.76\%$ & & $12.26\%$\\
			Sharpe Ratio & $0.60$ & & $0.03$ & & $0.64$ & & $1.06$\\
			Maximum Draw-down & $-22.92\%$ & & $-30.29\%$ & & $-13.58\%$ & & $-12.75\%$\\
			Alpha & $0.00\%$ & & $-7.64\%$ & & $1.87\%$ & & $13.28\%$\\
			Beta & $1.00$ & & $0.98$ & & $0.38$ & & $-0.04$\\
			\hline
		\end{tabular}
		
		\caption{Portfolio characteristics}\label{table:portfolio}
	\end{center}	
\end{table}

Table \ref{table:portfolio} provides some details for each portfolio. All numbers except maximum draw-down and beta are annualized. We can see that the sparse portfolio has low correlation with the market and its alpha over the market is more than $13\%$, and at the same time it shares low maximum draw-down as the minimum variance portfolio with gross exposure constraint. Understandably, its Sharpe ratio is the highest.

\section{Proof of the main results}
\label{sec:proofs}

The section contains the proof of all the main results. The organization is as follows:
\begin{enumerate}
\item Section \ref{tech:lemma} collects a few lemmas that will be useful in the later proof.
\item Section \ref{proof:thm:one:index} proves Theorem \ref{thm:one} and Proposition \ref{general:debiasing}.
\item Section \ref{proof:thm:two:index} proves Theorem \ref{lower:bound} and Proposition \ref{prop:one}.
\item Section \ref{proof:corollary2} proves Proposition \ref{prop:two}.
\item Section \ref{proof:corollary:three} proves Theorem \ref{corollary:three}.
\item Section \ref{proof:of:thm3} proves Theorem \ref{approximate:sparse:thm1}.
\item Section \ref{proof:lower:bound:lqball} proves Theorem \ref{approxiamte:sparse:thm2}.
\item Section \ref{dense:combine:all} collects the proof of Proposition \ref{lowerbound:dense}, Proposition \ref{gaussian:case} and Theorem \ref{subgaussian:case}, and some further discussions of functional estimation in the dense regime.
\item Section \ref{reference:material} puts together some reference materials.
\end{enumerate}

We introduce more notations below.

\noindent \textbf{Notation.} For an integer $k \geq 1$, let $[k]=\{1,2,\ldots, k\}$. For a set $S \subseteq [k]$, its cardinality is $|S|$; $\bu_S \in \mathbb{R}^{|S|}$ is the subvector of $\bu\in \RR^k$ indexed by $S$, and $\bA_{SS} \in \RR^{|S|\times |S|}$ is the submatrix of $\bA \in \RR^{k\times k}$ whose rows and columns are both indexed by $S$. For a matrix $\bA \in \RR^{k \times k}$, $\|\bA\|_F, \ltwonorm{\bA}$ and $\linfnorm{\bA}$ represent its Frobenius norm, spectral norm and $\ell_{\infty}/\ell_{\infty}$ operator norm, respectively. We use $\{\be_i\}_{i=1}^n$ to denote the standard basis in $\RR^n$, and $\diag(v_1,v_2,\ldots,v_k)$ to represent a $k \times k$ diagonal matrix with diagonal elements $v_1,\ldots, v_k$. For a number $x\in \mathbb{R}$, $(x)_+=0 \vee x$ denotes its positive part. Further define
\begin{align}
&\cK(s)=\Big\{\bu  \in \mathbb{R}^p: \lonenorm{\bu_{S^c}} \leq 3 \lonenorm{\bu_{S}}, |S| \leq s, S\subseteq [p] \Big\}, \nonumber \\
&\cA(\kappa)=\Big\{\linfnorm{(\hat{\bSigma}-\bSigma)\balpha} \leq \kappa/4, \linfnorm{\hat{\bmu}-\bmu}  \leq \kappa/4 \Big\}, \nonumber \\
&\cB(s, \kappa)=\Big\{\max_{\bu\in \mathcal{K}(s)\cap B_2(1) }\big |\sqrt{\bu^T \bSigma \bu} - \sqrt{\bu^T \hat{\bSigma} \bu} \big| \leq \kappa \Big\}, \nonumber  \\
&\cC(s, \kappa)=\Big\{\max_{\bu\in B_0(2s)\cap B_2(1) } \big |\sqrt{\bu^T \bSigma \bu} - \sqrt{\bu^T \hat{\bSigma} \bu} \big| \leq \kappa \Big\}. \nonumber
\end{align}

\subsection{Technical lemmas}
\label{tech:lemma}

\begin{lem}
\label{rare:events}
The followings hold:
\begin{itemize}
\item[(i)] For any given $\lambda >0$,
\begin{equation*}
\mathbb{P}(\mathcal{A}(\lambda)) \geq 1-8p\exp\Big(-\frac{c_1n\lambda^2}{\nu^2}\cdot \min\Big(\frac{1}{\lambda(\delta_{\bSigma}+1)\sqrt{\theta}},\frac{1}{(\nu^2\theta+1)\delta_{\bSigma} }\Big) \Big).
\end{equation*}
\item[(ii)] For any given $t\geq 0$,
\begin{equation*}
\mathbb{E}\Big(\|\hat{\bmu}-\bmu \|^2_{\infty}\cdot \mathbbm{1}_{\|\hat{\bmu}-\bmu \|_{\infty} >t}\Big) \leq  2p\cdot (t^2+(c_2n)^{-1} \delta_{\bSigma}\nu^2)\cdot e^{\frac{-c_2nt^2}{\nu^2\delta_{\bSigma}}}.
\end{equation*}
\item[(iii)] For any given $t\geq 0$,
\begin{equation*}
\mathbb{P}(\|\bI_p-\bSigma^{-1}\hat{\bSigma}\|_{\max}\geq t)\leq 6p^2\exp\Big(\frac{-c_3nt(1\wedge t)}{(\nu^4+\nu^2)(\delta_{\bSigma}+1)(\lambda^{-1}_{\min}(\bSigma)+1)}\Big)
\end{equation*}
\end{itemize}
Here $c_1,c_2, c_3>0$ are absolute constants.
\end{lem}

\begin{proof}
Throughout the proof, $C_i~(i=1,2,\ldots)$ are positive absolute constants.

Part (i):  We first bound $\|\hat{\bmu}-\bmu\|_{\infty}$. Since for each $1\leq j \leq p$, $\be_j^T(\bx_i-\bmu)$ is zero-mean and
\begin{equation*}
\|\be_j^T(\bx_i-\bmu)\|_{\psi_2} \leq \sqrt{\be_j^T\bSigma \be_j}\cdot \|\by_i\|_{\psi_2}\leq \sqrt{\delta_{\bSigma}}\cdot \nu,
\end{equation*}
the general Hoeffding's inequality (cf. Theorem \ref{hoeffding:quote}) enables us to obtain $\forall t \geq 0$,
\begin{equation}\label{mu:infinity}
\mathbb{P}(\|\hat{\bmu}-\bmu\|_{\infty}>t) \leq \sum_{j=1}^p\mathbb{P}\bigg(\bigg|\frac{1}{n}\sum_{i=1}^n\be_j^T(\bx_i-\bmu)\bigg|>t \bigg)\leq 2p\exp\bigg(-\frac{C_1n t^2}{\nu^2\delta_{\bSigma}}\bigg).
\end{equation}
Regarding the bound for $\|(\hat{\bSigma}-\bSigma)\balpha\|_{\infty}$, because of the following identity
\begin{equation}\label{following:identity}
(\hat{\bSigma}-\bSigma)\balpha=\frac{1}{n}\sum_{i=1}^n\Big[(\bx_i-\bmu)(\bx_i-\bmu)^T\balpha-\bSigma \balpha \Big]-(\hat{\bmu}-\bmu)(\hat{\bmu}-\bmu)^T\balpha,
\end{equation}
we bound the two terms on the above right-hand side, respectively. For each $1\leq j \leq p$,
\begin{align*}
&\|\be_j^T(\bx_i-\bmu)(\bx_i-\bmu)^T\balpha-\be_j^T\bSigma \balpha\|_{\psi_1} \leq C_2 \cdot \|\be_j^T(\bx_i-\bmu)(\bx_i-\bmu)^T\balpha\|_{\psi_1} \\
&\leq C_2 \cdot \|\be_j^T(\bx_i-\bmu)\|_{\psi_2} \cdot \|(\bx_i-\bmu)^T\balpha\|_{\psi_2} \leq C_2 \cdot \sqrt{\delta_{\bSigma}}\nu \cdot \sqrt{\theta} \nu,
\end{align*}
where in the first two inequalities we have used basic properties of sub-exponential norm (cf. Lemma 2.7.7 and Exercise 2.7.10 in \cite{versh18}), and the last inequality holds because $\bx_i-\bmu=\bSigma^{1/2}\by_i$ with $\subgnorm{\by_i}=\nu$. We can then apply Bernstein's inequality (cf. Theorem \ref{hoeffding:quote}) to obtain
\begin{align}
&\mathbb{P}\bigg(\Big \|\frac{1}{n}\sum_{i=1}^n \big[(\bx_i-\bmu)(\bx_i-\bmu)^T\balpha-\bSigma \balpha \big] \Big \|_{\infty}>\frac{\lambda}{8}\bigg)   \nonumber \\
\leq & \sum_{j=1}^p \mathbb{P}\bigg(\Big|\frac{1}{n}\sum_{i=1}^n \big[\be_j^T(\bx_i-\bmu)(\bx_i-\bmu)^T\balpha-\be_j^T\bSigma \balpha \big]\Big|>\frac{\lambda}{8}\bigg) \nonumber \\
\leq & 2p \exp\bigg(-\frac{C_3n\lambda^2}{\nu^2}\cdot \min\bigg(\frac{1}{\nu^2\delta_{\bSigma} \theta},\frac{1} {\lambda\sqrt{\delta_{\bSigma}\theta}}\bigg)\bigg). \label{sigma:alpha:one}
\end{align}
Moreover, since for each $1\leq i \leq n$, $\balpha^T(\bx_i-\bmu)$ is zero-mean and $\|\balpha^T(\bx_i-\bmu)\|_{\psi_2} \leq \sqrt{\theta} \nu$, the general Hoeffding's inequality (cf. Theorem \ref{hoeffding:quote}) gives us
\begin{eqnarray}\label{mean:nonunif}
\mathbb{P}(|\balpha^T(\hat{\bmu}-\bmu)|>t)=\mathbb{P}\bigg(\Big|\frac{1}{n}\sum_{i=1}^n\balpha^T(\bx_i-\bmu) \Big|>t\bigg)\leq 2\exp\big(-\frac{C_4nt^2}{\theta \nu^2}\big), \quad \forall t >0.
\end{eqnarray}
Based on the results from \eqref{mu:infinity} and \eqref{mean:nonunif}, it holds that
\begin{align}
&\mathbb{P}\bigg(\|(\hat{\bmu}-\bmu)(\hat{\bmu}-\bmu)^T\balpha\|_{\infty}>\frac{\lambda}{8} \bigg) \nonumber \\
\leq& \mathbb{P}\bigg(\|\hat{\bmu}-\bmu\|_{\infty}>\frac{\lambda^{1/2}}{4\theta^{1/4}}\bigg)+\mathbb{P}\bigg(|(\hat{\bmu}-\bmu)^T\balpha|>\frac{\lambda^{1/2}\theta^{1/4}}{2}\bigg) \nonumber \\
\leq &2p\exp\bigg(-\frac{C_1n\lambda}{16\nu^2\delta_{\bSigma}\sqrt{\theta}}\bigg)+ 2\exp\bigg(-\frac{C_4n\lambda}{4\sqrt{\theta} \nu^2}\bigg) \nonumber \\
 \leq &4p \exp\bigg(-\frac{C_5n \lambda}{\nu^2\sqrt{\theta}(\delta_{\bSigma}+1)}\bigg). \label{sigma:alpha:two}
\end{align}
Combining the results \eqref{following:identity}, \eqref{sigma:alpha:one} and \eqref{sigma:alpha:two} we obtain
\begin{equation}
\label{final:eq:one}
\PP\Big(\linfnorm{(\hat{\bSigma}-\bSigma)\balpha}>\frac{\lambda}{4}\Big)\leq 6p\exp\bigg(-\frac{C_6n\lambda^2}{\nu^2}\cdot \min\Big(\frac{1}{\nu^2\delta_{\bSigma}\theta},\frac{1}{\lambda \sqrt{\theta}(\delta_{\bSigma}+1)}\Big)\bigg).
\end{equation}
Therefore, the proof is completed by using \eqref{mu:infinity} and \eqref{final:eq:one} in the following bound,
\begin{align*}
\PP(\cA(\lambda))\geq 1-\PP(\linfnorm{(\hat{\bSigma}-\bSigma)\balpha}>\lambda/4)-\PP(\linfnorm{\hat{\bmu}-\bmu}>\lambda/4).
\end{align*}

Part (ii): Using the integral identity $\EE(z)=\int_0^{\infty}\PP(z>s)ds$ for any non-negative random $z \in \mathbb{R}_+$, we have
\begin{align*}
&\mathbb{E}\Big(\|\hat{\bmu}-\bmu \|^2_{\infty}\cdot \mathbbm{1}_{\|\hat{\bmu}-\bmu \|_{\infty} >t} \Big)=\int_0^{\infty} \mathbb{P}(\|\hat{\bmu}-\bmu \|^2_{\infty}\mathbbm{1}_{\|\hat{\bmu}-\bmu \|_{\infty} >t}>s)ds \\
&= t^2\mathbb{P}(\|\hat{\bmu}-\bmu \|_{\infty}>t)+\int_{t^2}^{\infty}\mathbb{P}(\|\hat{\bmu}-\bmu \|^2_{\infty}>s)ds \\
&= t^2\mathbb{P}(\|\hat{\bmu}-\bmu \|_{\infty}>t)+2\int_{t}^{\infty}s\mathbb{P}(\|\hat{\bmu}-\bmu \|_{\infty}>s)ds \\
&\leq  2t^2pe^{\frac{-C_1nt^2}{\nu^2 \delta_{\bSigma}}}+4p\int_t^{\infty}se^{\frac{-C_1ns^2}{\nu^2 \delta_{\bSigma}}}ds = 2p\cdot (t^2+(nC_1)^{-1} \delta_{\bSigma}\nu^2)\cdot e^{\frac{-C_1nt^2}{\nu^2\delta_{\bSigma}}},
\end{align*}
where the inequality above is due to \eqref{mu:infinity}. \\

Part (iii): Given that $\bx_i=\bmu+\bSigma^{1/2}\by_i$, if we define $\bar{\by}=n^{-1}\sum_{i=1}^n\by_i$, it is direct to confirm
\begin{align*}
\|\bI_p-\bSigma^{-1}\hat{\bSigma}\|_{\max}\leq \| n^{-1}\sum_{i=1}^n \bSigma^{-1/2}\by_i\by_i^T\bSigma^{1/2}-\bI_p\|_{\max}+\linfnorm{\bSigma^{1/2}\bar{\by}}\cdot \linfnorm{\bSigma^{-1/2}\bar{\by}}
\end{align*}
Since $\|\be_j^T\bSigma^{-1/2}\by_i\by_i^T\bSigma^{1/2}\be_{\tilde{j}}-\be_j^T\be_{\tilde{j}}\|_{\psi_1}\leq \delta^{1/2}_{\bSigma}\lambda^{-1/2}_{\min}(\bSigma)\nu^2$ for all $1\leq j,\tilde{j}\leq p$, the Bernstein's inequality from Theorem \ref{hoeffding:quote} combined with a union bound yields
\begin{align}
\label{finish:one}
&\mathbb{P}\Big(\| n^{-1}\sum_{i=1}^n \bSigma^{-1/2}\by_i\by_i^T\bSigma^{1/2}-\bI_p\|_{\max}\geq t\Big)\nonumber \\
\leq &2p^2\exp\Big(-C_1\min\big(\frac{nt^2}{\delta_{\bSigma}\lambda^{-1}_{\min}(\bSigma)\nu^4},\frac{nt}{\delta_{\bSigma}^{1/2}\lambda^{-1/2}_{\min}(\bSigma)\nu^2}\big)\Big), \quad \forall t\geq 0.
\end{align}
Moreover, as $\|\be_j^T\bSigma^{1/2}\by_i\|_{\psi_2}\leq \delta^{1/2}_{\bSigma}\nu, \|\be_j^T\bSigma^{-1/2}\by_i\|_{\psi_2}\leq \lambda^{-1/2}_{\min}(\bSigma)\nu$ for all $1\leq j\leq p$, the general Hoeffding's inequality in Theorem \ref{hoeffding:quote} combined with a union bound shows 
\begin{align}
\label{finish:two}
\mathbb{P}(\linfnorm{\bSigma^{1/2}\bar{\by}}\cdot \linfnorm{\bSigma^{-1/2}\bar{\by}}\geq t)\leq 4p\exp\Big(\frac{-C_2nt}{\nu^2(\delta_{\bSigma}+\lambda^{-1}_{\min}(\bSigma))}\Big)
\end{align}
Putting together \eqref{finish:one} and \eqref{finish:two} finishes the proof.
\end{proof}

\begin{lem}
\label{key:concentration:eq}
For any given constant $c>0$, with probability at least $1-4p^{-c}$, it holds that
\begin{align*}
\Big|\sqrt{\bu^T\hat{\bSigma}\bu}-\sqrt{\bu^T\bSigma\bu} \Big| \leq &~~ c_1\nu^2 \max(\delta_{\bSigma}^{1/2},1)\sqrt{\frac{\log p}{n}}  \|\bu\|_1+ \\
&\frac{c_2 \nu^2 \delta_{\bSigma} \frac{\log p}{n} \|\bu\|^2_1}{(\sqrt{\bu^T\bSigma \bu}-c_1\nu^2 \max(\delta_{\bSigma}^{1/2},1)\sqrt{\frac{\log p}{n}}  \|\bu\|_1)_+}, \quad~~ \forall \bu \in \RR^p,
\end{align*}
where $c_1,c_2>0$ are constants only depending on $c$, and $c_1,c_2\rightarrow \infty$, as $c\rightarrow \infty$.
\end{lem}
\begin{proof}
Throughout the proof, we use $C_i~(i=1,2,\ldots)$ to denote positive absolute constants. Since $\hat{\bSigma}=\frac{1}{n}\sum_{i=1}^n(\bx_i-\bmu)(\bx_i-\bmu)^T-(\hat{\bmu}-\bmu)(\hat{\bmu}-\bmu)^T$, we have
\begin{align}
\Big|\sqrt{\bu^T\hat{\bSigma}\bu}-\sqrt{\bu^T\bSigma\bu} \Big| \leq& \Big| \Big(n^{-1}\sum_{i=1}^n|\bu^T(\bx_i-\bmu)|^2\Big)^{\frac{1}{2}}-\sqrt{\bu^T\bSigma \bu} \Big|+ \Big|\sqrt{\bu^T\hat{\bSigma}\bu}- \Big(n^{-1}\sum_{i=1}^n|\bu^T(\bx_i-\bmu)|^2\Big)^{\frac{1}{2}}\Big| \nonumber \\
&\hspace{-2.5cm}\leq \Big| \Big(n^{-1}\sum_{i=1}^n|\bu^T(\bx_i-\bmu)|^2\Big)^{\frac{1}{2}}-\sqrt{\bu^T\bSigma \bu} \Big|+\frac{|\bu^T(\hat{\bmu}-\bmu)|^2}{\Big(n^{-1}\sum_{i=1}^n|\bu^T(\bx_i-\bmu)|^2\Big)^{\frac{1}{2}}}. \label{main:structure}
\end{align}
We bound the two terms on the right-hand side of the last inequality, respectively. Regarding the second term, from \eqref{mu:infinity} we have that $\forall a \geq 0$,
\begin{equation}
\label{bound:second:term}
|\bu^T(\hat{\bmu}-\bmu)|^2\leq \|\hat{\bmu}-\bmu\|_{\infty}^2 \cdot \|\bu\|^2_1 \leq a \nu^2 \delta_{\bSigma} \frac{\log p}{n} \|\bu\|^2_1
\end{equation}
holds with probability at least $1-2p^{1-C_1a}$. Once a bound for the first term is derived, using \eqref{bound:second:term} we can obtain the bound for the second term. The main part
of the proof uses the matrix deviation inequality (cf. Theorem \ref{matrix:deviation}) together with a slicing argument \citep{vanhandel16} to bound the first term.

Towards that goal,  let $\bA$ be a $n\times p$ matrix whose $i$th row is $(\bx_i-\bmu)^T\bSigma^{-1/2}$, for $i=1,2,\ldots, n$. It is straightforward to verify that
\begin{align}
\Big| \Big(n^{-1}\sum_{i=1}^n|\bu^T(\bx_i-\bmu)|^2\Big)^{\frac{1}{2}}-\sqrt{\bu^T\bSigma \bu} \Big|=\big |n^{-\frac{1}{2}}\|\bA\bSigma^{1/2}\bu\|_2-\|\bSigma^{1/2}\bu\|_2\big| \label{matrix:deviate:bound}
\end{align}
Define
\begin{equation*}
\cT(r)=\Big\{\bu\in\RR^p:\|\bSigma^{1/2}\bu\|_{\infty}= 1, \|\bSigma^{-1/2}\bu\|_1\leq r \Big \}, {\rm ~~~for~} r >0
\end{equation*}
with which it is clear that
\begin{equation}\label{transform:eq}
 \sup_{\|\bSigma \bu\|_{\infty}=1,\|\bu\|_1\leq r}\Big |n^{-\frac{1}{2}}\|\bA\bSigma^{1/2}\bu\|_2-\|\bSigma^{1/2}\bu\|_2\Big |=\sup_{\bu\in \cT(r)} \Big |n^{-\frac{1}{2}}\|\bA \bu\|_2-\|\bu\|_2\Big |.
\end{equation}
The matrix deviation inequality (cf. Theorem \ref{matrix:deviation}) enables us to conclude that $\forall b \geq 0$,
\begin{equation}\label{basic:bound:use}
\mathbb{P}\Big(\sup_{\bu\in \cT(r)}|n^{-\frac{1}{2}}\|\bA\bu\|_2-\|\bu\|_2| \leq C_2n^{-1/2}\nu^2 \big(w(\cT(r))+b \cdot {\rm rad}(\cT(r)) \big) \Big) \geq 1-2e^{-b^2}.
\end{equation}
Here,
\[
{\rm rad}(\cT(r))=\sup_{\bu\in \cT(r)}\|\bu\|_2\leq \sup_{\bu\in \cT(r)} \sqrt{\|\bSigma^{1/2}\bu\|_{\infty}\cdot \|\bSigma^{-1/2}\bu\|_1} \leq \sqrt{r},
\]
and
\[
w(\cT(r))=\EE\sup_{u\in \cT(r)} \bg^T\bu \leq \mathbb{E}\sup_{\bu \in \cT(r)} \|\bSigma^{\frac{1}{2}}\bg\|_{\infty}\cdot \|\bSigma^{-\frac{1}{2}}\bu\|_1 \leq r \mathbb{E}\|\bSigma^{\frac{1}{2}}\bg\|_{\infty}\leq C_3 \delta_{\bSigma}^{1/2}  r\sqrt{\log p},
\]
where the last inequality follows from Theorem \ref{maximal:noteq}. Therefore, setting $b=\max(\delta_{\bSigma}^{\frac{1}{2}},1)\tilde{b}C_3\sqrt{r\log p}$ in \eqref{basic:bound:use} combined with \eqref{transform:eq} yields that $\forall r> 0, \tilde{b}\geq 0$,
\begin{equation}\label{fixed:r}
\mathbb{P}\bigg( \sup_{\|\bSigma \bu\|_{\infty}=1,\|\bu\|_1\leq r}\Big|n^{-\frac{1}{2}}\|\bA\bSigma^{1/2}\bu\|_2-\|\bSigma^{1/2}\bu\|_2\Big| \geq C_4(1+\tilde{b})\nu^2r\max(\delta_{\bSigma}^{\frac{1}{2}},1)\sqrt{\frac{\log p}{n}} \bigg) \leq 2p^{-C_5\tilde{b}^2r\max(\delta_{\bSigma},1)}.
\end{equation}

We now utilize a slicing method to derive an upper bound on $|n^{-\frac{1}{2}}\|\bA\bSigma^{1/2}\bu\|_2-\|\bSigma^{1/2}\bu\|_2|$. Towards that end, denote $r_k=2^k, k=k_0, k_0+1,k_0+2\ldots$, where $k_0=\lfloor \log_2\delta_{\bSigma}^{-1} \rfloor$. We then have
\begin{align}
\label{slicing:argument}
&\mathbb{P}\Bigg(\sup_{\|\bSigma \bu\|_{\infty}=1} \frac{|n^{-\frac{1}{2}}\|\bA\bSigma^{1/2}\bu\|_2-\|\bSigma^{1/2}\bu\|_2|}{\lonenorm{\bu}} > t \Bigg) \nonumber \\
=&\mathbb{P}\Bigg(\sup_{k\geq k_0} \sup_{\substack{r_{k} \leq \lonenorm{\bu}\leq r_{k+1} \\ \|\bSigma \bu\|_{\infty}=1}} \frac{|n^{-\frac{1}{2}}\|\bA\bSigma^{1/2}\bu\|_2-\|\bSigma^{1/2}\bu\|_2|}{\lonenorm{\bu}} > t\Bigg) \nonumber \\
\leq & \sum_{k\geq k_0} \mathbb{P}\Bigg( \sup_{\substack{r_{k} \leq \lonenorm{\bu}\leq r_{k+1} \\  \|\bSigma \bu\|_{\infty}=1}} \frac{|n^{-\frac{1}{2}}\|\bA\bSigma^{1/2}\bu\|_2-\|\bSigma^{1/2}\bu\|_2|}{\lonenorm{\bu}} > t\Bigg)  \nonumber \\
\leq &\sum_{k\geq k_0} \mathbb{P}\Bigg( \sup_{\|\bSigma \bu\|_{\infty}=1, \lonenorm{\bu}\leq r_{k+1}} |n^{-\frac{1}{2}}\|\bA\bSigma^{1/2}\bu\|_2-\|\bSigma^{1/2}\bu\|_2| > r_{k} \cdot t\Bigg),
\end{align}
where the first equality holds since $\|\bSigma \bu\|_{\infty} \leq \delta_{\bSigma}\cdot \lonenorm{\bu}$. Choosing $t=2C_4(1+\tilde{b})\nu^2\max(\delta_{\bSigma}^{\frac{1}{2}},1)\sqrt{\frac{\log p}{n}}$ and using the result \eqref{fixed:r}, we can continue from \eqref{slicing:argument} to obtain that $\forall \tilde{b} \geq 0$
\begin{align*}
&\mathbb{P}\Bigg(\sup_{ \| \bSigma\bu \|_{\infty}=1} \frac{|n^{-\frac{1}{2}}\|\bA\bSigma^{1/2}\bu\|_2-\|\bSigma^{1/2}\bu\|_2|}{\lonenorm{\bu}} > 2C_4(1+\tilde{b})\nu^2 \max(\delta_{\bSigma}^{\frac{1}{2}},1)\sqrt{\frac{\log p}{n}} \Bigg) \\
\leq&2\sum_{k\geq k_0}\Big(p^{-2C_5\tilde{b}^2\max(\delta_{\bSigma},1)}\Big)^{2^k}\leq 2\sum_{m=1  }^{\infty}\Big(p^{-C_5\tilde{b}^2\max(\delta_{\bSigma},1) \delta_{\bSigma}^{-1} }\Big)^{m}  \leq \frac{2}{p^{C_5\tilde{b}^2}-1}.
\end{align*}
This is equivalent to saying that with probability at least $1-2(p^{C_5\tilde{b}^2}-1)^{-1}$,
\begin{align*}
\big |n^{-\frac{1}{2}}\|\bA\bSigma^{1/2}\bu\|_2-\|\bSigma^{1/2}\bu\|_2 \big | \leq 2C_4(1+\tilde{b})\nu^2 \max(\delta_{\bSigma}^{\frac{1}{2}},1)\sqrt{\frac{\log p}{n}}  \|\bu\|_1, ~~~\forall \bu \in \RR^p.
\end{align*}
The above result together with \eqref{matrix:deviate:bound} shows that with probability at least $1-2(p^{C_5\tilde{b}^2}-1)^{-1}$,
\begin{align}
\Big| \Big(n^{-1}\sum_{i=1}^n|\bu^T(\bx_i-\bmu)|^2\Big)^{\frac{1}{2}}-\sqrt{\bu^T\bSigma \bu} \Big|  \leq 2C_4(1+\tilde{b})\nu^2 \max(\delta_{\bSigma}^{\frac{1}{2}},1)\sqrt{\frac{\log p}{n}}  \|\bu\|_1, ~~~\forall \bu \in \RR^p. \label{bound:first:term}
\end{align}

Combining the results \eqref{main:structure}, \eqref{bound:second:term} and \eqref{bound:first:term} gives us that with probability at least $1-2(p^{C_5\tilde{b}^2}-1)^{-1}-2p^{1-C_1a}$,
\begin{align*}
\Big|\sqrt{\bu^T\hat{\bSigma}\bu}-\sqrt{\bu^T\bSigma\bu} \Big| \leq&~~ 2C_4(1+\tilde{b})\nu^2 \max(\delta_{\bSigma}^{\frac{1}{2}},1)\sqrt{\frac{\log p}{n}}  \|\bu\|_1+ \\
&\frac{a \nu^2 \delta_{\bSigma} \frac{\log p}{n} \|\bu\|^2_1}{(\sqrt{\bu^T\bSigma \bu}-2C_4(1+\tilde{b})\nu^2 \max(\delta_{\bSigma}^{\frac{1}{2}},1)\sqrt{\frac{\log p}{n}}  \|\bu\|_1)_+}.
\end{align*}
Finally since the above inequality holds for all $a \geq 0$ and $\tilde{b}\geq 0$, thus for any given constant $c>0$, setting $a=\frac{1+c}{C_1}, \tilde{b}=\sqrt{\frac{\log(p^c+1)}{C_5\log p}}$ establishes the desired result.
\end{proof}

\subsection{Proof of Theorem \ref{thm:one} and Proposition \ref{general:debiasing}}
\label{proof:thm:one:index}

\begin{proof}
Given that Theorem \ref{thm:one} is a specialized result of Proposition \ref{general:debiasing}, we will only present the proof for Theorem \ref{thm:one}. The proof of Proposition \ref{general:debiasing} follows directly by replacing $\tilde{\balpha},\tilde{\theta}, \hat{\bmu}, \hat{\bSigma}$ with $\hat{\bvarpi},\hat{\vartheta},\hat{\bxi},\hat{\bUpsilon}$ respectively. Throughout the proof, $C_1, C_2, \ldots$ are used to denote positive constants that possibly depend on $\nu, c_L, c_U$. Some of them may depend on additional quantities, and clarification will be made in such cases.

Lemma \ref{l2error:alpha} combined with the fact that $\ltwonorm{\balpha}\leq \gamma, \ltwonorm{\tilde{\balpha}}\leq \gamma$ shows that
\begin{align}
\mathbb{E}\|\tilde{\balpha}-\balpha\|_2^2 &= \mathbb{E}\big(\|\tilde{\balpha}-\balpha\|_2^2\mathbbm{1}_{\mathcal{A}(\lambda)\cap \mathcal{B}(s, \kappa)} \big) + \mathbb{E}\big(\|\tilde{\balpha}-\balpha\|_2^2\mathbbm{1}_{\mathcal{A}^c(\lambda)\cup \mathcal{B}^c(s, \kappa)} \big)  \nonumber \\
&\leq \frac{9t^2\nu^2(1+\tau) s\log p}{n(\lambda^{1/2}_{\min}(\bSigma)-\kappa)^4}+\frac{16\tau}{c_L}(\mathbb{P}(\mathcal{A}^c(\lambda))+\mathbb{P}(\mathcal{B}^c(s, \kappa))). \label{alpha:bound1}
\end{align}
According to Lemma \ref{rare:events} Part (i), it is straightforward to obtain the following bound,
\begin{align}
\label{alpha:bound2}
\mathbb{P}(\mathcal{A}^c(\lambda)) \leq 8p(e^{-C_1t^2\log p}+e^{-C_2t\sqrt{n\log p}}) \leq 8p^{1-C_3t},
\end{align}
where we have used the condition $\frac{s\log p}{n}<1$ and $t>1$. Moreover, since $\ltwonorm{\bu}\leq 1, \lonenorm{\bu}\leq 4\sqrt{s}$ for $\bu \in \cK(s)\cap B_2(1)$, Lemma \ref{key:concentration:eq} implies that for $\forall c>0$, as long as $\frac{s\log p}{n}\leq \frac{c_L}{4c_1^2\nu^4c_U}$,
\begin{align}
\label{alpha:bound3}
\PP(\cB^c(s, \kappa))\leq 4p^{-c},
\end{align}
with $\kappa=C_4\sqrt{\frac{s\log p}{n}}$. Here, $C_4>0$ is a constant depending on $c$ (in addition to $\nu, c_L, c_U$), and $C_4 \rightarrow \infty$, as $c\rightarrow \infty$. Observe that the constant $\tilde{c}$ in Theorem \ref{thm:one} can be chosen as $\min(\frac{c_L}{4c_1^2\nu^4c_U}, \frac{1}{2}, \frac{c_L}{4C_4^2})$, thus $\lambda^{1/2}_{\min}(\Sigma)-\kappa \geq \frac{\sqrt{c_L}}{2}$. Therefore, putting together the results \eqref{alpha:bound1}, \eqref{alpha:bound2} and \eqref{alpha:bound3} establishes the desired bound for $\mathbb{E}\|\tilde{\balpha}-\balpha\|_2^2$.

We now bound $\mathbb{E}|\tilde{\theta}-\theta|$. Set the same value for $\kappa$ as in the preceding proof. According to Lemma \ref{error:functional} we obtain
\begin{align}
\label{theta:bound:one}
\mathbb{E}|(\tilde{\theta}-\theta)\mathbbm{1}_{\mathcal{A}(\lambda)\cap \mathcal{B}(s,\kappa)}| &\leq \frac{C_5t^2(1+\tau)s\log p}{n} +2 \mathbb{E}|\balpha^T(\hat{\bmu}-\bmu)| + \mathbb{E}|\balpha^T(\hat{\bSigma}-\bSigma)\balpha|,
\end{align}
Moreover, observe that
\begin{equation}
\label{theta:bound:two}
\mathbb{E}|\balpha^T(\hat{\bmu}-\bmu)| \leq \sqrt{\var(\balpha^T(\hat{\bmu}-\bmu))} =\sqrt{\frac{\balpha^T\bSigma \balpha}{n}}\leq \sqrt{\frac{\tau}{n}},
\end{equation}
and
\begin{align}
 \mathbb{E}|\balpha^T(\hat{\bSigma}-\bSigma)\balpha| &\leq  \mathbb{E} \Big|\frac{1}{n}\sum_{i=1}^n|\balpha^T(\bx_i-\bmu)|^2-\balpha^T\bSigma \balpha \Big| +\mathbb{E} |\balpha^T(\hat{\bmu}-\bmu)|^2 \nonumber \\
&\leq \sqrt{\frac{\mathbb{E}|\balpha^T(\bx_1-\bmu)|^4}{n}}+\frac{\tau}{n} \leq C_6 \frac{\tau \nu^2}{\sqrt{n}}+\frac{\tau}{n}, \label{theta:bound:three}
\end{align}
where the last inequality is due to $\|\balpha^T(\bx_1-\bmu)\|_{\psi_2}\leq \sqrt{\tau}\nu$ and the moments inequality for sub-gaussian variables (cf. Proposition 2.5.2 in \cite{versh18}).

The rest of the proof is to bound $\mathbb{E}|(\tilde{\theta}-\theta)\mathbbm{1}_{\mathcal{A}^c(\lambda)\cup \mathcal{B}^c(s, \kappa)}|$. Towards that goal, we first bound $\mathbb{E}|\tilde{\theta}|^2$. By the definition of $\tilde{\balpha}$,
\begin{equation*}
\frac{1}{2}\tilde{\balpha}^T\hat{\bSigma}\tilde{\balpha}-\tilde{\balpha}^T\hat{\bmu}+\lambda \lonenorm{\tilde{\balpha}}\leq 0,
\end{equation*}
yielding $\tilde{\theta}\geq 2\lambda \|\tilde{\balpha}\|_1\geq 0$. Hence,
\begin{align}
\mathbb{E}|\tilde{\theta}|^2&\leq  4 \mathbb{E}|\tilde{\balpha}^T\hat{\bmu}|^2 \leq 12\mathbb{E}|\bmu^T \tilde{\balpha}|^2+12\mathbb{E}|(\hat{\bmu}-\bmu)^T\balpha|^2+12\mathbb{E}|(\hat{\bmu}-\bmu)^T(\tilde{\balpha}-\balpha)|^2 \nonumber \\
&\overset{(a)}{\leq}  \frac{48p c_U \tau^2}{c_L} + \frac{12\tau}{n}+12 \mathbb{E}(\|\hat{\bmu}-\bmu\|^2_{\infty} \cdot \|\tilde{\balpha}-\balpha\|^2_1), \nonumber \\
&\overset{(b)}{\leq}  \frac{48pc_U \tau^2}{c_L} + \frac{12\tau}{n} + 12b^2\mathbb{E}\|\tilde{\balpha}-\balpha\|^2_1+\frac{192 p\tau}{c_L}\mathbb{E}(\|\hat{\bmu}-\bmu\|^2_{\infty}\mathbbm{1}_{\|\hat{\bmu}-\bmu\|_{\infty}>b}), \quad \forall b \geq 0. \label{theta:square:bound}
\end{align}
Here, in $(a)$ we have used that $|\bmu^T \tilde{\balpha}|^2\leq \ltwonorm{\bmu}^2\cdot \ltwonorm{\tilde{\balpha}}^2\leq \gamma^2\ltwonorm{\bmu}^2\leq 4pc_Uc_L^{-1}\tau^2$; $(b)$ is due to $\lonenorm{\tilde{\balpha}-\balpha}^2\leq p\ltwonorm{\tilde{\balpha}-\balpha}^2\leq 4p\gamma^2$. Furthermore, we use Lemma \ref{l2error:alpha} and the upper bound for $\mathbb{E}\|\tilde{\balpha}-\balpha\|_2^2$ to obtain
\begin{align}
\mathbb{E}\|\tilde{\balpha}-\balpha\|^2_1 &= \mathbb{E}(\|\tilde{\balpha}-\balpha\|^2_1 \mathbbm{1}_{\mathcal{A}(\lambda)})+ \mathbb{E}(\|\tilde{\balpha}-\balpha\|^2_1 \mathbbm{1}_{\mathcal{A}^c(\lambda)})  \nonumber \\
&\leq 16 s \mathbb{E}\|\tilde{\balpha}-\balpha\|_2^2+  \frac{16 p \tau}{c_L}\mathbb{P}(\mathcal{A}^c(\lambda)) \nonumber \\
&\leq C_7 \cdot s\Big(\frac{t^2(1+\tau)s\log p}{n}+\tau p^{-(C_3t-2)\wedge c} \Big) \label{l1error:alpha}
\end{align}
According to Lemma \ref{rare:events} Part (ii), it is possible to set $b=C_8\sqrt{\frac{\log p}{n}}$ in \eqref{theta:square:bound} to have
\begin{equation}
\label{muhat:error}
\mathbb{E}(\|\hat{\bmu}-\bmu\|^2_{\infty}\mathbbm{1}_{\|\hat{\bmu}-\bmu\|_{\infty}>b}) \leq \frac{C_{9}}{np}.
\end{equation}
Based on the results \eqref{theta:square:bound}, \eqref{l1error:alpha} and \eqref{muhat:error}, we can derive an upper bound for $\mathbb{E}|\tilde{\theta}|^2$,
\begin{equation*}
\mathbb{E}|\tilde{\theta}|^2 \leq C_{10}\cdot \Big(p\tau^2+\frac{\tau}{n}+\frac{t^4(1+\tau)^2s^2\log^2 p}{n^2}+\tau p^{-(C_3t-2)\wedge c}\Big).
\end{equation*}
As a result, we are able to conclude
\begin{align*}
&\mathbb{E}|(\tilde{\theta}-\theta)\mathbbm{1}_{\mathcal{A}^c(\lambda)\cup \mathcal{B}^c(s, \kappa)}| \leq  \sqrt{2(\mathbb{E}|\tilde{\theta}|^2+\mathbb{E}|\theta|^2)}\cdot \sqrt{\mathbb{P}(\mathcal{A}^c(\lambda))+\mathbb{P}(\mathcal{B}^c(s, \kappa))} \\
 \leq &C_{11} \cdot \Big(\sqrt{p}\tau+\sqrt{\frac{\tau}{n}}+\frac{t^2(1+\tau)s\log p}{n}+\sqrt{\tau}p^{-\frac{(C_3t-2)\wedge c}{2}}\Big) \cdot p^{-\frac{(C_3t-1)\wedge c}{2}} \\
 \leq & C_{11} \cdot \Big(\sqrt{\frac{\tau}{n}}+\frac{t^2(1+\tau)s\log p}{n}+(\tau+\sqrt{\tau})p^{-\frac{(C_3t-2)\wedge (c-1)}{2}}  \Big),
\end{align*}
where in the last step we have used the condition $C_3t-2>0$. The above result combined with \eqref{theta:bound:one}, \eqref{theta:bound:two}, and \eqref{theta:bound:three} yields the desired bound for $\EE|\tilde{\theta}-\theta|$.
\end{proof}

\begin{lem}
\label{l2error:alpha}
Denote $s=\lzeronorm{\balpha}$, and set $\gamma$ in \eqref{lasso:alpha} such that $\ltwonorm{\balpha}\leq \gamma$. It holds that
\begin{itemize}
\item[(i)] On the event $\mathcal{A}(\lambda)$, $\tilde{\balpha}-\balpha\in \cK(s)$.
\item[(ii)] On the event $\cA(\lambda) \cap \cB(s, \kappa)$ with $\kappa < \lambda^{1/2}_{\min}(\bSigma)$,
\begin{equation}
\label{l2bound:alpha}
\ltwonorm{\tilde{\balpha}-\balpha}  \leq \frac{3 \lambda \sqrt{s}}{(\lambda^{1/2}_{\min}(\bSigma)-\kappa)^2}.
\end{equation}
\end{itemize}
\end{lem}

The above results might be obtained by the general analysis framework for high-dimensional M-estimator developed in \cite{nrwy12}. For completeness, we give a proof tailored for our problem.

\begin{proof}
Let the support of $\balpha$ be indexed by the set $S \subseteq [p]$. Since $\bmu=\bSigma \balpha$, it is clear that on $\mathcal{A}(\lambda)$,
\begin{equation}
\label{basic:eq:one}
\linfnorm{\hat{\bSigma}\balpha-\hat{\bmu}} \leq \linfnorm{(\hat{\bSigma}-\bSigma)\balpha}+\linfnorm{\hat{\bmu}-\bmu} \leq \frac{\lambda}{2}.
\end{equation}
Then by the definition of $\tilde{\balpha}$ in \eqref{lasso:alpha}, we have on the event $\mathcal{A}(\lambda)$:
\begin{align}
0& \geq \frac{1}{2}\tilde{\balpha}^T \hat{\bSigma}\tilde{\balpha}-\tilde{\balpha}^T \hat{\bmu}-\frac{1}{2}\balpha^T \hat{\bSigma}\balpha+\balpha^T \hat{\bmu}+\lambda(\|\tilde{\balpha}\|_1-\|\balpha\|_1)  \nonumber \\
&= \frac{1}{2}(\tilde{\balpha}-\balpha)^T\hat{\bSigma}(\tilde{\balpha}-\balpha)+ (\hat{\bSigma}\balpha-\hat{\bmu})^T(\tilde{\balpha}-\balpha)+\lambda(\|\tilde{\balpha}\|_1-\|\balpha\|_1) \nonumber \\
&\geq \frac{1}{2}(\tilde{\balpha}-\balpha)^T\hat{\bSigma}(\tilde{\balpha}-\balpha)-\|\hat{\bSigma}\balpha-\hat{\bmu}\|_{\infty} \cdot \|\tilde{\balpha}-\balpha\|_1+\lambda(\|\tilde{\balpha}\|_1-\|\balpha\|_1)  \nonumber. \\
&\overset{(a)}{\geq}  \frac{1}{2}(\tilde{\balpha}-\balpha)^T\hat{\bSigma}(\tilde{\balpha}-\balpha)+\frac{\lambda}{2}(-\|\tilde{\balpha}-\balpha\|_1+2\|\tilde{\balpha}\|_1-2\|\balpha\|_1) \nonumber \\
&= \frac{1}{2}(\tilde{\balpha}-\balpha)^T\hat{\bSigma}(\tilde{\balpha}-\balpha)+\frac{\lambda}{2}\Big[\|\tilde{\balpha}_{S^c}\|_1-\|\tilde{\balpha}_S-\balpha_S\|_1+2(\|\tilde{\balpha}_S\|_1-\|\balpha_S\|_1)\Big] \nonumber \\
&\overset{(b)}{\geq}  \frac{1}{2}(\tilde{\balpha}-\balpha)^T\hat{\bSigma}(\tilde{\balpha}-\balpha)+\frac{\lambda}{2}(\|\tilde{\balpha}_{S^c}\|_1-3\|\tilde{\balpha}_S-\balpha_S\|_1) \label{primal:dev} \\
&\geq \frac{\lambda}{2}(\|\tilde{\balpha}_{S^c}\|_1-3\|\tilde{\balpha}_S-\balpha_S\|_1),  \nonumber
\end{align}
where $(a)$ holds by \eqref{basic:eq:one}, and $(b)$ is due to $\lonenorm{\balpha_S} \leq \lonenorm{\tilde{\balpha}_S-\balpha_S}+\lonenorm{\tilde{\balpha}_S}$. We thus obtain $\lonenorm{\tilde{\balpha}_{S^c}} \leq 3 \lonenorm{\tilde{\balpha}_S-\balpha_S}$, i.e., $\tilde{\balpha}-\balpha \in \cK(s)$.

To prove the second part, using the fact that $\tilde{\balpha}-\balpha \in \cK(s)$ on $\cA(\lambda)$, we can conclude that on the event $\mathcal{A}(\lambda)\cap \mathcal{B}(s, \kappa)$,
\begin{align}
\sqrt{(\tilde{\balpha}-\balpha)^T\hat{\bSigma}(\tilde{\balpha}-\balpha)} &\geq  \ltwonorm{\tilde{\balpha}-\balpha} \cdot \Big(\min_{\bu\in B_2(1) }\sqrt{\bu^T\bSigma\bu}-\max_{\bu\in \mathcal{K}(s)\cap B_2(1) }\big |\sqrt{\bu^T \bSigma \bu} - \sqrt{\bu^T \hat{\bSigma} \bu} \big|\Big) \nonumber \\
& \geq  \ltwonorm{\tilde{\balpha}-\balpha} \cdot (\lambda^{1/2}_{\min}(\bSigma)-\kappa)   \label{basic:eq:two}
\end{align}
Therefore, on the event $\mathcal{A}(\lambda)\cap \mathcal{B}(s, \kappa)$, we can continue from \eqref{primal:dev} to obtain that when $\kappa<\lambda^{1/2}_{\min}(\bSigma)$,
\begin{align*}
0&\geq\frac{1}{2} (\lambda^{1/2}_{\min}(\bSigma)-\kappa)^2 \ltwonorm{\tilde{\balpha}-\balpha}^2+\frac{\lambda}{2}(\|\tilde{\balpha}_{S^c}\|_1-3\|\tilde{\balpha}_S-\balpha_S\|_1)\\
&\overset{(c)}{\geq} \frac{1}{2} (\lambda^{1/2}_{\min}(\bSigma)-\kappa)^2 \ltwonorm{\tilde{\balpha}-\balpha}^2-\frac{3\lambda\sqrt{s}}{2}\|\tilde{\balpha}-\balpha\|_2 \\
&= \frac{\|\tilde{\balpha}-\balpha\|_2}{2}\Big[(\lambda^{1/2}_{\min}(\bSigma)-\kappa)^2\|\tilde{\balpha}-\balpha\|_2-3\lambda \sqrt{s}\Big],
\end{align*}
where $(c)$ is simply by Cauchy-Schwarz inequality $\lonenorm{\tilde{\balpha}_S-\balpha_S}\leq \sqrt{|S|}\cdot \ltwonorm{\tilde{\balpha}_S-\balpha_S}\leq \sqrt{s}\ltwonorm{\tilde{\balpha}-\balpha}$. The upper bound on $\|\tilde{\balpha}-\balpha\|_2$ follows.
\end{proof}

\begin{lem}\label{error:functional}
Denote $s=\lzeronorm{\balpha}$, and set $\gamma$ in \eqref{lasso:alpha} such that $\ltwonorm{\balpha}\leq \gamma$. On the event $\mathcal{A}(\lambda) \cap \mathcal{B}(s, \kappa)$ with $\kappa<\lambda^{1/2}_{\min}(\bSigma)$, it holds that
\[
|\tilde{\theta}-\theta| \leq \frac{24s\lambda^2(5\lambda_{\min}(\bSigma)+8\kappa^2)}{(\lambda_{\min}^{1/2}(\bSigma)-\kappa)^4}+2 |\balpha^T(\hat{\bmu}-\bmu)| + |\balpha^T(\hat{\bSigma}-\bSigma)\balpha|.
\]
\end{lem}

\begin{proof}
Let the support of $\balpha$ be indexed by the set $S \subseteq [p]$. Define the function $f: \RR^p \rightarrow \RR$ as $f(\bx)=2\bmu^T\bx -\bx^T\bSigma \bx$. It is straightforward to verify that $f(\balpha)=\theta$, and $\bx=\balpha$ is the global maximizer of $f(\bx)$. Therefore,
\begin{equation}\label{simple:fact}
\theta =f(\balpha) \geq f(\tilde{\balpha})=2\bmu^T \tilde{\balpha}-\tilde{\balpha}^T\bSigma \tilde{\balpha},
\end{equation}
from which we can proceed to derive an upper bound on $\tilde{\theta}-\theta$,
\begin{align}\label{first:upper}
\tilde{\theta}-\theta&= (2\hat{\bmu}^T\tilde{\balpha}-\tilde{\balpha}^T\hat{\bSigma} \tilde{\balpha}) -f(\balpha) \leq (2\hat{\bmu}^T\tilde{\balpha}-\tilde{\balpha}^T\hat{\bSigma} \tilde{\balpha}) -f(\tilde{\balpha}) \nonumber \\
&=(2\hat{\bmu}^T\tilde{\balpha}-\tilde{\balpha}^T\hat{\bSigma} \tilde{\balpha})- (2\bmu^T \tilde{\balpha}-\tilde{\balpha}^T\bSigma \tilde{\balpha}) \nonumber \\
&\leq -2\balpha^T(\bSigma-\hat{\bSigma})(\balpha-\tilde{\balpha})+(\balpha-\tilde{\balpha})^T\bSigma(\balpha-\tilde{\balpha})-\balpha^T(\hat{\bSigma}-\bSigma)\balpha  \nonumber \\
&~~~-2(\tilde{\balpha}-\balpha)^T(\bmu-\hat{\bmu})-2\balpha^T(\bmu-\hat{\bmu})   \nonumber \\
&\leq 2\|(\hat{\bSigma}-\bSigma)\balpha\|_{\infty}\cdot \|\tilde{\balpha}-\balpha\|_1+(\balpha-\tilde{\balpha})^T\bSigma(\balpha-\tilde{\balpha})+|\balpha^T(\hat{\bSigma}-\bSigma)\balpha|    \nonumber \\
&~~~+ 2\|\hat{\bmu}-\bmu\|_{\infty}\cdot \|\tilde{\balpha}-\balpha\|_1 +2|\balpha^T(\hat{\bmu}-\bmu)|.
\end{align}
Moreover, \eqref{primal:dev} implies that on $\mathcal{A}(\lambda)$,
\[
(\balpha-\tilde{\balpha})^T\hat{\bSigma}(\balpha-\tilde{\balpha}) \leq 3\lambda\|\tilde{\balpha}_S-\balpha_S\|_1\leq 3\lambda \sqrt{s}\ltwonorm{\tilde{\balpha}-\balpha},
\]
which further yields that on $\mathcal{A}(\lambda) \cap \mathcal{B}(s, \kappa)$,
\begin{align}
\label{key:piece:relax}
(\balpha-\tilde{\balpha})^T\bSigma (\balpha-\tilde{\balpha})& \leq 2(\balpha-\tilde{\balpha})^T\hat{\bSigma}(\balpha-\tilde{\balpha}) +2\ltwonorm{\balpha-\tilde{\balpha}}^2\cdot \max_{\bu\in \mathcal{K}(s)\cap B_2(1) }\big |\sqrt{\bu^T \bSigma \bu} - \sqrt{\bu^T \hat{\bSigma} \bu} \big|^2 \nonumber \\
&\leq 6\lambda \sqrt{s}\ltwonorm{\tilde{\balpha}-\balpha}+2\kappa^2\ltwonorm{\balpha-\tilde{\balpha}}^2
\end{align}

Since $\|\tilde{\balpha}_{S^c}\|_1\leq 3\|\tilde{\balpha}_S-\balpha_S\|_1$ on event $\mathcal{A}(\lambda)$ by Lemma \ref{l2error:alpha}, it holds that
\begin{equation}\label{basic:eq:three}
\|\tilde{\balpha}-\balpha\|_1\leq 4 \|\tilde{\balpha}_S-\balpha_S\|_1\leq 4\sqrt{s}\|\tilde{\balpha}-\balpha\|_2.
\end{equation}
Hence on the event $\mathcal{A}(\lambda) \cap \mathcal{B}(s, \kappa)$, putting together the results \eqref{l2bound:alpha}, \eqref{first:upper}, \eqref{key:piece:relax} and \eqref{basic:eq:three} gives us that
\begin{align}
\tilde{\theta}-\theta &\leq 2\kappa^2 \|\tilde{\balpha}-\balpha\|^2_2+10\sqrt{s}\lambda \|\tilde{\balpha}-\balpha\|_2+  |\balpha^T(\hat{\bSigma}-\bSigma)\balpha|+2 |\balpha^T(\hat{\bmu}-\bmu)|. \nonumber \\
&\leq \frac{6s\lambda^2(5\lambda_{\min}(\bSigma)+8\kappa^2)}{(\lambda^{1/2}_{\min}(\bSigma)-\kappa)^4}+ |\balpha^T(\hat{\bSigma}-\bSigma)\balpha|+ 2 |\balpha^T(\hat{\bmu}-\bmu)|. \label{upper:bound:alphal2}
\end{align}
We now turn to lower bounding $\tilde{\theta}-\theta$. On the event $\mathcal{A}(\lambda)$,
\begin{align}
\tilde{\theta}-\theta &= (2\hat{\bmu}^T\tilde{\balpha}-\tilde{\balpha}^T\hat{\bSigma} \tilde{\balpha}) -(2\bmu^T\balpha-\balpha^T\bSigma \balpha) \nonumber \\
&\overset{(a)}{\geq} (2\hat{\bmu}^T\balpha-\balpha^T\hat{\bSigma} \balpha-2\lambda \|\balpha\|_1+2\lambda \|\tilde{\balpha}\|_1)-(2\bmu^T\balpha-\balpha^T\bSigma \balpha) \nonumber \\
&= -\balpha^T(\hat{\bSigma}-\bSigma)\balpha-2\balpha^T(\bmu-\hat{\bmu}) +2\lambda (\|\tilde{\balpha}\|_1-\|\balpha\|_1) \nonumber \\
&\geq -|\balpha^T(\hat{\bSigma}-\bSigma)\balpha|-2|\balpha^T(\bmu-\hat{\bmu})|-2\lambda \|\tilde{\balpha}-\balpha\|_1 \label{new:back:look} \\
&\overset{(b)}{\geq}-8\lambda \sqrt{s}\ltwonorm{\tilde{\balpha}-\balpha}-|\balpha^T(\hat{\bSigma}-\bSigma)\balpha|-2 |\balpha^T(\hat{\bmu}-\bmu)|   \nonumber \\
&\overset{(c)}{\geq} \frac{-24s\lambda^2}{(\lambda^{1/2}_{\min}(\bSigma)-\kappa)^2}-|\balpha^T(\hat{\bSigma}-\bSigma)\balpha|-2 |\balpha^T(\hat{\bmu}-\bmu)|. \label{lower:bound:alphal2}
\end{align}
Here, $(a)$ holds by the definition of $\tilde{\balpha}$; $(b)$ is by \eqref{basic:eq:three}; and $(c)$ is due to \eqref{l2bound:alpha}. Combining the upper bound \eqref{upper:bound:alphal2} and lower bound \eqref{lower:bound:alphal2} for $\tilde{\theta}-\theta$ completes the proof.
\end{proof}

\subsection{Proof of Theorem \ref{lower:bound} and Proposition \ref{prop:one}}
\label{proof:thm:two:index}

Proposition \ref{prop:one} is a simple by-product of Theorem \ref{lower:bound}. In the following, we present the proof for Theorem \ref{lower:bound}. The proof of Proposition \ref{prop:one} will be appropriately mentioned at some point in the proof.

\vspace{0.4cm}

\noindent To obtain the lower bounds, it is sufficient to consider Gaussian distributions and covariance matrices with bounded maximum eigenvalue $\lambda_{\max}(\bSigma)\leq c_U$. In this section, a given pair $(\bmu, \bSigma)$ that belongs to $\cH(s,\tau)$ represents a Gaussian distribution with mean $\bmu$ and covariance matrix $\bSigma$. Let $\mathbb{P}^n_{(\bmu,\bSigma)}$ be the joint distribution of n i.i.d samples from $N(\bmu,\bSigma)$. We use $C_1,C_2,\ldots$ to denote positive constants possibly depending on $c_L,c_U$.

\subsubsection{Lower bound for $\mathbb{E}|\hat{\theta}-\theta|$} \label{lowerbound:theta}

We first derive the lower bound $\tau \wedge \frac{\tau + \sqrt{\tau}}{\sqrt{n}}$. This can be obtained by reducing the estimation problem to a problem of testing between two distributions $N(\bmu, \bSigma)$ and $N(\tilde{\bmu},\tilde{\bSigma})$. Specifically, according to Theorem 2.2 (iii) in \cite{tsybakov09}, a lower bound
\[
\inf_{\hat{\theta}}\sup_{(\bmu,\bSigma)\in \mathcal{H}(s,\tau)}\mathbb{E}|\hat{\theta}-\theta| \geq \frac{\zeta}{4}e^{-\eta}
\]
holds if the followings are true:
\begin{itemize}
\item[(1)] $(\bmu,\bSigma), (\tilde{\bmu},\tilde{\bSigma})\in \mathcal{H}(s, \tau)$.
\item[(2)] $D_{KL}\Big(\mathbb{P}^n_{(\bmu,\bSigma)} \| \mathbb{P}^n_{(\tilde{\bmu},\tilde{\bSigma})}\Big) \leq \eta $, where $D_{KL}(\cdot\|\cdot)$ is the KL-divergence.
\item[(3)] $|\bmu^T\bSigma^{-1}\bmu-\tilde{\bmu}^T\tilde{\bSigma}^{-1}\tilde{\bmu}| \geq 2\zeta$.
\end{itemize}
Depending on the scaling of $\tau$, we construct different testing problems.
\begin{itemize}
\item[(i)] $\tau \geq 1$. We consider $\bmu=\tilde{\bmu}=(\sqrt{\tau c_L},0,\ldots, 0)^T, \bSigma=\diag(c_L,c_L,\ldots,c_L), \tilde{\bSigma}=\diag(c_L+\frac{c_U-c_L}{\sqrt{n}},c_L,\ldots,c_L)$. We verify (1)-(3) one by one. Part (1) is obvious. For Part (3), it is straightforward to see
\begin{equation*}
|\bmu^T\bSigma^{-1}\bmu-\tilde{\bmu}^T\tilde{\bSigma}^{-1}\tilde{\bmu}|=\frac{\tau (c_U-c_L)}{(\sqrt{n}-1)c_L+c_U} \geq \frac{\tau(c_U-c_L)}{\sqrt{n}(c_U+c_L)}.
\end{equation*}
Regarding Part (2), we have
\begin{align*}
&D_{KL}\Big(\mathbb{P}^n_{(\bmu,\bSigma)} \| \mathbb{P}^n_{(\tilde{\bmu},\tilde{\bSigma})}\Big)=\frac{n}{2}\big[\log |\tilde{\bSigma}|-\log|\bSigma|-p+\mbox{tr}(\bSigma\tilde{\bSigma}^{-1})\big] \\
&=\frac{n}{2}\Big(\log \big(1+\frac{c_U/c_L-1}{\sqrt{n}}\big)-\frac{c_U/c_L-1}{\sqrt{n}+c_U/c_L-1} \Big) \\
&\leq \frac{n}{2}\Big(\frac{c_U/c_L-1}{\sqrt{n}}+\frac{(c_U/c_L-1)^2}{2n}-\frac{c_U/c_L-1}{\sqrt{n}+c_U/c_L-1}  \Big) \leq \frac{(c_U-c_L)^2}{c^2_L},
\end{align*}
where we have used the fact that $\log(1+z)\leq z+\frac{z^2}{2}$ for $z\geq 0$. Above all, we obtain the lower bound
\[
\inf_{\hat{\theta}}\sup_{(\bmu,\bSigma)\in \mathcal{H}(s,\tau)}\mathbb{E}|\hat{\theta}-\theta| \geq  C_1 \frac{\tau}{\sqrt{n}}.
\]

\item[(ii)] $\frac{1}{n}\leq \tau \leq 1$. We set up the parameters: $\bmu=(\sqrt{\frac{\tau c_L}{2}}, 0,\ldots, 0)^T, \tilde{\bmu}=(\sqrt{\frac{\tau c_L}{2}}+\sqrt{\frac{c_U-c_L}{2n}},0,\ldots,0)^T, \bSigma=\tilde{\bSigma}=\diag(c_U,\ldots,c_U)$. Since $\frac{1}{n}\leq \tau$,
\[
\tilde{\bmu}^T\tilde{\bSigma}^{-1}\tilde{\bmu} \leq c^{-1}_U(\tau c_L+\frac{c_U-c_L}{n})\leq \tau.
\]
Then Part (1) follows directly. For Part (2),
\[
D_{KL}\Big(\mathbb{P}^n_{(\bmu,\bSigma)} \| \mathbb{P}^n_{(\tilde{\bmu},\tilde{\bSigma})} \Big)=\frac{n\|\bmu-\tilde{\bmu}\|_2^2}{2c_U}=\frac{c_U-c_L}{4c_U}.
\]
Finally, $|\bmu^T\bSigma^{-1}\bmu-\tilde{\bmu}^T\tilde{\bSigma}^{-1}\tilde{\bmu}| \geq  \sqrt{\frac{c_L(c_U-c_L)\tau}{c^2_U n}}$. We thus have the bound
\[
\inf_{\hat{\theta}}\sup_{(\bmu,\bSigma)\in \mathcal{H}(s,\tau)}\mathbb{E}|\hat{\theta}-\theta| \geq C_2 \sqrt{\frac{\tau}{n}}.
\]

\item[(iii)] $\tau \leq \frac{1}{n}$. We consider the parameters $\bmu=(\frac{\sqrt{\tau c_L}}{2},0,\ldots, 0)^T, \tilde{\bmu}=(\sqrt{\tau c_L},0,\ldots, 0)^T, \bSigma=\tilde{\bSigma}=\diag(c_L,\ldots,c_L)$. Part (1) clearly holds. Regarding Part (2), since $\tau  \leq \frac{1}{n}$
\[
D_{KL}(\mathbb{P}^n_{(\bmu,\bSigma)} \| \mathbb{P}^n_{(\tilde{\bmu},\tilde{\bSigma})} )=\frac{n\|\bmu-\tilde{\bmu}\|_2^2}{2c_L}=\frac{n\tau}{8}\leq 1.
\]
For Part (3), we have $|\bmu^T\bSigma^{-1}\bmu-\tilde{\bmu}^T\tilde{\bSigma}^{-1}\tilde{\bmu}|=\frac{3\tau}{4}$. Hence,
\[
\inf_{\hat{\theta}}\sup_{(\bmu,\bSigma)\in \mathcal{H}(s,\tau)}\mathbb{E}|\hat{\theta}-\theta| \geq C_3 \tau.
\]
\end{itemize}

\noindent We proceed to obtain the other lower bound $\big[\tau \wedge \frac{(\tau+1)s\log p}{n}\big]c_0 \exp(-e^{2s^2p^{c_6 c_0-1}})$. We apply the method of two fuzzy hypotheses, a generalization of the two-point testing technique. The method has been used to obtain minimax lower bound for functional estimation problems \citep{frw15}. Denote $\mathcal{S}=\{S\subseteq [p-1]: |S|=s-1\}$. We reduce the estimation problem to a testing between $\mathbb{P}^n_{(\bmu^0,\bSigma^0)}$ and $\frac{1}{|\mathcal{S}|}\sum_{S\in \mathcal{S}}\mathbb{P}^n_{(\bmu^S,\bSigma^S)}$, where the parameters $(\bmu^0, \bSigma^0), \{(\bmu^S,\bSigma^S)\}_{S\in \cS}$ will be constructed adaptively depending on the scaling of $\tau$. According to Theorem 2.15 in \cite{tsybakov09}, if we are able to show the following:
\begin{itemize}
\item[(1)] $(\bmu^0,\bSigma^0),(\bmu^S,\bSigma^S) \in \mathcal{H}(s,\tau), \forall S\in \mathcal{S}$,
\item[(2)] $\chi^2\Big(\frac{1}{|\mathcal{S}|}\sum_{S\in \mathcal{S}}\mathbb{P}^n_{(\bmu^S,\bSigma^S)},\mathbb{P}^n_{(\bmu^0,\bSigma^0)} \Big) \leq \eta$, where $\chi^2(\cdot,\cdot)$ is the $\chi^2$-divergence,
\item[(3)] $|(\bmu^0)^T(\bSigma^0)^{-1}\bmu^0-(\bmu^S)^T(\bSigma^S)^{-1}\bmu^S| \geq 2\zeta, \forall S \in \mathcal{S}$,
\end{itemize}
then
\[
\inf_{\hat{\theta}}\sup_{(\bmu, \bSigma)\in \mathcal{H}(s,\tau)}\mathbb{E}|\hat{\theta}-\theta| \geq \frac{\zeta}{4}e^{-\eta}.
\]
We perform the analyses for three different cases, respectively.

\begin{itemize}
\item[(i)] $\tau \geq 1$. Let $\bone \in\mathbb{R}^{p-1}$ be the vector with all the entries equal to 1. Set $\bmu^0=\bmu^S=(\sqrt{\frac{(c_L+c_U)\tau}{4}},0,\ldots,0)^T, \bSigma^0=\frac{c_L+c_U}{2}\bI_p$, and for $S\in \cS$
\[
\bSigma^S=\frac{c_L+c_U}{2}\cdot
\begin{bmatrix}
1& a\bone^T_S\\
a\bone_S&I_{p-1}
\end{bmatrix}
, ~~a=\frac{(c_U-c_L)\rho}{(c_U+c_L)\sqrt{2(s-1)}},~0\leq \rho \leq 1,
\]
where the value of $\rho$ will be specified later. For Part (1), first note that $(\bmu^0,\bSigma^0)\in \mathcal{H}(s,\tau)$ is trivially true. Moreover, by blockwise matrix inversion we obtain
\[
(\bSigma^S)^{-1}=\frac{2}{c_L+c_U}\bI_p+\frac{2}{c_L+c_U}\cdot
\begin{bmatrix}
\frac{a^2(s-1)}{1-a^2(s-1)}& \frac{-a}{1-a^2(s-1)}\bone^T_S \\
\frac{-a}{1-a^2(s-1)} \bone_S & \frac{a^2}{1-a^2(s-1)}\bone_S\bone^T_S
\end{bmatrix}.
\]
Hence,
\[
\|(\bSigma^S)^{-1}\bmu^S\|_0\leq s, \quad (\bmu^S)^T(\bSigma^S)^{-1}\bmu^S=\frac{\tau}{2-\frac{(c_U-c_L)^2\rho^2}{(c_L+c_U)^2}}\leq \tau.
\]
To prove $(\bmu^S,\bSigma^S)\in \mathcal{H}(s,\tau)$, it remains to show $c_L \leq \lambda_{\min}(\bSigma^S)\leq \lambda_{\max}(\bSigma^S)\leq c_U$. This holds because by Weyl's inequality,
\begin{align*}
&\max(|\lambda_{\max}(\bSigma^S)-\frac{c_L+c_U}{2}|, |\lambda_{\min}(\bSigma^S)-\frac{c_L+c_U}{2}|)  \\
&\leq \|\bSigma^S-\bSigma^0\|_F=\frac{(c_U-c_L)\rho}{2}\leq \frac{c_U-c_L}{2}.
\end{align*}
Regarding Part (3), a straightforward calculation delivers
\[
|(\bmu^0)^T(\bSigma^0)^{-1}\bmu^0-(\bmu^S)^T(\bSigma^S)^{-1}\bmu^S| \geq \frac{(c_U-c_L)^2\rho^2\tau}{4(c_L+c_U)^2}.
\]
Finally, we bound the $\chi^2$-divergence in Part (2). A similar $\chi^2$-divergence has been analyzed in \cite{frw15}. A simple change of variables enables us to directly borrow their results to obtain
\begin{align*}
&\chi^2\Big(\frac{1}{|\mathcal{S}|}\sum_{S\in \mathcal{S}}\mathbb{P}^n_{(\bmu^S,\bSigma^S)},\mathbb{P}^n_{(\bmu^0,\bSigma^0)}\Big) =\int \frac{\Big[\frac{1}{|\mathcal{S}|}\sum_{S\in \mathcal{S}}\prod_{i=1}^n\mathbb{P}_{(\bmu^S,\bSigma^S)}(\bx_i) \Big]^2}{\prod_{i=1}^n\mathbb{P}_{(\bmu^0,\bSigma^0)}(\bx_i)}d\bx_1\cdots d\bx_n-1 \\
&= \frac{1}{|S|^2}\sum_{S\in \mathcal{S}}\sum_{\tilde{S}\in \mathcal{S}} \Big| \frac{2}{c_L+c_U}(\bSigma^S+\bSigma^{\tilde{S}})-\frac{4}{(c_L+c_U)^2}\bSigma^S\bSigma^{\tilde{S}} \Big|^{-n/2}   -1 \\
&= \frac{1}{|S|^2}\sum_{S\in \mathcal{S}}\sum_{\tilde{S}\in \mathcal{S}}|1-a^2\bone^T_S\bone_{\tilde{S}}|^{-n}-1 \leq \Big[(e^{2na^2}-1)\frac{s}{p-1}+1\Big]^s \leq e^{\frac{s^2e^{2na^2}}{p-1}},
\end{align*}
where the last two equations and the first inequality follow from Lemma A.1 in \cite{frw15}, and the second inequality is due to the fact that $\log(1+x)\leq x, \forall x\geq 0$. Therefore, by choosing $\rho=\rho_0\sqrt{\frac{s\log p}{n}}$, we can conclude the lower bound,
\begin{align*}
\inf_{\hat{\theta}}\sup_{(\bmu, \bSigma)\in \mathcal{H}(s,\tau)}\mathbb{E}|\hat{\theta}-\theta| \geq C_1 \rho_0^2 \tau\frac{s\log p}{n}\exp(-e^{2s^2p^{C_2\rho_0^2-1}}),~~ \forall \rho_0 \leq \sqrt{\frac{n}{s\log p}}.
\end{align*}

\item[(ii)] $\frac{s\log p}{n} \leq  \tau \leq 1$. We set the parameters: $\bmu^0=\bzero, \bmu^S=(\gamma_0\sqrt{\frac{c_L\log p}{n}}\bone_S, 0)^T$ for $\gamma_0 \in (0,1]$, and $\bSigma^0=\bSigma^S=c_L\bI_p$. Clearly, $(\bmu^0,\bSigma^0) \in \mathcal{H}(s,\tau)$. Also since $\frac{s\log p}{n} \leq \tau$ and $\gamma_0\leq 1$,
\[
(\bmu^S)^T(\bSigma^S)^{-1}\bmu^S= \frac{\gamma_0^2(s-1)\log p}{n} \leq \tau,
\]
thus $(\bmu^S,\bSigma^S)\in \mathcal{H}(s,\tau), \forall S \in \cS$. For Part (3), we obtain
\[
|(\bmu^0)^T(\bSigma^0)^{-1}\bmu^0-(\bmu^S)^T(\bSigma^S)^{-1}\bmu^S|=\frac{\gamma_0^2(s-1)\log p}{n},~\forall S\in \mathcal{S}.
\]
The key step is to bound the $\chi^2$-divergence in Part (2):
\begin{align*}
& \chi^2 \Big(\frac{1}{|\mathcal{S}|}\sum_{S\in \mathcal{S}}\mathbb{P}^n_{(\bmu^S,\bSigma^S)},\mathbb{P}^n_{(\bmu^0,\bSigma^0)} \Big)  \\
=& \frac{1}{|\mathcal{S}|^2}\sum_{S\in \mathcal{S}}\sum_{\tilde{S}\in \mathcal{S}}\Big[(2\pi)^{-p/2}\int e^{-\frac{1}{2}\|\bx\|^2_2+c_L^{-1/2}(\bmu^S+\bmu^{\tilde{S}})^T\bx-\frac{1}{2c_L}\|\bmu^S\|_2^2-\frac{1}{2c_L}\|\bmu^{\tilde{S}}\|_2^2}d\bx \Big]^n-1 \\
=&\frac{1}{|\mathcal{S}|^2}\sum_{S\in\mathcal{S}}\sum_{\tilde{S}\in \mathcal{S}}e^{nc_L^{-1}(\bmu^S)^T\bmu^{\tilde{S}}}-1=\frac{1}{|\mathcal{S}|^2}\sum_{S\in\mathcal{S}}\sum_{\tilde{S} \in \mathcal{S}}e^{\gamma_0^2 \bone^T_{S} \bone_{\tilde{S}}\log p}-1 \\
=& \mathbb{E}_ke^{k \gamma_0^2 \log p }-1 \leq \Big[(e^{\gamma_0^2\log p}-1)\frac{s}{p-1}+1\Big]^s \leq e^{2s^2p^{\gamma_0^2-1}},
\end{align*}
where the random number $k$ equals $|S\cap \tilde{S}|$ with $S,\tilde{S}$ being two independent copies of a set randomly chosen from $\mathcal{S}$; the first inequality is again due to the result in Lemma A.1 from \cite{frw15}; and the last inequality holds because $\log(1+x)\leq x, \forall x\geq 0$. As a result, we obtain the lower bound,
\[
\inf_{\hat{\theta}}\sup_{(\bmu, \bSigma)\in \mathcal{H}(s,\tau)}\mathbb{E}|\hat{\theta}-\theta| \geq C_3 \gamma_0^2 \frac{s\log p}{n}\exp(-e^{2s^2p^{\gamma_0^2-1}}),~~ \forall \gamma_0 \in (0,1].
\]

\textbf{Remark:} Proposition \ref{prop:one} can be proved similarly as in Case (ii), by setting $\bmu^0=\bzero, \bmu^S=(\frac{\gamma_0}{\sqrt{s}}\bone_S, 0)^T, \bSigma^0=\bSigma^S=c_L\bI_p$, where $s=\frac{p}{n}$ and $\gamma_0>0$ is some fixed constant. We thus do not repeat the details.

\item[(iii)] $\tau \leq \frac{s\log p}{n}$. We construct the same parameters as we did in Case (ii) except setting $\bmu^S=(\gamma_0 \sqrt{\frac{c_L\tau}{s}}\bone_S,0)^T$. It is direct to confirm that Part (1) holds, and in Part (3),
\[
|(\bmu^0)^T(\bSigma^0)^{-1}\bmu^0-(\bmu^S)^T(\bSigma^S)^{-1}\bmu^S|=\frac{\gamma_0^2(s-1)\tau}{s}.
\]
Regarding Part (2), since $\tau \leq \frac{s\log p}{n}$, the bound for the $\chi^2$-divergence derived in Case (ii) continues to hold here. Hence, we have
\begin{eqnarray*}
\inf_{\hat{\theta}}\sup_{(\bmu, \bSigma)\in \mathcal{H}(s,\tau)}\mathbb{E}|\hat{\theta}-\theta| \geq C_4 \gamma_0^2 \tau \exp(-e^{2s^2p^{\gamma_0^2-1}}),~~ \forall \gamma_0 \in (0,1].
\end{eqnarray*}

\end{itemize}

\subsubsection{Lower bound for $\mathbb{E}\|\hat{\balpha}-\balpha\|^2_2$}
\label{lower:bound:alpha:sparse}

Since
\[
\inf_{\hat{\balpha}}\sup_{(\bmu, \bSigma)\in \mathcal{H}(s,\tau)}\mathbb{E}\|\hat{\balpha}-\balpha\|^2_2 \geq \Big(\inf_{\hat{\balpha}}\sup_{(\bmu, \bSigma)\in \mathcal{H}(s,\tau)}\mathbb{E}\|\hat{\balpha}-\balpha\|_2\Big)^2,
\]
the rest of the proof will be focused on obtaining
\[
\inf_{\hat{\balpha}}\sup_{(\bmu, \bSigma)\in \mathcal{H}(s,\tau)}\mathbb{E}\|\hat{\balpha}-\balpha\|_2 \geq C_1 \Big[\sqrt{\tau} \wedge \sqrt{\frac{(1+\tau)s\log(p/s)}{n}}\Big].
\]
We use Fano's inequality \citep{ct12} to lower bound $\mathbb{E}\|\hat{\balpha}-\balpha\|_2$. More specifically, according to Corollary 2.6 in \cite{tsybakov09}, if the followings hold,
\begin{itemize}
\item[(1)] $(\bmu^j,\bSigma^j)\in \mathcal{H}(s,\tau), j=0,1,\ldots, M$, for $M\geq 2$,
\item[(2)] $\frac{1}{M+1}\sum_{j=1}^M D_{KL}\Big(\mathbb{P}^n_{(\bmu^j,\bSigma^j)} \| \mathbb{P}^n_{(\bmu^0,\bSigma^0)} \Big) \leq  \eta \log M$ for some $\eta \in (0,1)$,
\item[(3)] $\|(\bSigma^i)^{-1}\bmu^i-(\bSigma^j)^{-1}\bmu^j\|_2 \geq 2 \zeta, 0\leq i < j \leq M$,
\end{itemize}
then
\begin{eqnarray*}
\inf_{\hat{\balpha}}\sup_{(\bmu, \bSigma)\in \mathcal{H}(s,\tau)}\mathbb{E}\|\hat{\balpha}-\balpha\|_2  \geq \zeta \cdot \Big(\frac{\log(M+1)-\log 2}{\log M}-\eta\Big).
\end{eqnarray*}
As in Section \ref{lowerbound:theta}, we will construct different parameters $\{(\bmu^j,\bSigma^j)\}_{j=0}^M$ based on the scaling of $\tau$, and verify Parts (1)-(3) to derive the lower bounds.

\begin{itemize}
\item[(i)] $\tau \geq 1$. Let $\Lambda=\{1,2,\ldots, p-\frac{s}{2}\}$, $\cQ$ be the set of all subsets of $\Lambda$ with cardinality $\frac{s}{2}$, and $\bone \in \mathbb{R}^{p-s/2}$ be the vector with all the elements equal to 1. For each $0\leq j \leq M$, we consider
\[
\bmu^j=\bar{\bmu}=(\underbrace{\sqrt{c_L\tau/s}, \sqrt{c_L\tau/s},\ldots, \sqrt{c_L\tau /s}}_{\frac{s}{2}~ {\rm times}}, 0,\ldots, 0)^T, ~~\bSigma^j=c_L \bI_p +\frac{(c_0-c_L)\ba^j(\ba^j)^T}{\|\ba^j\|^2_2},
\]
where $c_0  \in [c_L,c_U], \ba^j=\bar{\bmu}+\sqrt{\frac{c_L\tau}{s}}(\bzero,\bone_{Q^j}), Q^j\in \cQ$. The specific choice of $c_0, Q^j, M$ will be made clear later. We first verify Part (1). Observe that $\lambda_{\min}(\bSigma^j)=c_L, \lambda_{\max}(\bSigma^j)=c_0 \leq c_U$. Moreover, applying Woodbury matrix identity gives
\[
(\bSigma^j)^{-1}=\frac{1}{c_L}\bI_p-\frac{(c_0-c_L)\ba^j(\ba^j)^T}{c_Lc_0\|\ba^j\|_2^2}.
\]
Hence, $(\bmu^j)^T (\bSigma^j)^{-1}\bmu^j \leq \frac{\|\bar{\bmu}\|^2_2}{c_L} <\tau$, and $\|(\bSigma^j)^{-1}\bmu^j\|_0 \leq s$.
For Part (3), $\forall~ 0\leq i < j \leq M$,
\begin{equation}\label{pair:distance}
\|(\bSigma^i)^{-1}\bmu^i-(\bSigma^j)^{-1}\bmu^j\|_2=\frac{c_0-c_L}{2c_0c_L}\|\ba^i-\ba^j\|_2=\frac{(c_0-c_L)\sqrt{\tau}}{2c_0\sqrt{c_Ls}}\sqrt{s-2|Q^j\cap Q^i|}.
\end{equation}
Regarding Part (2), we can do the following calculations,
\begin{align}
&\frac{1}{M+1}\sum_{j=1}^M D_{KL}\Big(\mathbb{P}^n_{(\bmu^j,\bSigma^j)} \| \mathbb{P}^n_{(\bmu^0,\bSigma^0)} \Big)=\frac{1}{M+1}\sum_{j=1}^M \frac{n}{2}\Big(\log|\bSigma^0|-\log |\bSigma^j|-p+\mbox{tr}(\bSigma^j(\bSigma^0)^{-1}) \Big) \nonumber \\
&=\frac{1}{M+1}\sum_{j=1}^M \frac{n(c_0-c_L)^2(\|\ba^j\|_2^2 \cdot \|\ba^0\|_2^2-((\ba^0)^T\ba^j)^2)}{2c_0c_L\|\ba^j\|_2^2 \cdot \|\ba^0\|_2^2} \leq \frac{n(c_0-c_L)^2}{2c_0c_L}. \label{kl:divergence}
\end{align}
Based on \eqref{pair:distance} and \eqref{kl:divergence}, we now choose the value for $c_0, M$ and $Q^j~(j=0,\ldots,M)$. To obtain tighter lower bounds, it is desirable to make $M$ as large as possible while keeping the norms in \eqref{pair:distance} ``not small" (equivalently $|Q^j\cap Q^i|, i\neq j$ ``not large"). Towards that goal, we utilize an existing combinatorics result (cf. Lemma 4 in \cite{bm01}), which implies that if $p \geq 2s$ then there exists a subset $\mathcal{C}$ of $\cQ$, such that $|Q^i\cap Q^j|\leq \frac{s}{4}$ for all $Q^i \neq Q^j \in \mathcal{C}$, and
\[
\log |\mathcal{C}|\geq \frac{s}{4}\log \frac{e(p-s/2)}{4s}\geq \frac{s}{4}\log \frac{3ep}{16s}.
\]
Thus we choose $\{Q^j\}_{j=0}^M =\mathcal{C}$, and $M=|\mathcal{C}|-1$. Accordingly, from \eqref{kl:divergence} we can set $c_0=c_L(1+\sqrt{\frac{\log \frac{1}{2}+\frac{s}{4}\log\frac{3ep}{16s}}{n}})$ to satisfy Part (2) with $\eta=\frac{1}{2}$, given the condition that $p/s>C_2$. It is also clear that $c_0 \in [c_L, c_U]$ as long as $\frac{s\log p/s}{n} \leq C_3$. Finally, using the property $|Q^i\cap Q^j|\leq \frac{s}{4}$ for $i\neq j$, Equation \eqref{pair:distance} implies that Part (3) holds for $\zeta=C_4 \sqrt{\frac{\tau s\log (p/s)}{n}}$. Above all, we reach the lower bound,
\[
\inf_{\hat{\balpha}}\sup_{(\bmu, \bSigma)\in \mathcal{H}(s,\tau)}\mathbb{E}\|\hat{\balpha}-\balpha\|_2 \geq C_5  \sqrt{\frac{\tau s\log (p/s)}{n}}.
\]

\item[(ii)] $\frac{s \log (p/s)}{n}\leq \tau \leq 1$. With a bit abuse of notations, let $\cQ$ be the set of all subsets of $\{1,2,\ldots, p\}$ with cardinality $s$. We set the parameters:
\[
\bmu^j=\sqrt{\frac{\rho c_L\log(p/s)}{n}}\bone_{Q^j}, ~\rho \in (0,1), ~Q^j \in \cQ,~~ \bSigma^j=c_L \bI_p, ~j=0,1,\ldots, M.
\]
We will specify $M$ and $\{Q^j\}_{j=0}^M$ shortly. Since $\frac{s \log (p/s)}{n}\leq \tau$, it is straightforward to confirm that Part (1) holds. In Parts (2) and (3), we obtain respectively,
\[
\frac{1}{M+1}\sum_{j=1}^M D_{KL}\Big(\mathbb{P}^n_{(\bmu^j,\bSigma^j)} \| \mathbb{P}^n_{(\bmu^0,\bSigma^0)} \Big) = \frac{\rho \log(p/s)}{M+1}\sum_{j=1}^M(s-|Q^j\cap Q^0|)\leq \rho s\log (p/s),
\]
and
\[
\|(\bSigma^i)^{-1}\bmu^i-(\bSigma^j)^{-1}\bmu^j\|_2=\sqrt{\frac{2\rho \log(p/s)}{nc_L}}\sqrt{s-|Q^i \cap Q^j|}.
\]
We then use the same arguments as in Case (i) to choose $M$ and $\{Q^j\}_{j=0}^M$. Under the condition that $p/s>C_2$, we will be able to choose sufficiently small $\rho$ such that $\rho s\log(p/s) \leq \frac{1}{2}\log M$. As a result, we will achieve the lower bound,
\[
\inf_{\hat{\balpha}}\sup_{(\bmu, \bSigma)\in \mathcal{H}(s,\tau)}\mathbb{E}\|\hat{\balpha}-\balpha\|_2 \geq C_6  \sqrt{\frac{ s\log (p/s)}{n}}.
\]
\item[(iii)] $\tau \leq \frac{s\log(p/s)}{n}$. We consider
\[
\bmu^j=\sqrt{\frac{ \rho c_L\tau}{s}}\bone_{Q^j}, ~\rho \in (0,1), ~Q^j \in \cQ, ~~\bSigma^j=c_L \bI_p, ~j=0,1,\ldots, M.
\]
Here $\cQ$ is the same as in Case (ii). Clearly Part (1) holds. For Parts (2) and (3), with the same choice of $\rho, M$ and $\{Q^j\}_{j=0}^M$ as in Case (ii), it is not hard to verify
\begin{align*}
\frac{1}{M+1}\sum_{j=1}^M D_{KL}\Big(\mathbb{P}^n_{(\bmu^j,\bSigma^j)} \| \mathbb{P}^n_{(\bmu^0,\bSigma^0)} \Big)&=\frac{\rho n\tau}{s(M+1)}\sum_{j=1}^M(s-|Q^j\cap Q^0|) \leq \frac{1}{2}\log M, \\
 \|(\bSigma^i)^{-1}\bmu^i-(\bSigma^j)^{-1}\bmu^j\|_2&=\sqrt{\frac{2\rho \tau}{c_Ls}}\sqrt{s-|Q^i \cap Q^j|}\geq C_7\sqrt{\tau},
\end{align*}
where we have used the fact $\tau \leq \frac{s\log(p/s)}{n}$ in the first inequality. Therefore, we obtain the lower bound,
\[
\inf_{\hat{\balpha}}\sup_{(\bmu, \bSigma)\in \mathcal{H}(s,\tau)}\mathbb{E}\|\hat{\balpha}-\balpha\|_2 \geq C_8 \sqrt{\tau}.
\]
\end{itemize}

\subsection{Proof of Proposition \ref{prop:two}}
\label{proof:corollary2}
\begin{proof}
By the definition of $\tilde{\balpha}$ in \eqref{lasso:alpha}, it is clear that the quadratic function $(\frac{1}{2}\tilde{\balpha}^T\hat{\bSigma}\tilde{\balpha})w^2+(\lambda \|\tilde{\balpha}\|_1-\tilde{\balpha}^T\hat{\bmu})w$ achieves minimum over $[0,1]$ at $w=1$, which implies
\begin{align*}
\hat{\bmu}^T\tilde{\balpha}-\tilde{\balpha}^T\hat{\bSigma}\tilde{\balpha} \geq \lambda \|\tilde{\balpha}\|_1\geq \lambda (\|\balpha\|_1-\|\tilde{\balpha}-\balpha\|_1).
\end{align*}
The above inequality results in the following bounds,
\begin{align*}
|\tilde{\theta}_{c}-\theta|&=|\tilde{\theta}_c-\tilde{\theta}+\tilde{\theta}-\theta| \geq   |c-2| \cdot |\hat{\bmu}^T\tilde{\balpha}-\tilde{\balpha}^T\hat{\bSigma}\tilde{\balpha}| - |\tilde{\theta}-\theta| \\
&\geq \lambda |c-2|  \cdot \|\balpha\|_1- \lambda |c-2| \cdot \|\tilde{\balpha}-\balpha\|_1- |\tilde{\theta}-\theta|.
\end{align*}
Hence we can obtain
\begin{align}
\sup_{(\bmu,\bSigma)\in \cH(s,\tau)}\EE|\tilde{\theta}_c-\theta| \geq & \underbrace{\lambda |c-2|\cdot \sup_{(\bmu,\bSigma)\in \cH(s,\tau)}\lonenorm{\balpha}}_{:=\cJ_1} -\underbrace{\lambda |c-2|\cdot \sqrt{\sup_{(\bmu,\bSigma)\in \cH(s,\tau)} \EE\|\tilde{\balpha}-\balpha\|^2_1}}_{:=\cJ_2} \nonumber \\
&-\underbrace{\sup_{(\bmu,\bSigma)\in \cH(s,\tau)}\EE|\tilde{\theta}-\theta|}_{:=\cJ_3}. \nonumber
\end{align}
Let $\bmu^*=(\underbrace{\sqrt{\tau c_U/s},\ldots, \sqrt{\tau c_U/s}}_{s {\rm~times}}, 0,\ldots, 0),\bSigma^*=\diag(c_U,c_U,\ldots,c_U)$. It is straightforward to confirm that $(\bmu^*,\bSigma^*)\in \cH(s,\tau)$. So
\[
\sup_{(\bmu,\bSigma)\in \cH(s,\tau)}\lonenorm{\balpha}\geq \lonenorm{(\bSigma^*)^{-1}\bmu^*}=\sqrt{\frac{s\tau}{c_U}},
\]
from which it holds that $\cJ_1 \gtrsim \sqrt{\frac{\tau(1+\tau)s\log p}{n}}$. Regarding $\cJ_2$, under the scaling $\frac{s\log p}{n}=o(1), p^{-\delta}\lesssim \frac{s\log p}{n}$, the upper bound we have derived for $\mathbb{E}\|\tilde{\balpha}-\balpha\|^2_1$  in \eqref{l1error:alpha} leads to $\cJ_2 \lesssim \frac{(1+\tau)s\log p}{n}$. For $\cJ_3$, we already know from \eqref{clean:upper:bound},
\[
\cJ_3 \lesssim \frac{(1+\tau)s\log p}{n} + \frac{\tau+\sqrt{\tau}}{\sqrt{n}}.
\]
Given the condition $\tau \gtrsim \frac{s\log p}{n}$, it is not hard to verify that $\cJ_1$ is the dominant term among $\cJ_i, i=1,2,3$. This completes the proof.
\end{proof}

\subsection{Proof of Theorem \ref{corollary:three}}
\label{proof:corollary:three}

\begin{proof}
Throughout the proof, we use $C_1, C_2, \ldots$ to denote positive constants that possibly depend on $\nu, c_L, c_U$. Some of them may depend on additional quantities, and clarification will be made when necessary.

Recall the definition of $\check{\balpha}$ in \eqref{alpha:lzero}. The basic inequality holds that
\begin{equation*}
\frac{1}{2}\check{\balpha}\hat{\bSigma}\check{\balpha}-\check{\balpha}^T\hat{\bmu} \leq \frac{1}{2}\balpha\hat{\bSigma}\balpha-\balpha^T \hat{\bmu},
\end{equation*}
which is equivalent to
\begin{align*}
\frac{1}{2}(\balpha-\check{\balpha})^T\hat{\bSigma} (\balpha-\check{\balpha}) \leq &\balpha^T(\hat{\bSigma}-\bSigma)(\balpha-\check{\balpha}) +(\bmu-\hat{\bmu})^T(\balpha-\check{\balpha}).
\end{align*}

Since $\|\check{\balpha}-\balpha\|_0\leq 2s$, on the event $\mathcal{A}(\lambda) \cap \cC(s, \kappa)$, using a similar inequality to \eqref{basic:eq:two} we can continue from the above inequality to obtain,
\begin{equation*}
\frac{1}{2}(\lambda^{1/2}_{\min}(\bSigma)-\kappa)^2\|\check{\balpha}-\balpha\|_2^2\leq \frac{\sqrt{s}\lambda}{\sqrt{2}}\|\check{\balpha}-\balpha\|_2,
\end{equation*}
leading to for $0<\kappa <\lambda^{1/2}_{\min}(\bSigma)$,
\begin{equation*}
\|\check{\balpha}-\balpha\|^2_2 \leq \frac{2s\lambda^2}{(\lambda^{1/2}_{\min}(\bSigma)-\kappa)^4}.
\end{equation*}
The rest of the derivation of the upper bound for $\EE\ltwonorm{\check{\balpha}-\balpha}^2$ is similar to the one for $\mathbb{E}\ltwonorm{\tilde{\balpha}-\balpha}^2$ in Theorem \ref{thm:one}, we thus do not repeat the arguments.

Regarding the bound for $\EE|\check{\theta}_c-\theta|$, we first consider the case $c=2$. With a minor modification of the proof of Lemma \ref{error:functional}, it can be obtained that on the event $\mathcal{A}(\lambda) \cap \cC(s, \kappa)$, the same upper bound (up to constants) as the one in Lemma \ref{error:functional} holds for $|\check{\theta}_2-\theta|$. Accordingly, the bound for $\EE|\check{\theta}_2-\theta|$ can be derived in the same way as for $\EE|\tilde{\theta}-\theta|$ in Theorem \ref{thm:one}. We will not repeat the details here. When $c \neq 2$, we use
\begin{align}\label{theta:l0:noteq2}
\sup_{(\bmu,\bSigma)\in \mathcal{H}(s,\tau)} \mathbb{E}|\check{\theta}_c-\theta|\leq \sup_{(\bmu,\bSigma)\in \mathcal{H}(s,\tau)} \mathbb{E}|\check{\theta}_{2}-\theta|+|c-2|\sup_{(\bmu,\bSigma)\in \mathcal{H}(s,\tau)}\mathbb{E}|\check{\balpha}^T\hat{\bmu}-\check{\balpha}^T\hat{\bSigma}\check{\balpha}|.
\end{align}
We have already shown for the case $c=2$ that
\begin{align}
\label{theta:l0:noteq2:one}
\sup_{(\bmu,\bSigma)\in \mathcal{H}(s,\tau)} \mathbb{E}|\check{\theta}_{2}-\theta| \lesssim \frac{(1+\tau)s\log p}{n}+\frac{\tau+\sqrt{\tau}}{\sqrt{n}}.
\end{align}
The remainder of the proof is to bound $\mathbb{E}|\check{\balpha}^T\hat{\bmu}-\check{\balpha}^T\hat{\bSigma}\check{\balpha}|$. Define the event
\[
\cD=\bigg\{\sup_{|S| \leq s}\ltwonorm{\hat{\bSigma}^{-1}_{SS}\hat{\bmu}_{S}}\leq \gamma \bigg \}.
\]
Based on the definition of $\check{\balpha}$ in \eqref{alpha:lzero}, it is straightforward to verify that on the event $\cD$, $\check{\balpha}$ will take the following form,
\begin{align*}
\check{\balpha}_{\hat{S}}=\hat{\bSigma}_{\hat{S}\hat{S}}^{-1}\hat{\bmu}_{\hat{S}},
\end{align*}
where $\hat{S}$ is the support of $\check{\balpha}$, and $|\hat{S}| \leq s$. As a result, on the event $\cD$ it holds that
\begin{align*}
\check{\balpha}^T\hat{\bmu}-\check{\balpha}^T\hat{\bSigma}\check{\balpha}=0,
\end{align*}
thus we can bound $\mathbb{E}|\check{\balpha}^T\hat{\bmu}-\check{\balpha}^T\hat{\bSigma}\check{\balpha}|$ in the following way,
\begin{align}
&\EE|\check{\balpha}^T\hat{\bmu}-\check{\balpha}^T\hat{\bSigma}\check{\balpha}| =\EE|\check{\balpha}^T\hat{\bmu}-\check{\balpha}^T\hat{\bSigma}\check{\balpha}|\mathbbm{1}_{\cD^c} \leq 3\EE(\check{\balpha}^T\hat{\bmu}\mathbbm{1}_{\cD^c}) \nonumber \\
&= 3\EE(\check{\balpha}^T\bmu \mathbbm{1}_{\cD^c})+3\EE(\check{\balpha}^T(\hat{\bmu}-\bmu) \mathbbm{1}_{\cD^c}), \label{bound:starting}
\end{align}
where the inequality is due to the fact that $\frac{1}{2}\check{\balpha}^T\hat{\bSigma}\check{\balpha}-\check{\balpha}^T\hat{\bmu}\leq 0$. Moreover, we have
\begin{equation}
\label{bound:continue:one}
|\EE(\check{\balpha}^T\bmu \mathbbm{1}_{\cD^c})| \leq \ltwonorm{\bmu}\cdot  \ltwonorm{\check{\balpha}} \cdot \PP(\cD^c) \leq \gamma\sqrt{pc_U\tau}\PP(\cD^c),
\end{equation}
and
\begin{align}
& |\EE(\check{\balpha}^T(\hat{\bmu}-\bmu) \mathbbm{1}_{\cD^c})| \leq \EE(\lonenorm{\check{\balpha}}\cdot \linfnorm{\hat{\bmu}-\bmu}\mathbbm{1}_{\cD^c})\leq \sqrt{s}\gamma\EE( \linfnorm{\hat{\bmu}-\bmu} \mathbbm{1}_{\cD^c})  \nonumber \\
&\leq \sqrt{s}\gamma \sqrt{\EE\linfnorm{\hat{\bmu}-\bmu}^2 \cdot \PP(\cD^c)} \leq C_1 \gamma \sqrt{\frac{s\log p}{n}\cdot \PP(\cD^c)}. \label{bound:continue:two}
\end{align}
Here, to obtain the last inequality we have used \eqref{mu:infinity} to compute
\begin{align*}
\EE\linfnorm{\hat{\bmu}-\bmu}^2&=\int_{0}^{\frac{C_2\log p}{n}}\PP(\linfnorm{\hat{\bmu}-\bmu}^2>t)dt+\int_{\frac{C_2\log p}{n}}^{\infty}\PP(\linfnorm{\hat{\bmu}-\bmu}^2>t)dt \\
&\leq \frac{C_2\log p}{n}+2p\int_{\frac{C_2\log p}{n}}^{\infty}\exp(-C_3n t)dt\leq  \frac{C_4 \log p}{n}.
\end{align*}
We now bound $\PP(\cD^c)$. We first have
\begin{align*}
\sup_{|S| \leq s}\ltwonorm{\hat{\bSigma}_{SS}^{-1}\hat{\bmu}_{S}}^2 & \leq \Big[\inf_{|S| \leq s} \lambda_{\min}(\hat{\bSigma}_{SS})\Big]^{-1}\cdot \Big[\inf_{|S| \leq s} \lambda_{\min}(\bSigma^{-1/2}_{SS}\hat{\bSigma}_{SS}\bSigma^{-1/2}_{SS})\Big]^{-1} \cdot \sup_{|S| \leq s} \ltwonorm{\bSigma^{-1/2}_{SS}\hat{\bmu}_S}^2.
\end{align*}
We bound the three terms on the above right-hand side, respectively. According to Lemma \ref{key:concentration:eq}, it holds with probability at least $1-4p^{-c}$ that
\begin{align*}
\inf_{|S| \leq s}\lambda_{\min}(\hat{\bSigma}_{SS})&=\inf_{\ltwonorm{\bu}=1, \lzeronorm{\bu}\leq s}\bu^T\hat{\bSigma} \bu \geq \Big(\inf_{\ltwonorm{\bu}=1}\sqrt{\bu \bSigma \bu}- \sup_{\bu \in B_0(s)\cap B_2(1)}\big|\sqrt{\bu\hat{\bSigma}\bu}-\sqrt{\bu \bSigma \bu}\big|\Big)^2 \nonumber \\
& \geq \Bigg(\sqrt{c_L}-C_5\sqrt{\frac{s\log p}{n}}-\frac{C_6\frac{s\log p}{n}}{\big(\sqrt{c_L}-C_5\sqrt{\frac{s\log p}{n}}\big)_+}\Bigg)_+^2, \label{mini:v1}
\end{align*}
and with similar arguments that
\begin{align*}
\inf_{|S| \leq s}\lambda_{\min}(\bSigma^{-1/2}_{SS}\hat{\bSigma}_{SS}\bSigma^{-1/2}_{SS}) \geq \Bigg(1-C_7\sqrt{\frac{s\log p}{n}}-\frac{C_8\frac{s\log p}{n}}{\big(1-C_7\sqrt{\frac{s\log p}{n}}\big)_+}\Bigg)^2_+,
\end{align*}
where for $j=5,6,7,8$, $C_j >0 $ depend on $c_U,\nu, c$ and $C_j \rightarrow \infty$ as $c\rightarrow \infty$. Regarding the third term,
\begin{align*}
\sup_{|S| \leq s} \ltwonorm{\bSigma^{-1/2}_{SS}\hat{\bmu}_S}^2&\leq 2\Big(\sup_{|S| \leq s} \ltwonorm{\bSigma^{-1/2}_{SS}\bmu_S}^2+ \sup_{|S| \leq s} \ltwonorm{\bSigma^{-1/2}_{SS}(\hat{\bmu}_S-\bmu_S)}^2\Big)\\
&\leq 2(\tau+c_L^{-1}s\|\hat{\bmu}-\bmu\|_{\infty}^2).
\end{align*}
We further upper bound the above using \eqref{mu:infinity} to yield
\[
\PP(\linfnorm{\hat{\bmu}-\bmu}^2\leq \tau  t /s)\geq 1-2p\exp(-C_8n\tau t/s), \quad \forall t>0.
\]
Above all, as long as $\frac{s\log p}{n}$ is sufficiently small, we will be able to set $\gamma=C_9\sqrt{(1+t)\tau}$ such that
\begin{align}
\label{bound:on:d}
\PP(\cD^c)=\PP\bigg(\sup_{|S| \leq s}\ltwonorm{\hat{\bSigma}_{SS}^{-1}\hat{\bmu}_{S}}^2 > \gamma^2 \bigg) \leq 4p^{-c}+2p\exp(-C_8n\tau t/s).
\end{align}
Under the conditions that $\frac{s\log p}{n}=o(1), p^{-\delta}\lesssim \frac{s\log p}{n}, \tau \gtrsim \frac{s\log p}{n}$, we can choose the constants $c, t$ in \eqref{bound:on:d} large enough so that
\[
\PP(\cD^c) \lesssim \frac{p^{-1/2}s\log p}{n}.
\]
The above combined with \eqref{bound:starting}, \eqref{bound:continue:one} and \eqref{bound:continue:two} shows that
\begin{align}
\label{theta:l0:noteq2:two}
\sup_{(\bmu,\bSigma)\in\cH(s,\tau)}\mathbb{E}|\check{\balpha}^T\hat{\bmu}-\check{\balpha}^T\hat{\bSigma}\check{\balpha}| \lesssim \frac{(\tau+\sqrt{\tau})s\log p}{n}.
\end{align}
Putting together the results \eqref{theta:l0:noteq2}, \eqref{theta:l0:noteq2:one}, and \eqref{theta:l0:noteq2:two} establishes the upper bound for $\EE|\check{\theta}_c-\theta|$.
\end{proof}


\subsection{Proof of Theorem \ref{approximate:sparse:thm1}}
\label{proof:of:thm3}

The proof is a direct generalization of the one in the exact sparsity setting. We simply highlight the major different steps, and refer to the proof of Theorem \ref{thm:one} for details. With approximate sparsity, we define a ``pseudo-support" set of $\balpha$ as\footnote{With a bit abuse of notations, we have adopted some notations from the exact sparsity case with slightly different meanings.}
\[
S=\Big\{i\in [p]: |\alpha_i|> (1+\sqrt{\tilde{\tau}})\sqrt{\frac{\log p}{n}}\Big \},
\]
and let $|S|=s$. The following two lemmas are useful throughout the proof.

\begin{lem}
\label{signal:strength}
For any $(\bmu, \bSigma) \in \cH_q(R,\tau)$, it holds that
\[
\bmu^T\bSigma^{-1} \bmu \leq \tilde{\tau}, \quad \sum_{i=1}^p |\alpha_i|^q \leq \tilde{R}.
\]
\end{lem}
\begin{proof}
Recall $\balpha=\bSigma^{-1}\bmu$. Given $(\bmu, \bSigma) \in \cH_q(R,\tau)$, it satisfies that
\[
\bmu^T\bSigma^{-1}\bmu \leq \tau \quad \sum_{i=1}^p |\alpha_i|^q\leq R.
\]
Hence,
\[
\bmu^T\bSigma^{-1}\bmu =\balpha^T \bSigma \balpha \leq \delta_{\bSigma}\lonenorm{\balpha}^2\leq c_U \Big(\sum_{i=1}^n|\alpha_i|^q\Big)^{\frac{2}{q}}\leq c_U R^{\frac{2}{q}},
\]
yielding that $\bmu^T\bSigma^{-1}\bmu \leq (c_UR^{\frac{2}{q}}) \wedge \tau=\tilde{\tau}$. Moreover, by H\"older's inequality
\[
\sum_{i=1}^p|\alpha_i|^q\leq \Big(\sum_{i=1}^p|\alpha|^2\Big)^{\frac{q}{2}}\cdot \Big(\sum_{i=1}^p 1\Big)^{1-\frac{q}{2}}\leq p^{1-\frac{q}{2}}c_L^{-\frac{q}{2}}(\balpha^T\bSigma\balpha)^{\frac{q}{2}} \leq p^{1-\frac{q}{2}}c_L^{-\frac{q}{2}} \tau^{\frac{q}{2}}.
\]
Thus $\sum_{i=1}^p|\alpha_i|^q \leq (p^{1-\frac{q}{2}}c_L^{-\frac{q}{2}} \tau^{\frac{q}{2}}) \wedge R=\tilde{R}$.
\end{proof}

\begin{lem}
\label{pseudo:support}
For any $\balpha$ with $\lqnorm{\balpha}^q\leq \tilde{R}$, it holds that
\begin{align*}
s \leq \tilde{R} (1+\sqrt{\tilde{\tau}})^{-q}\Big(\frac{\log p}{n} \Big)^{-\frac{q}{2}}, \quad \lonenorm{\balpha_{S^c}} \leq \tilde{R}(1+\sqrt{\tilde{\tau}})^{1-q}\Big(\frac{\log p}{n}\Big)^{\frac{1-q}{2}}.
\end{align*}
\end{lem}
\begin{proof}
For the bound on $s$, observe that
\begin{align*}
\tilde{R}\geq \sum_{i=1}^p |\alpha_i|^q\mathbbm{1}_{|\alpha_i|>(1+\sqrt{\tilde{\tau}}) \sqrt{\frac{\log p}{n}}} \geq s(1+\sqrt{\tilde{\tau}})^q \Big(\frac{\log p}{n}\Big)^{\frac{q}{2}}.
\end{align*}
The bound for $\lonenorm{\balpha_{S^c}}$ is due to
\begin{align*}
\tilde{R}\geq \sum_{i=1}^p|\alpha_i|\cdot |\alpha_i|^{q-1}\mathbbm{1}_{|\alpha_i| \leq  (1+\sqrt{\tilde{\tau}})\sqrt{\frac{\log p}{n}}} \geq (1+\sqrt{\tilde{\tau}})^{q-1}\Big(\frac{\log p}{n}\Big)^{\frac{q-1}{2}} \lonenorm{\balpha_{S^c}}.
\end{align*}
\end{proof}

\noindent We are in the position to derive the bounds for $\EE\ltwonorm{\tilde{\balpha}-\balpha}^2$ and $\EE|\tilde{\theta}-\theta|$. Throughout the proof, $C_1, C_2, \ldots$ are used to denote positive constants that possibly depend on $\nu, c_L, c_U$. Some of them may depend on additional quantities, and clarification will be made when necessary.

\vspace{0.2cm}

\emph{Upper bound for $\EE\ltwonorm{\tilde{\balpha}-\balpha}^2$}. Due to the approximate sparsity, we have a generalization of \eqref{primal:dev}: on the event $\cA(\lambda)$
\begin{align}
\label{gen:421}
0 \geq \frac{1}{2}(\tilde{\balpha}-\balpha)^T\hat{\bSigma}(\tilde{\balpha}-\balpha)+\frac{\lambda}{2} (\lonenorm{\tilde{\balpha}_{S^c}-\balpha_{S^c}}-3\lonenorm{\tilde{\balpha}_S-\balpha_S}-4\lonenorm{\balpha_{S^c}}),
\end{align}
implying
\[
\lonenorm{\tilde{\balpha}_{S^c}-\balpha_{S^c}}\leq 3\lonenorm{\tilde{\balpha}_S-\balpha_S}+4\lonenorm{\balpha_{S^c}}.
\]
Thus we obtain
\begin{align}
\label{thm3:eqone}
\lonenorm{\tilde{\balpha}-\balpha} \leq 4\lonenorm{\tilde{\balpha}_S-\balpha_S}+4\lonenorm{\balpha_{S^c}}  \leq 4\sqrt{s}\ltwonorm{\tilde{\balpha}-\balpha}+4\lonenorm{\balpha_{S^c}}.
\end{align}
Combining this result with Lemma \ref{key:concentration:eq}, we can proceed from \eqref{gen:421} to conclude, with probability at least $1-4p^{-c}-\PP(\cA^c(\lambda))$,
\begin{align}
0 \geq &\frac{1}{2}\Bigg(\sqrt{c_L}\ltwonorm{\tilde{\balpha}-\balpha}-C_1\sqrt{\frac{\log p}{n}}\lonenorm{\tilde{\balpha}-\balpha} -\frac{C_2\frac{\log p}{n}\lonenorm{\tilde{\balpha}-\balpha}^2}{(\sqrt{c_L}\ltwonorm{\tilde{\balpha}-\balpha}-C_1\sqrt{\frac{\log p}{n}}\lonenorm{\tilde{\balpha}-\balpha})_+}\Bigg)_+^2 \nonumber \\
&-\frac{\lambda}{2}(3\sqrt{s}\ltwonorm{\tilde{\balpha}-\balpha}+4\lonenorm{\balpha_{S^c}}), \label{basic:starting:approx}
\end{align}
where the positive constants $C_1,C_2 \rightarrow \infty$ as $c\rightarrow \infty$. We now show that the above inequality leads to the result
\begin{align}
\label{approx:sparse:alpha:error}
\ltwonorm{\tilde{\balpha}-\balpha}^2 \lesssim (1+ \tilde{\tau})^{1-\frac{q}{2}}\tilde{R}\Big(\frac{\log p}{n}\Big)^{1-\frac{q}{2}}.
\end{align}
Towards that goal, under the scaling $\tilde{R}(1+\tilde{\tau})^{-\frac{q}{2}}(\frac{\log p}{n})^{1-\frac{q}{2}}=o(1)$, we can assume
\[
\ltwonorm{\tilde{\balpha}-\balpha} \gg (1+\sqrt{\tilde{\tau}})^{1-q}\tilde{R}\Big(\frac{\log p}{n}\Big)^{1-\frac{q}{2}}.
\]
Otherwise \eqref{approx:sparse:alpha:error} trivially holds.  Using this assumption together with Lemma \ref{pseudo:support} and \eqref{thm3:eqone}, it is straightforward to verify that
\[
\sqrt{\frac{\log p}{n}}\lonenorm{\tilde{\balpha}-\balpha}  \ll \ltwonorm{\tilde{\balpha}-\balpha}.
\]
Hence, \eqref{basic:starting:approx} can be simplified to
\begin{align}
\label{follow:direct:alpha}
\ltwonorm{\tilde{\balpha}-\balpha}^2- (1+ \tilde{\tau})^{\frac{1}{2}-\frac{q}{4}}\tilde{R}^{\frac{1}{2}}\Big(\frac{\log p}{n}\Big)^{\frac{1}{2}-\frac{q}{4}}\cdot \ltwonorm{\tilde{\balpha}-\balpha} - (1+ \tilde{\tau})^{1-\frac{q}{2}}\tilde{R}\Big(\frac{\log p}{n}\Big)^{1-\frac{q}{2}}\lesssim 0,
\end{align}
where we have used Lemma \ref{pseudo:support}. The bound \eqref{approx:sparse:alpha:error} follows directly from \eqref{follow:direct:alpha}. We thus have shown
\begin{align}
\label{alpha:error:proof}
\PP\bigg(\ltwonorm{\tilde{\balpha}-\balpha}^2 \lesssim (1+ \tilde{\tau})^{1-\frac{q}{2}}\tilde{R}\Big(\frac{\log p}{n}\Big)^{1-\frac{q}{2}}\bigg) \geq 1-4p^{-c}-\PP(\cA^c(\lambda)).
\end{align}
Finally, by Lemma \ref{signal:strength} it holds that $\ltwonorm{\tilde{\balpha}}\lesssim \tilde{\tau},\ltwonorm{\balpha}\lesssim \tilde{\tau}$. We obtain
\begin{align*}
\EE\ltwonorm{\tilde{\balpha}-\balpha}^2\lesssim (1+ \tilde{\tau})^{1-\frac{q}{2}}\tilde{R}\Big(\frac{\log p}{n}\Big)^{1-\frac{q}{2}}+\tilde{\tau} (p^{-c}+\PP(\cA^c(\lambda)).
\end{align*}
The proof is completed by the fact that under the condition $p^{-\delta}\lesssim \tilde{R}\Big(\frac{\log p}{n}\Big)^{1-\frac{q}{2}}( \tilde{\tau}^{-\frac{q}{2}} \vee  \tilde{\tau}^{-1})$,
\[
\tilde{\tau} (p^{-c}+\PP(\cA^c(\lambda))\lesssim (1+\tilde{\tau})^{1-\frac{q}{2}}\tilde{R}\Big(\frac{\log p}{n}\Big)^{1-\frac{q}{2}},
\]
where $\PP(\cA^c(\lambda)) \lesssim p^{-c}$ holds by Lemma \ref{rare:events} Part (i).

\emph{Upper bound for $\EE|\tilde{\theta}-\theta|$}. The key step is to obtain an analog of Lemma \ref{error:functional}. Observe that the inequality \eqref{first:upper} continues to hold here. Hence on the event $\mathcal{A}(\lambda)$ we have
\begin{align*}
\tilde{\theta}-\theta\leq \lambda \lonenorm{\tilde{\balpha}-\balpha}+(\balpha-\tilde{\balpha})^T\bSigma(\balpha-\tilde{\balpha})+|\balpha^T(\hat{\bSigma}-\bSigma)\balpha|  +2|\balpha^T(\hat{\bmu}-\bmu)|
\end{align*}
From the proof of Lemma \ref{key:concentration:eq}, we can directly obtain that with probability at least $1-4p^{-c}$,
\begin{align*}
\sqrt{\bu^T\bSigma\bu}\leq \sqrt{\bu^T\hat{\bSigma}\bu+|\bu^T(\hat{\bmu}-\bmu)|^2}+C_3 \sqrt{\frac{\log p}{n}}\lonenorm{\bu}\leq \sqrt{\bu^T\hat{\bSigma}\bu}+C_4 \sqrt{\frac{\log p}{n}}\lonenorm{\bu}, \quad \forall \bu\in \RR^p.
\end{align*}
This result combined with \eqref{gen:421} enables us to bound $(\balpha-\tilde{\balpha})^T\bSigma(\balpha-\tilde{\balpha})$,
\begin{align*}
(\balpha-\tilde{\balpha})^T\bSigma(\balpha-\tilde{\balpha})\lesssim \lambda (\sqrt{s}\ltwonorm{\tilde{\balpha}-\balpha}+\lonenorm{\balpha_{S^c}})+\frac{\log p}{n}\lonenorm{\balpha-\tilde{\balpha}}^2.
\end{align*}
We thus have that with probability at least $1-4p^{-c}-\PP(\cA^c(\lambda))$,
\begin{align*}
\tilde{\theta}-\theta \lesssim & \underbrace{\frac{\log p}{n}\lonenorm{\tilde{\balpha}-\balpha}^2}_{:=\cJ_1}+ \underbrace{\lambda (\sqrt{s}\ltwonorm{\tilde{\balpha}-\balpha}+\lonenorm{\balpha_{S^c}})}_{:=\cJ_2} +\underbrace{\lambda \lonenorm{\tilde{\balpha}-\balpha}}_{:=\cJ_3} +|\balpha^T(\hat{\bSigma}-\bSigma)\balpha|+2|\balpha^T(\hat{\bmu}-\bmu)|.
\end{align*}
Based on the results \eqref{thm3:eqone}, \eqref{alpha:error:proof} and  Lemma \ref{pseudo:support}, it is direct to verify that
\[
\cJ_1=O\Big((1+ \tilde{\tau})^{1-q}\tilde{R}^2\big(\frac{\log p}{n}\big)^{2-q}\Big),~~\cJ_2=O\Big((1+\tilde{\tau})^{1-\frac{q}{2}}\tilde{R}\big(\frac{\log p}{n}\big)^{1-\frac{q}{2}}\Big),~~\cJ_3=O\Big((1+\tilde{\tau})^{1-\frac{q}{2}}\tilde{R}\big(\frac{\log p}{n}\big)^{1-\frac{q}{2}}\Big),
\]
and under the scaling $\tilde{R}(1+\tilde{\tau})^{-\frac{q}{2}}(\frac{\log p}{n})^{1-\frac{q}{2}}=o(1)$, $\cJ_3$ is the dominant term. Hence, it holds with probability at least $1-4p^{-c}-\PP(\cA^c(\lambda))$ that
\[
\tilde{\theta}-\theta\lesssim (1+ \tilde{\tau})^{1-\frac{q}{2}}\tilde{R}\big(\frac{\log p}{n}\big)^{1-\frac{q}{2}}+|\balpha^T(\hat{\bSigma}-\bSigma)\balpha|+|\balpha^T(\hat{\bmu}-\bmu)|.
\]
Regarding the lower bound, \eqref{new:back:look} remains valid. Given the order of $\cJ_3$ we have derived, it holds that with probability at least $1-4p^{-c}-\PP(\cA^c(\lambda))$,
\[
\tilde{\theta}-\theta \gtrsim -(1+ \tilde{\tau})^{1-\frac{q}{2}}\tilde{R}\big(\frac{\log p}{n}\big)^{1-\frac{q}{2}}- |\balpha^T(\hat{\bSigma}-\bSigma)\balpha|-|\balpha^T(\hat{\bmu}-\bmu)|.
\]
Equipped with Lemma \ref{signal:strength} and under the scaling condition $p^{-\delta}\lesssim \tilde{R}\Big(\frac{\log p}{n}\Big)^{1-\frac{q}{2}} (\tilde{\tau}^{-\frac{q}{2}}\vee \tilde{\tau}^{-1})$, the remainder of the proof follows the same line of arguments in the proof of Theorem \ref{thm:one}. We thus do not repeat the details.


\subsection{Proof of Theorem \ref{approxiamte:sparse:thm2}}
\label{proof:lower:bound:lqball}

The proof generalizes the one of Theorem \ref{lower:bound}. Hence we do not detail out every step, and will refer to the proof of Theorem \ref{lower:bound} on many occasions. We use $C_1,C_2,\ldots$ to denote positive constants possibly depending on $c_L, c_U$.

\subsubsection{Lower bound for $\EE|\hat{\theta}-\theta|$}

The derivation for the lower bound $\tilde{\tau}\wedge \frac{\tilde{\tau}+\sqrt{\tilde{\tau}}}{\sqrt{n}}$ is almost the same as for the lower bound $\tau \wedge \frac{\tau+\sqrt{\tau}}{\sqrt{n}}$ in Theorem \ref{lower:bound}. The only modification is to replace $\tau$ by $\tilde{\tau}$. Then all the arguments there continue to hold (up to constants). The next step is to obtain the other lower bound $\tilde{\tau} \wedge \Big[ (1+ \tilde{\tau})^{1-\frac{q}{2}} \tilde{R} \big(\frac{\log p}{n}\big)^{1-\frac{q}{2}} \Big]$. We follow closely the arguments that were used to derive the lower bound $\big[\tau \wedge \frac{(\tau+1)s\log p}{n}\big]c_0 \exp(-e^{2s^2p^{c_6 c_0-1}})$ in Theorem \ref{lower:bound}. For simplicity, we merely point out the differences in the following.

\begin{itemize}
\item[(i)] $\tilde{\tau} \geq 1$. Replace $s$ by $\tilde{R}(\frac{\log p}{n})^{-\frac{q}{2}}\tilde{\tau}^{-\frac{q}{2}}$, and $\tau$ by $\tilde{\tau}$. Under the scaling $\tilde{R}\tilde{\tau}^{-\frac{q}{2}}\big(\frac{\log p}{n}\big)^{1-\frac{q}{2}}=o(1)$, all the arguments remain valid (up to constants), except that the verification of $\lqnorm{(\bSigma^0)^{-1}\bmu^0}^q \leq R, \lqnorm{(\bSigma^S)^{-1}\bmu^S}^q \leq R$ is required. This can be easily verified as
\begin{align*}
\lqnorm{(\bSigma^0)^{-1}\bmu^0}^q\lesssim \tilde{\tau}^{\frac{q}{2}}\lesssim R, \quad  \lqnorm{(\bSigma^S)^{-1}\bmu^S}^q \lesssim \tilde{\tau}^{\frac{q}{2}}+c^q\tilde{\tau}^{\frac{q}{2}}s \lesssim \tilde{\tau}^{\frac{q}{2}}+\tilde{R}\lesssim R.
\end{align*}
Thus we obtain the lower bound, $\forall \rho_0 \leq  (\tilde{R}(\frac{\log p}{n})^{1-\frac{q}{2}}\tilde{\tau}^{-\frac{q}{2}})^{-\frac{1}{2}}$
\begin{align*}
\inf_{\hat{\theta}}\sup_{(\bmu, \bSigma)\in \mathcal{H}_q(R,\tau)}\mathbb{E}|\hat{\theta}-\theta| \geq C_1 \rho_0^2 \tilde{\tau}^{1-\frac{q}{2}} \tilde{R}\Big(\frac{\log p}{n}\Big)^{1-\frac{q}{2}}\exp(-e^{2 \tilde{R}^2(\frac{\log p}{n})^{-q}\tilde{\tau}^{-q} p^{C_2\rho_0^2-1}}).
\end{align*}
Given the condition $\tilde{R}^2\tilde{\tau}^{-q}(\frac{\log p}{n})^{-q} \lesssim p^{1-\delta}$ for some $\delta >0$, we can choose $\rho_0 \asymp 1$ so that
\[
2 \tilde{R}^2\Big(\frac{\log p}{n}\Big)^{-q}\tilde{\tau}^{-q} p^{C_2\rho_0^2-1} \leq 1,
\]
leading to the lower bound
\[
\inf_{\hat{\theta}}\sup_{(\bmu,\bSigma) \in \cH_q(R, \tau)} \EE|\hat{\theta}-\theta| \gtrsim \tilde{\tau}^{1-\frac{q}{2}}\tilde{R}\Big(\frac{\log p}{n}\Big)^{1-\frac{q}{2}}.
\]
\item[(ii)] $\tilde{R}(\frac{\log p}{n})^{1-\frac{q}{2}} \leq \tilde{\tau} \leq 1$. Replace $s$ by $\tilde{R}(\frac{\log p}{n})^{-\frac{q}{2}}$, and $\tau$ by $\tilde{\tau}$. With the scaling $\tilde{R}(\frac{\log p}{n})^{1-\frac{q}{2}}=o(1)$, all continue to hold, up to constants. The result $\lqnorm{(\bSigma^0)^{-1}\bmu^0}^q \leq R, \lqnorm{(\bSigma^S)^{-1}\bmu^S}^q \leq R$ is also straightforward to confirm. Using the condition $\tilde{R}^2(\frac{\log p}{n})^{-q} \lesssim p^{1- \delta}$ for some $\delta>0$, we can choose $\gamma_0\asymp 1$ to obtain the bound,
\begin{align*}
\inf_{\hat{\theta}}\sup_{(\bmu,\bSigma) \in \cH_q(R, \tau)} \EE|\hat{\theta}-\theta| \gtrsim \tilde{R}\Big(\frac{\log p}{n}\Big)^{1-\frac{q}{2}}.
\end{align*}
\item[(iii)] $\tilde{\tau} \leq \tilde{R}(\frac{\log p}{n})^{1-\frac{q}{2}}$. The same modification as in Case (ii). We will have
\begin{align*}
\inf_{\hat{\theta}}\sup_{(\bmu,\bSigma) \in \cH_q(R, \tau)} \EE|\hat{\theta}-\theta| \gtrsim  \tilde{\tau}.
\end{align*}
\end{itemize}

\subsubsection{Lower bound for $\EE\ltwonorm{\hat{\balpha}-\balpha}^2$}

Again, we adapt the derivation of the lower bound for $\EE\ltwonorm{\hat{\balpha}-\balpha}^2$ in Theorem \ref{lower:bound} (cf. Section \ref{lower:bound:alpha:sparse}). We summarize the modifications below.

\begin{itemize}
\item[(i)] $\tilde{\tau}\geq 1$. Replace $s$ by $\tilde{R}\tilde{\tau}^{-\frac{q}{2}}(\frac{\log p}{n})^{-\frac{q}{2}}$, $\tau$ by $\tilde{\tau}$, and set $\bmu^j=\bar{\bmu}=(\sqrt{c_L\tilde{\tau}}, 0, 0, \ldots, 0)^T$. Under the condition $1\lesssim \tilde{R}(1+\tilde{\tau})^{-\frac{q}{2}}(\frac{\log p}{n})^{-\frac{q}{2}}\lesssim p^{1-\delta}$ for some $\delta$, the support size $\tilde{R}\tilde{\tau}^{-\frac{q}{2}}(\frac{\log p}{n})^{-\frac{q}{2}}$ is a legitimate number between $1$ and $p$. Moreover, it is straightforward to confirm that under the scaling $\tilde{R}(1+\tilde{\tau})^{-\frac{q}{2}}(\frac{\log p}{n})^{-\frac{q}{2}}\lesssim p^{1-\delta}, \tilde{R}(1+\tilde{\tau})^{-\frac{q}{2}}(\frac{\log p}{n})^{1-\frac{q}{2}}=o(1)$,  all the arguments in Case (i) of Section \ref{lower:bound:alpha:sparse} remain valid (up to constants). Additionally, we need verify  $\lqnorm{(\bSigma^j)^{-1}\bmu^j}^q \leq R$. This is true because
\begin{align*}
\lqnorm{(\bSigma^j)^{-1}\bmu^j}^q \lesssim (\sqrt{\tilde{\tau}})^q + (c_0-c_L)^q \tilde{\tau}^{\frac{q}{2}}\Big(\tilde{R}\big(\frac{\log p}{n} \big)^{-\frac{q}{2}}\tilde{\tau}^{-\frac{q}{2}}\Big)^{1-\frac{q}{2}} \lesssim R.
\end{align*}
where $c_0-c_L\asymp \sqrt{\tilde{R}(\frac{\log p}{n})^{1-\frac{q}{2}}\tilde{\tau}^{-\frac{q}{2}}}$ according to Section \ref{lower:bound:alpha:sparse}. Hence we obtain
\[
\inf_{\hat{\balpha}}\sup_{(\bmu,\bSigma)\in \cH_q(R, \tau)}\EE\ltwonorm{\hat{\balpha}-\balpha}^2 \gtrsim \tilde{\tau}^{1-\frac{q}{2}}\Big(\frac{\log p}{n} \Big)^{1-\frac{q}{2}}\tilde{R}.
\]
\item[(ii)] $\tilde{R}(\frac{\log p}{n})^{1-\frac{q}{2}} \leq  \tilde{\tau} \leq 1$. Replace $s$ by $\tilde{R}(\frac{\log p}{n})^{-\frac{q}{2}}$, and $\tau$ by $\tilde{\tau}$. Then under the scaling $\tilde{R}(\frac{\log p}{n})^{-\frac{q}{2}}\lesssim p^{1-\delta}, \tilde{R}(\frac{\log p}{n})^{1-\frac{q}{2}}=o(1)$, the arguments in Section \ref{lower:bound:alpha:sparse} continue to hold, and $\lqnorm{(\bSigma^j)^{-1}\bmu^j}^q \leq R$ can be verified easily (up to constants). Therefore, we can conclude
\[
\inf_{\hat{\balpha}}\sup_{(\bmu,\bSigma)\in \cH_q(R, \tau)}\EE\ltwonorm{\hat{\balpha}-\balpha}^2 \gtrsim  \tilde{R}\Big(\frac{\log p}{n}\Big)^{1-\frac{q}{2}}.
\]
\item[(iii)] $\tilde{\tau} \leq \tilde{R}(\frac{\log p}{n})^{1-\frac{q}{2}}$. Do the same change as in Case (ii) and we are able to obtain
\[
\inf_{\hat{\balpha}}\sup_{(\bmu,\bSigma)\in \cH_q(R, \tau)}\EE\ltwonorm{\hat{\balpha}-\balpha}^2 \gtrsim \tilde{\tau}.
\]
\end{itemize}

\subsection{Proof and Discussion in Dense Regime}
\label{dense:combine:all}

\subsubsection{Proof of Proposition \ref{lowerbound:dense}}
\label{proof:denselower}

\begin{proof}
Given the condition $p=O(s^2)$, without loss of generality, we assume $p\leq s^2$. The proof is a direct modification of that of Theorem \ref{lower:bound}(b) in Section \ref{lowerbound:theta}. The lower bound $\tau \wedge \frac{\tau + \sqrt{\tau}}{\sqrt{n}}$ has been already obtained in Section \ref{lowerbound:theta}. We now derive the other one $\tau \wedge \frac{(1+\tau)\sqrt{p}}{n}$. This is done by a simple adaptation of the three cases in the derivation of the lower bound $\big[\tau \wedge \frac{(\tau+1)s\log p}{n}\big]c_0 \exp(-e^{2s^2p^{c_6 c_0-1}})$ from Section \ref{lowerbound:theta}. We thus do not repeat the details and only point out the changes:
\begin{itemize}
\item[(i)] $\tau\geq 1$. Set $s=\sqrt{p}$ and $\rho=(n^{-1/2}p^{1/4})\wedge 1$ in Case (i) of Section \ref{lowerbound:theta}, which yields 
\[
\inf_{\hat{\theta}}\sup_{(\bmu,\bSigma)\in \mathcal{H}(s,\tau)}\mathbb{E}|\hat{\theta}-\theta| \gtrsim \tau \wedge \frac{\tau \sqrt{p}}{n} \gtrsim \tau \wedge \frac{(1+\tau)\sqrt{p}}{n}.
\]
\item[(ii)] $0<\tau \leq 1$. Set $\mathcal{S}=\{S\subseteq [p]: |S|=\sqrt{p}\}, \bmu^{S}=\sqrt{c_L}[n^{-1/2}\wedge (p^{-1/4}\tau^{1/2})]\bone_S$ in Case (ii) of Section \ref{lowerbound:theta}, leading to the bound
\[
\inf_{\hat{\theta}}\sup_{(\bmu,\bSigma)\in \mathcal{H}(s,\tau)}\mathbb{E}|\hat{\theta}-\theta| \gtrsim \tau \wedge \frac{\sqrt{p}}{n} \gtrsim \tau \wedge \frac{(1+\tau)\sqrt{p}}{n}.
\]
\end{itemize}
\end{proof}

\subsubsection{Proof of Proposition \ref{gaussian:case}}
\label{proof:prop2new}

\begin{proof}
Since $\bx_i \overset{i.i.d.}{\sim}N(\bmu,\bSigma)$, it is known that  $\hat{\bmu}\sim N(\bmu, n^{-1}\bSigma), \hat{\bSigma}^{-1}\sim n\cdot \mathcal{W}_p^{-1}(\bSigma^{-1},n-1)$, and $\hat{\bmu}$ is independent from $\hat{\bSigma}^{-1}$. We use these results to first show that $\hat{\theta}_g$ is unbiased. 
\begin{align}
\label{bias:zero}
\mathbb{E}(\hat{\theta}_g)=\frac{n-p-2}{n}\mbox{tr}[\mathbb{E}(\hat{\bSigma}^{-1})\cdot \mathbb{E}(\hat{\bmu}\hat{\bmu}^T)]-\frac{p}{n} =\mbox{tr}\big[\bSigma^{-1}\cdot (\bmu\bmu^T+n^{-1}\bSigma)\big]-\frac{p}{n}=\bmu^T\bSigma^{-1}\bmu.
\end{align}
Here, we have used the fact that the mean of an inverse Wishart distribution $\mathcal{W}_p^{-1}({\bf{\Psi}},k)$ is $\frac{{\bf \Psi}}{k-p-1}$. We next evaluate the variance of $\hat{\theta}_g$. 
\begin{align}
\label{variance:decomp}
\var(\hat{\theta}_g)=\frac{(n-p-2)^2}{n^2}\cdot \Big(\underbrace{\mathbb{E}[\mbox{var}(\hat{\bmu}^T\hat{\bSigma}^{-1}\hat{\bmu}\mid \hat{\bSigma})]}_{:=\cJ_1}+\underbrace{\mbox{var}[\mathbb{E}(\hat{\bmu}^T\hat{\bSigma}^{-1}\hat{\bmu}\mid \hat{\bSigma})]}_{:=\cJ_2}\Big).
\end{align}
It is direct to verify that for a given symmetric matrix $\bA$ and $\by\sim N(\bxi,\bLambda)$, it holds that $\var(\by^T\bA\by)\lesssim \|\bLambda^{1/2}\bA\bLambda^{1/2}\|_F^2+\bxi^T\bA\bLambda \bA\bxi$, which implies 
\begin{align*}
\cJ_1 \lesssim \mathbb{E}\big(\|n^{-1}\bSigma^{1/2}\hat{\bSigma}^{-1}\bSigma^{1/2}\|_F^2+n^{-1}\bmu^T\hat{\bSigma}^{-1}\bSigma\hat{\bSigma}^{-1}\bmu\big)=\mathbb{E}(\|\bH\|_F^2)+n(\bSigma^{-1/2}\bmu)^T\cdot \mathbb{E}(\bH^2)\cdot (\bSigma^{-1/2}\bmu),
\end{align*}
where $\bH$ is defined as $\bH=n^{-1}\bSigma^{1/2}\hat{\bSigma}^{-1}\bSigma^{1/2}\sim \mathcal{W}_p^{-1}(\bI_p, n-1)$. According to moments of the inverse Wishart distribution \citep{vr88}, it is straightforward to compute the above upper bound to obtain,
\begin{align}
\label{var:j1}
\cJ_1\lesssim \frac{(n-2)(p+n\theta)}{(n-p-1)(n-p-2)(n-p-4)}
\end{align}
Regarding $\cJ_2$, we have $\cJ_2=\var\big[\bmu^T\hat{\bSigma}^{-1}\bmu+n^{-1}\mbox{tr}(\hat{\bSigma}^{-1}\bSigma)\big]\leq 2\var(\bmu^T\hat{\bSigma}^{-1}\bmu)+2\var(n^{-1}\mbox{tr}(\hat{\bSigma}^{-1}\bSigma))$. We again utilize the moments of the inverse Wishart distribution in \cite{vr88} to derive 
\begin{align*}
\var(\bmu^T\hat{\bSigma}^{-1}\bmu)=\frac{2n^2\theta^2}{(n-p-2)^2(n-p-4)}, ~\var(n^{-1}\mbox{tr}(\hat{\bSigma}^{-1}\bSigma))=\frac{2p(n-2)}{(n-p-1)(n-p-2)^2(n-p-4)}.
\end{align*}
Hence,
\begin{align}
\label{var:j2}
\cJ_2 \lesssim \frac{2p(n-2)+2\theta^2n^2(n-p-1)}{(n-p-1)(n-p-2)^2(n-p-4)}.
\end{align}
Under the scaling condition $\limsup_{n\rightarrow \infty}\frac{p}{n}< 1$, putting together the results \eqref{bias:zero}, \eqref{variance:decomp}, \eqref{var:j1} and \eqref{var:j2} completes the proof. 
\end{proof}

\subsubsection{Proof of Theorem \ref{subgaussian:case}}
\label{proof:thm3new}
\begin{proof}
Let $b_n=\lfloor \frac{n}{m+1}\rfloor$ denote the sample size of the $(j+1)$th part ($1\leq j \leq m$). So the sample size of the first part is $n-mb_n$. Setting $\bA=\tilde{\bSigma}^{-1/2}_{(0)}\bSigma\tilde{\bSigma}^{-1/2}_{(0)}$ in Lemma \ref{identity:deviation} Part (i) yields
\begin{align}
\label{key:eq}
&\bSigma^{-1}-\sum_{k=0}^m \tilde{\bSigma}^{-1/2}_{(0)}(\bI_p-\tilde{\bSigma}^{-1/2}_{(0)}\bSigma\tilde{\bSigma}^{-1/2}_{(0)})^k\tilde{\bSigma}^{-1/2}_{(0)} \nonumber \\
=& \tilde{\bSigma}^{-1/2}_{(0)}(\tilde{\bSigma}^{1/2}_{(0)}\bSigma^{-1}\tilde{\bSigma}^{1/2}_{(0)})^{1/2}(\bI_p-\tilde{\bSigma}^{-1/2}_{(0)}\bSigma\tilde{\bSigma}^{-1/2}_{(0)})^{m+1}(\tilde{\bSigma}^{1/2}_{(0)}\bSigma^{-1}\tilde{\bSigma}^{1/2}_{(0)})^{1/2}\tilde{\bSigma}^{-1/2}_{(0)}:=\Delta.
\end{align}
We first evaluate the bias of $\hat{\theta}_{sg}$. Using the mutual independence between different parts,
\begin{align*}
\mathbb{E}(\hat{\theta}_{sg})&=\mathbb{E}\sum_{k=0}^m\Big(\hat{\bmu}_{(0)}^T\tilde{\bSigma}^{-1/2}_{(0)}\cdot (\bI_p-\tilde{\bSigma}^{-1/2}_{(0)}\bSigma\tilde{\bSigma}_{(0)}^{-1/2})^k\cdot \tilde{\bSigma}^{-1/2}_{(0)}\hat{\bmu}_{(0)}\Big)-\frac{p(m+1)}{n}\\
&=\mathbb{E}\big(\hat{\bmu}_{(0)}^T\bSigma^{-1}\hat{\bmu}_{(0)}-\hat{\bmu}_{(0)}^T\Delta \hat{\bmu}_{(0)}\big)-\frac{p(m+1)}{n}=\theta+\frac{p}{n-mb_n}-\frac{p(m+1)}{n}-\mathbb{E}(\hat{\bmu}_{(0)}^T\Delta \hat{\bmu}_{(0)})
\end{align*}
where the second equality is due to \eqref{key:eq}. Hence,
\begin{align*}
&|\mathbb{E}(\hat{\theta}_{sg})-\theta|\leq \mathbb{E}|\hat{\bmu}_{(0)}^T\Delta \hat{\bmu}_{(0)}|+\frac{p(m+1)}{n}-\frac{p}{n-mb_n} \nonumber \\
\leq & \mathbb{E}\Big[\|\hat{\bmu}_{(0)}^T\tilde{\bSigma}^{-1/2}_{(0)}(\tilde{\bSigma}^{1/2}_{(0)}\bSigma^{-1}\tilde{\bSigma}^{1/2}_{(0)})^{1/2}\|_2^2\cdot \|\bI_p-\tilde{\bSigma}^{-1/2}_{(0)}\bSigma\tilde{\bSigma}^{-1/2}_{(0)}\|_2^{m+1}\Big]+\frac{pm(m+1)^2}{n^2+nm(m+1)} \nonumber \\
=&\mathbb{E}\Big[\hat{\bmu}_{(0)}^T\bSigma^{-1}\hat{\bmu}_{(0)} \cdot \|\bI_p-\tilde{\bSigma}^{-1/2}_{(0)}\bSigma\tilde{\bSigma}^{-1/2}_{(0)}\|_2^{m+1}\Big]+\frac{pm(m+1)^2}{n^2+nm(m+1)}\nonumber \\
\leq &\sqrt{\mathbb{E}(\hat{\bmu}_{(0)}^T\bSigma^{-1}\hat{\bmu}_{(0)})^2}\cdot \sqrt{\mathbb{E}\|\bI_p-\tilde{\bSigma}^{-1/2}_{(0)}\bSigma\tilde{\bSigma}^{-1/2}_{(0)}\|_2^{2m+2}}+\frac{pm(m+1)^2}{n^2+nm(m+1)} \nonumber \\
\overset{(a)}{\lesssim} & \Big(\theta+\sqrt{\frac{\theta}{n}}+\frac{p}{n}\Big)\cdot \Big(\frac{p}{n}\Big)^{\frac{m+1}{2}}+\frac{p}{n^2} \lesssim \frac{\theta+\sqrt{\theta}}{\sqrt{n}}+\frac{\sqrt{p}}{n},
\end{align*}
where the last inequality holds since $p=O(n^{\alpha}), m=\lceil \frac{\alpha}{1-\alpha}\rceil$. To obtain (a), we have used Lemma \ref{identity:deviation} Parts (iv) and (v). \\
We now turn to analyze the variance of $\hat{\theta}_{sg}$. Denote the first part of the data by $\mathcal{D}_0$. We have $\mbox{var}(\hat{\theta}_{sg})=\mbox{var}(\mathbb{E}(\hat{\theta}_{sg}\mid \mathcal{D}_0))+\mathbb{E}(\mbox{var}(\hat{\theta}_{sg}\mid \mathcal{D}_0))$. The preceding calculation of $\mathbb{E}(\hat{\theta}_{sg})$ implies
\begin{align*}
\mbox{var}(\mathbb{E}(\hat{\theta}_{sg}\mid \mathcal{D}_0))=\mbox{var}\big(\hat{\bmu}_{(0)}^T\bSigma^{-1}\hat{\bmu}_{(0)}-\hat{\bmu}_{(0)}^T\Delta\hat{\bmu}_{(0)}\big) \leq 2\mbox{var}(\hat{\bmu}_{(0)}^T\bSigma^{-1}\hat{\bmu}_{(0)})+2\mbox{var}(\hat{\bmu}_{(0)}^T\Delta\hat{\bmu}_{(0)}).
\end{align*}
As in bounding the bias, we apply Lemma \ref{identity:deviation} Parts (iv) and (v) to have
\begin{align*}
&\mbox{var}(\hat{\bmu}_{(0)}^T\Delta\hat{\bmu}_{(0)})\leq \mathbb{E}(\hat{\bmu}_{(0)}^T\Delta\hat{\bmu}_{(0)})^2\leq \sqrt{\mathbb{E}(\hat{\bmu}_{(0)}^T\bSigma^{-1}\hat{\bmu}_{(0)})^4}\cdot \sqrt{\mathbb{E}\|\bI_p-\tilde{\bSigma}^{-1/2}_{(0)}\bSigma\tilde{\bSigma}^{-1/2}_{(0)}\|_2^{4m+4}} \\
\lesssim & \Big(\theta^2+\frac{\theta}{n}+\frac{p^2}{n^2}\Big)\cdot \Big(\frac{p}{n}\Big)^{m+1}\lesssim \frac{\theta^2+\theta}{n}+\frac{p}{n^2}.
\end{align*}
We use the Efron-Stein inequality to bound $\var(\hat{\bmu}_{(0)}^T\bSigma^{-1}\hat{\bmu}_{(0)})$. Let $\bxi_{-1}=\sum_{i=2}^{n-mb_n}\bx_i$. Then
\begin{align*}
&\var(\hat{\bmu}_{(0)}^T\bSigma^{-1}\hat{\bmu}_{(0)})\leq (n-mb_n)\cdot \mathbb{E}(\hat{\bmu}_{(0)}^T\bSigma^{-1}\hat{\bmu}_{(0)}-(n-mb_n)^{-2}\bxi_{-1}^T\bSigma^{-1}\bxi_{-1})^2 \\
\lesssim &\frac{1}{n^3}\mathbb{E}(\bx_1^T\bSigma^{-1}\bx_1)^2+\frac{1}{n^3}\mathbb{E}(\bx_1^T\bSigma^{-1}\bxi_{-1})^2 \overset{(a)}{\lesssim} \frac{\theta^2+\theta+p^2}{n^3}+
\frac{1}{n^3}(\mathbb{E}(\bmu^T\bSigma^{-1}\bxi_{-1})^2+\mathbb{E}(\bxi_{-1}^T\bSigma^{-1}\bxi_{-1})) \\
\lesssim & \frac{\theta^2+\theta+p^2}{n^3}+\frac{n^2\theta^2+n\theta+n^2\theta+np}{n^3} \lesssim \frac{\theta^2+\theta}{n}+\frac{p}{n^2},
\end{align*}
where in (a) we have used the independence between $\bx_1$ and $\bxi_{-1}$ and Lemma \ref{identity:deviation} Part (v). We have thus showed that $\mbox{var}(\mathbb{E}(\hat{\theta}_{sg}\mid \mathcal{D}_0))\lesssim \frac{\theta^2+\theta}{n}+\frac{p}{n^2}$. \\
It remains to derive the same upper bound for $\mathbb{E}(\mbox{var}(\hat{\theta}_{sg}\mid \mathcal{D}_0))$. Since $m$ is a constant, we have
\begin{align}
\label{start:from}
\mathbb{E}(\mbox{var}(\hat{\theta}_{sg}\mid \mathcal{D}_0))\lesssim \sum_{k=1}^m\mathbb{E}\Big(\var(\underbrace{\hat{\bmu}_{(0)}^T\tilde{\bSigma}^{-1/2}_{(0)}\cdot \prod_{j=1}^k(\bI_p-\tilde{\bSigma}^{-1/2}_{(0)}\hat{\bSigma}_{(j)}\tilde{\bSigma}_{(0)}^{-1/2})\cdot \tilde{\bSigma}^{-1/2}_{(0)}\hat{\bmu}_{(0)}}_{:=f_k(\bx_{1},\ldots, \bx_{n-(m-k)b_n})} \mid \mathcal{D}_0)\Big)
\end{align}
Let $f^{(i)}_k(\bx_{1},\ldots, \bx_{n-(m-k)b_n})$ be the function obtained by replacing $\bx_i$ in $f_k(\bx_{1},\ldots, \bx_{n-(m-k)b_n})$ with $\bmu$. Conditioned on $\mathcal{D}_0$, $f_k(\bx_{1},\ldots, \bx_{n-(m-k)b_n})$ is a function of samples from the second to the $(k+1)$th parts.  We use the Efron-Stein inequality to bound the conditional variance to obtain
\begin{align}
\label{efron-stein:cond}
&\mathbb{E}[\var(f_k(\bx_{1},\ldots, \bx_{n-(m-k)b_n})\mid \mathcal{D}_0)] \nonumber \\
=&\sum_{i=n-mb_n+1}^{n-mb_n+kb_n}\mathbb{E}[(f_k(\bx_{1},\ldots, \bx_{n-(m-k)b_n})-f^{(i)}_k(\bx_{1},\ldots, \bx_{n-(m-k)b_n}))^2].
\end{align}
It is straightforward to verify that each $f_k(\bx_{1},\ldots, \bx_{n-(m-k)b_n})-f^{(i)}_k(\bx_{1},\ldots, \bx_{n-(m-k)b_n})$ above takes the form:
\begin{align*}
\hat{\bmu}_{(0)}^T\tilde{\bSigma}^{-1/2}_{(0)}\cdot \underbrace{\prod_{j=1}^{\tilde{k}-1}(\bI_p-\tilde{\bSigma}^{-1/2}_{(0)}\hat{\bSigma}_{(j)}\tilde{\bSigma}_{(0)}^{-1/2})}_{:=\bA}\cdot \bH \cdot \underbrace{\prod_{j=\tilde{k}+1}^{k}(\bI_p-\tilde{\bSigma}^{-1/2}_{(0)}\hat{\bSigma}_{(j)}\tilde{\bSigma}_{(0)}^{-1/2})}_{:=\bB}\cdot \underbrace{\tilde{\bSigma}^{-1/2}_{(0)}\hat{\bmu}_{(0)}}_{:=\ba},
\end{align*}
where $1\leq \tilde{k}\leq k, \bH=-b_n(b_n-1)^{-1}\tilde{\bSigma}^{-1/2}_{(0)}\big(b_n^{-1}(\bx_j-\bmu)(\bx_j-\bmu)^T-b_n^{-2}(\bx_j-\bmu)(\bx_j-\bmu)^T-b_n^{-2}(\bx_j-\bmu)(\bxi_{-j}-(b_n-1)\bmu)^T-b_n^{-2}(\bxi_{-j}-(b_n-1)\bmu)(\bx_j-\bmu)^T\big)\tilde{\bSigma}^{-1/2}_{(0)}$, $\bx_j$ is a sample from the $(\tilde{k}+1)$th part, and $\bxi_{-j}$ is the sum of samples from the $(\tilde{k}+1)$th part with $\bx_j$ removed. We now bound each $\mathbb{E}(\ba^T\bA\bH\bB\ba)^2$ in \eqref{efron-stein:cond}. According to the form of $\bH$, we have  
\begin{align*}
\mathbb{E}(\ba^T\bA\bH\bB\ba)^2 \lesssim& \mathbb{E}(\ba^T\bA\tilde{\bSigma}^{-1/2}_{(0)}b_n^{-1}(\bx_j-\bmu)(\bx_j-\bmu)^T\tilde{\bSigma}^{-1/2}_{(0)}\bB\ba)^2+ \\
&\mathbb{E}(\ba^T\bA\tilde{\bSigma}^{-1/2}_{(0)}b_n^{-2}(\bx_j-\bmu)(\bxi_{-j}-(b_n-1)\bmu)^T\tilde{\bSigma}^{-1/2}_{(0)}\bB\ba)^2 +\\
&\mathbb{E}(\ba^T\bA\tilde{\bSigma}^{-1/2}_{(0)}b_n^{-2}(\bxi_{-j}-(b_n-1)\bmu)(\bx_j-\bmu)^T\tilde{\bSigma}^{-1/2}_{(0)}\bB\ba)^2:=\cJ_1+\cJ_2+\cJ_3.
\end{align*}
We first bound $\cJ_1$. 
\begin{align*}
\cJ_1 & \leq  b_n^{-2} \sqrt{\mathbb{E}(\ba^T\bA\tilde{\bSigma}^{-1/2}_{(0)}(\bx_j-\bmu))^4}\cdot \sqrt{\mathbb{E}((\bx_j-\bmu)^T\tilde{\bSigma}^{-1/2}_{(0)}\bB\ba)^4} \\
&\overset{(a)}{\lesssim}b_n^{-2}  \sqrt{\mathbb{E}\|\ba^T\bA\tilde{\bSigma}^{-1/2}_{(0)}\bSigma^{1/2}\|_2^4}\cdot \sqrt{\mathbb{E}\|\bSigma^{1/2}\tilde{\bSigma}^{-1/2}_{(0)}\bB\ba\|_2^4} \\
& \lesssim b_n^{-2}  \sqrt{\mathbb{E}(\|\ba\|_2^4\|\bA\|_2^4\|\tilde{\bSigma}^{-1/2}_{(0)}\bSigma^{1/2}\|_2^4)}\cdot \sqrt{\mathbb{E}(\|\bSigma^{1/2}\tilde{\bSigma}^{-1/2}_{(0)}\|_2^4\|\bB\|_2^4\|\ba\|_2^4)}  \\
& \overset{(b)}{\lesssim} b_n^{-2} (\mathbb{E}\|\ba\|_2^8)^{1/2}\cdot (\mathbb{E}\|\bA\|_2^{16})^{1/8}\cdot (\mathbb{E}\|\bB\|_2^{16})^{1/8}\cdot (\mathbb{E}\|\tilde{\bSigma}^{-1/2}_{(0)}\bSigma^{1/2}\|_2^{16})^{1/4},
\end{align*}
Here, (a) is due to the independence between $\bx_j$ and $\{\ba, \bA, \bB\}$, and $\|\bSigma^{-1/2}(\bx_j-\bmu)\|_{\psi_2}\lesssim 1$; (b) is by multiple applications of H\"{o}lder's inequality. Moreover, 
\begin{align*}
\mathbb{E}\|\ba\|_2^8=&\mathbb{E}|\hat{\bmu}_{(0)}^T\tilde{\bSigma}^{-1}_{(0)}\hat{\bmu}_{(0)}|^4\lesssim \mathbb{E}|\hat{\bmu}_{(0)}^T\bSigma^{-1/2}(\bSigma^{1/2}\tilde{\bSigma}^{-1}_{(0)}\bSigma^{1/2}-\bI_p)\bSigma^{-1/2}\hat{\bmu}_{(0)}|^4+\mathbb{E}|\hat{\bmu}^T_{(0)}\bSigma^{-1}\hat{\bmu}_{(0)}|^4 \\
\leq & (\mathbb{E}|\hat{\bmu}^T_{(0)}\bSigma^{-1}\hat{\bmu}_{(0)}|^8)^{1/2}\cdot (\mathbb{E}\|\bSigma^{1/2}\tilde{\bSigma}^{-1}_{(0)}\bSigma^{1/2}-\bI_p\|_2^8)^{1/2}+\mathbb{E}|\hat{\bmu}_{(0)}^T\bSigma^{-1}\hat{\bmu}_{(0)}|^4  \lesssim \theta^4+\frac{\theta^2}{n^2}+\frac{p^4}{n^4},
\end{align*}
where we have used Lemma \ref{identity:deviation} Parts (iii) and (v) to obtain the last inequality. We apply Lemma \ref{identity:deviation} Parts (iii) again to get
\begin{align}
\label{reuse:arg}
\mathbb{E}\|\tilde{\bSigma}^{-1/2}_{(0)}\bSigma^{1/2}\|_2^{16}=\mathbb{E}\|\bSigma^{1/2}\tilde{\bSigma}^{-1}_{(0)}\bSigma^{1/2}\|_2^{8}\lesssim \mathbb{E}\|\bSigma^{1/2}\tilde{\bSigma}^{-1}_{(0)}\bSigma^{1/2}-\bI_p\|_2^{8}+\mathbb{E}\|\bI_p\|_2^8\lesssim 1.
\end{align}
We now bound $\mathbb{E}\|\bA\|_2^{16}$ and $\mathbb{E}\|\bB\|_2^{16}$. First note that
\begin{align*}
&\|\bI_p-\tilde{\bSigma}^{-1/2}_{(0)}\hat{\bSigma}_{(j)}\tilde{\bSigma}_{(0)}^{-1/2}\|_2 \\
=&\|(\tilde{\bSigma}^{-1/2}_{(0)}\bSigma^{1/2})\cdot [(\tilde{\bSigma}^{-1/2}_{(0)}\bSigma^{1/2})^{-1}(\bSigma^{1/2}\tilde{\bSigma}^{-1/2}_{(0)})^{-1}-\bSigma^{-1/2}\hat{\bSigma}_{(j)}\bSigma^{-1/2}]\cdot (\bSigma^{1/2}\tilde{\bSigma}^{-1/2}_{(0)})\|_2 \\
\leq& \|\tilde{\bSigma}^{-1/2}_{(0)}\bSigma^{1/2}\|_2^2\cdot \Big(\|\underbrace{\bI_p-\bSigma^{-1/2}\hat{\bSigma}_{(j)}\bSigma^{-1/2}}_{:=\bA_j}\|_2+\|\underbrace{\bI_p-\bSigma^{-1/2}\tilde{\bSigma}_{(0)}\bSigma^{-1/2}}_{:=\bA_0}\|_2\Big)
\end{align*}
The above bound combined with H\"{o}der's inequality yields
\begin{align*}
\mathbb{E}\|\bA\|_2^{16} \lesssim & \sqrt{\mathbb{E}\|\tilde{\bSigma}^{-1/2}_{(0)}\bSigma^{1/2}\|_2^{64(\tilde{k}-1)}}\cdot \sqrt{\mathbb{E}\prod_{j=1}^{\tilde{k}-1}(\|\bA_j\|^{32}_2+\|\bA_0\|^{32}_2)} \\
\overset{(a)}{=}&\sqrt{\mathbb{E}\|\tilde{\bSigma}^{-1/2}_{(0)}\bSigma^{1/2}\|_2^{64(\tilde{k}-1)}}\cdot \sqrt{\mathbb{E}(\mathbb{E}\|\bA_1\|^{32}_2+\|\bA_0\|^{32}_2)^{\tilde{k}-1}} \\
 \lesssim &\sqrt{\mathbb{E}\|\tilde{\bSigma}^{-1/2}_{(0)}\bSigma^{1/2}\|_2^{64(\tilde{k}-1)}}\cdot \sqrt{(\mathbb{E}\|\bA_1\|^{32}_2)^{\tilde{k}-1}+\mathbb{E}\|\bA_0\|^{32(\tilde{k}-1)}_2}  \overset{(b)}{\lesssim}1,
\end{align*}
where (a) holds because $\bA_0,\ldots,\bA_{\tilde{k}-1}$ are mutually independent and $\bA_1,\ldots, \bA_{\tilde{k}-1}$ are identically distributed; (b) is implied by the arguments in \eqref{reuse:arg} and Lemma \ref{identity:deviation} Part (ii). We can then obtain $\mathbb{E}\|\bB\|_2^{16}\lesssim 1$ in a similar way. We have thus derived the bound for $\cJ_1$ as $\cJ_1\lesssim \frac{\theta^2+\theta}{n^2}+\frac{p^2}{n^4}$. The bound for $\cJ_2$ can be obtained in the following,
\begin{align*}
\cJ_2& =b_n^{-4}\cdot \mathbb{E}\big\{\mathbb{E}[(\ba^T\bA\tilde{\bSigma}^{-1/2}_{(0)}(\bx_j-\bmu))^2\cdot ((\bxi_{-j}-(b_n-1)\bmu)^T\tilde{\bSigma}^{-1/2}_{(0)}\bB\ba)^2\mid \bA,\bB,\mathcal{D}_0]\big\}\\
&=\frac{b_n-1}{b_n^4}\cdot \mathbb{E}(\|\ba^T\bA\tilde{\bSigma}^{-1/2}_{(0)}\bSigma^{1/2}\|^2_2\cdot \|\bSigma^{1/2}\tilde{\bSigma}^{-1/2}_{(0)}\bB\ba\|_2^2) \\
&\leq \frac{b_n-1}{b_n^4}\cdot  \sqrt{\mathbb{E}\|\ba^T\bA\tilde{\bSigma}^{-1/2}_{(0)}\bSigma^{1/2}\|_2^4}\cdot \sqrt{\mathbb{E}\|\bSigma^{1/2}\tilde{\bSigma}^{-1/2}_{(0)}\bB\ba\|_2^4},
\end{align*}
where the second equation is due to that $\bx_j,\bxi_{-j}$ are independent and they are independent from $\mathcal{D}_0,\bA,\bB$. Recall that the last upper bound has been controlled in the analysis of $\cJ_1$. Hence we have $\cJ_2\lesssim \frac{\theta^2+\theta}{n^2}+\frac{p^2}{n^4}$. The same bound holds for $\cJ_3$ using the same arguments. Therefore, we can conclude that each term in the summation of \eqref{efron-stein:cond} is $O(\frac{\theta^2+\theta}{n^2}+\frac{p^2}{n^4})$. This result combined with \eqref{start:from} leads to $\mathbb{E}(\mbox{var}(\hat{\theta}_{sg}\mid \mathcal{D}_0))\lesssim \frac{\theta^2+\theta}{n}+\frac{p}{n^2}$.
\end{proof}

\begin{lem}
\label{identity:deviation}
It is assumed in (ii)--(iv) that $p=O(n^{\alpha})$ for some $\alpha<1$ and $c_L\leq \lambda_{\min}(\bSigma)\leq \delta_{\bSigma}\leq c_U$ for two constants $c_L,c_U>0$.
\begin{itemize}
\item[(i)] $\bA^{-1}-\sum_{k=0}^m(\bI-\bA)^k=\bA^{-1/2}(\bI-\bA)^{m+1}\bA^{-1/2},\quad ~ \forall ~\bA \succ 0, ~m=1,2,\ldots$
\item[(ii)] $\mathbb{E}\|\bSigma^{-1/2}\hat{\bSigma}_{(0)}\bSigma^{-1/2}-\bI_p\|^k_2\lesssim \Big(\frac{p}{n}\Big)^{k/2}, \quad \forall ~k=1,2,\ldots$
\item[(iii)] $\mathbb{E}\|\bSigma^{1/2}\tilde{\bSigma}_{(0)}^{-1}\bSigma^{1/2}-\bI_p\|^k_2\lesssim \Big(\frac{p}{n}\Big)^{k/2}, \quad \forall ~k=1,2,\ldots$
\item[(iv)] $\mathbb{E}\|\tilde{\bSigma}_{(0)}^{-1/2}\bSigma\tilde{\bSigma}_{(0)}^{-1/2}-\bI_p\|^k_2\lesssim \Big(\frac{p}{n}\Big)^{k/2}, \quad \forall ~k=1,2,\ldots$
\item[(v)] $\mathbb{E}(\hat{\bmu}_{(0)}^T\bSigma^{-1}\hat{\bmu}_{(0)})^{2k}\lesssim \theta^{2k}+\Big(\frac{\theta}{n}\Big)^k+\Big(\frac{p}{n}\Big)^{2k}, \quad \forall ~k=1,2,\ldots$
\end{itemize}
\end{lem}

\begin{proof}
Part (i): It is direct to use the spectral decomposition of $\bA$ to confirm the identity in (i).   \\
Part (ii): First note that $\bSigma^{-1/2}\hat{\bSigma}_{(0)}\bSigma^{-1/2}=\hat{\bSigma}_{\by}-\bar{\by}\bar{\by}^T$, where $\hat{\bSigma}_{\by}=\ell^{-1}_n\sum_{i=1}^{\ell_n}\by_i\by_i^T, \bar{\by}=\ell_n^{-1}\sum_{i=1}^{\ell_n}\by_i$, $\ell_n=n-m\lfloor \frac{n}{m+1}\rfloor$, and $\by_i$'s are i.i.d. from the zero-mean isotropic sub-Gaussian distribution. According to Theorem 5.39 in \cite{versh10} and that $\ell_n$ is of order $n$, it holds
\begin{align}
\label{spectral:bound}
\mathbb{P}\Big(\|\hat{\bSigma}_{\by}-\bI_p\|_2> \max(\delta, \delta^2)\Big)\leq 2\exp(-ct^2), \quad \forall t\geq 0,~~~\delta=C\sqrt{\frac{p}{n}}+\frac{t}{\sqrt{n}},
\end{align}
for some constants $c,C>0$. We use the above non-asymptotic result to bound 
\begin{align*}
\mathbb{E}\|\hat{\bSigma}_{\by}-\bI_p\|^k_2 =&\int_{0}^{C\sqrt{\frac{p}{n}}}ku^{k-1}\mathbb{P}(\|\hat{\bSigma}_{\by}-\bI_p\|_2>u)du+\int_{C\sqrt{\frac{p}{n}}}^1ku^{k-1}\mathbb{P}(\|\hat{\bSigma}_{\by}-\bI_p\|_2>u)du \\
&+\int_1^{\infty}ku^{k-1}\mathbb{P}(\|\hat{\bSigma}_{\by}-\bI_p\|_2>u)du:=\cJ_1+\cJ_2+\cJ_3
\end{align*}
It is obvious that $\cJ_1\leq \int_{0}^{C\sqrt{\frac{p}{n}}}ku^{k-1}du\leq \big(C\sqrt{\frac{p}{n}}\big)^k \lesssim (\frac{p}{n})^{k/2}$. Using \eqref{spectral:bound} and a change of variable $v=\sqrt{\frac{n}{p}}u-C$, we obtain 
\begin{align*}
\cJ_2&=\Big(\frac{p}{n}\Big)^{k/2}\int_{0}^{\sqrt{\frac{n}{p}}-C}k(v+C)^{k-1}\cdot \mathbb{P}\big(\|\hat{\bSigma}_{\by}-\bI_p\|_2>(C+v)\sqrt{p/n}\big)dv \\
&\leq 2\Big(\frac{p}{n}\Big)^{k/2}\int_{0}^{\sqrt{\frac{n}{p}}-C}k(v+C)^{k-1}e^{-cpv^2}dv \lesssim \Big(\frac{p}{n}\Big)^{k/2}.
\end{align*}
Similarly, we use \eqref{spectral:bound} and a change of variable $v=\sqrt{\frac{n}{p}}\cdot \sqrt{u}-C$ to bound $\cJ_3$,
\begin{align*}
\cJ_3 \lesssim \Big(\frac{p}{n}\Big)^{k} \int_{\sqrt{\frac{n}{p}}-C}^{\infty}(v+C)^{2k-1}e^{-cpv^2}dv\lesssim \Big(\frac{p}{n}\Big)^{k/2}
\end{align*}
So far we have showed that $\mathbb{E}\|\hat{\bSigma}_{\by}-\bI_p\|^k_2\lesssim (\frac{p}{n})^{k/2}$. To finish the proof, it is sufficient to further show that $\mathbb{E}\|\bar{\by}\bar{\by}^T\|_2^{k}=\mathbb{E}\|\bar{\by}\|_2^{2k}\lesssim (\frac{p}{n})^{k/2}$. Towards this end, note that $\bar{\by}$ is sub-Gaussian with $\|\bar{\by}\|_{\psi_2}\lesssim n^{-1/2}$, therefore
\begin{align*}
\mathbb{E}\|\bar{\by}\|_2^{2k}=\mathbb{E}\Big(\sum_{j=1}^p\bar{\by}^2_j\Big)^k\leq p^{k-1}\cdot \sum_{j=1}^p\mathbb{E}\bar{\by}^{2k}_j \lesssim \Big(\frac{p}{n}\Big)^k\lesssim \Big(\frac{p}{n}\Big)^{k/2}.
\end{align*}
Part (iii): Define the probability event $\mathcal{E}=\{\lambda_{\min}(\bSigma^{-1/2}\tilde{\bSigma}_{(0)}\bSigma^{-1/2})>1/3\}$. We have
\begin{align*}
&\mathbb{E}\|\bSigma^{1/2}\tilde{\bSigma}_{(0)}^{-1}\bSigma^{1/2}-\bI_p\|^k_2 \\
=&  \mathbb{E}\big(\|\bSigma^{1/2}\tilde{\bSigma}_{(0)}^{-1}\bSigma^{1/2}-\bI_p\|^k_2\cdot \mathbbm{1}_{\mathcal{E}}\big) + \mathbb{E}\big(\|\bSigma^{1/2}\tilde{\bSigma}_{(0)}^{-1}\bSigma^{1/2}-\bI_p\|^k_2\cdot \mathbbm{1}_{\mathcal{E}^c}\big):=\cJ_1+\cJ_2.
\end{align*}
We bound $\cJ_1$ in the following,
\begin{align*}
\cJ_1&\leq 3^k\cdot \mathbb{E}(\|\bSigma^{-1/2}\tilde{\bSigma}_{(0)}\bSigma^{-1/2}-\bI_p\|_2^k\cdot \mathbbm{1}_{\mathcal{E}}) \leq 3^k\cdot \mathbb{E}\|\bSigma^{-1/2}\tilde{\bSigma}_{(0)}\bSigma^{-1/2}-\bI_p\|_2^k \\
&\lesssim \mathbb{E}\|\bSigma^{-1/2}\hat{\bSigma}_{(0)}\bSigma^{-1/2}-\bI_p\|_2^k+\|\epsilon_n\bSigma^{-1}\|_2^k\lesssim \Big(\frac{p}{n}\Big)^{k/2},
\end{align*}
where we have used Part (ii) and the conditions $\lambda_{\min}(\bSigma)\geq c_L, \epsilon_n=\sqrt{\frac{p}{n}}$ in the last inequality. To bound $\cJ_2$, we first get
\[
\|\bSigma^{1/2}\tilde{\bSigma}_{(0)}^{-1}\bSigma^{1/2}-\bI_p\|_2\leq \|\bSigma\|_2\cdot \|\tilde{\bSigma}_{(0)}^{-1}\|_2+1\lesssim p\sqrt{\frac{n}{p}}+1,
\]
where the last result is because $\|\bSigma\|_2\leq \|\bSigma\|_F\leq pc_U$. Further use Part (ii) to derive
\begin{align*}
\mathbb{P}(\mathcal{E}^c)&\leq \mathbb{P}(\lambda_{\min}(\bSigma^{-1/2}\hat{\bSigma}_{(0)}\bSigma^{-1/2})\leq 1/3)\leq \mathbb{P}(\|\bSigma^{-1/2}\hat{\bSigma}_{(0)}\bSigma^{-1/2}-\bI_p\|_2\geq 2/3)
 \\
&\leq (3/2)^{\ell}\cdot \mathbb{E}\|\bSigma^{-1/2}\hat{\bSigma}_{(0)}\bSigma^{-1/2}-\bI_p\|_2^{\ell}\lesssim \Big(\frac{p}{n}\Big)^{\ell/2}
\end{align*}
Choosing $\ell=\lceil \frac{2k}{1-\alpha}\rceil$ above combined with the condition $p=O(n^{\alpha})$ we can obtain
\begin{align*}
\cJ_2\leq (1+p\sqrt{n/p})^k\cdot \mathbb{P}(\mathcal{E}^c)\lesssim p^{\frac{k+\ell}{2}}n^{\frac{k-\ell}{2}}\lesssim \Big(\frac{p}{n}\Big)^{k/2}.
\end{align*}
Part (iv): Combine Part (iii) and the simple identity $\|\tilde{\bSigma}_{(0)}^{-1/2}\bSigma\tilde{\bSigma}_{(0)}^{-1/2}-\bI_p\|_2=\|\bSigma^{1/2}\tilde{\bSigma}_{(0)}^{-1}\bSigma^{1/2}-\bI_p\|_2$. \\
Part (v): Note that $\hat{\bmu}_{(0)}^T\bSigma^{-1}\hat{\bmu}_{(0)}=\theta+2\bar{\by}^T\bSigma^{-1/2}\bmu+\|\bar{\by}\|_2^2$. Therefore, 
\begin{align*}
\mathbb{E}(\hat{\bmu}_{(0)}^T\bSigma^{-1}\hat{\bmu}_{(0)})^{2k}\lesssim \theta^{2k}+\mathbb{E}(|\bar{\by}^T\bSigma^{-1/2}\bmu|^{2k})+\mathbb{E}\|\bar{\by}\|_2^{4k}.
\end{align*}
Given that $\bar{\by}$ has sub-Gaussian norm $\|\bar{\by}\|_{\psi_2}\lesssim n^{-1/2}$, we have $\mathbb{E}(|\bar{\by}^T\bSigma^{-1/2}\bmu|^{2k})\lesssim \frac{\theta^k}{n^k}$. Finally, $\mathbb{E}\|\bar{\by}\|_2^{4k}\lesssim \frac{p^{2k}}{n^{2k}}$ is obtained in the proof of Part (ii). 
\end{proof}

\subsubsection{Further discussions in the dense regime}

As pointed out in \textbf{Remark 8}, the minimax rate of functional estimation under the scaling $p\lesssim s^2, n\lesssim  p\lesssim n^2$ remains unknown. Here we provide some partial results with extra assumptions made on the model. To fix the idea, we focus on Gaussian distributions throughout this section. 

\begin{itemize}
\item[(1)] \emph{Block structures on $\bSigma$}. Suppose the covariance matrix $\bSigma$ has a known block structure (up to a permutation): 
\begin{align*}
\bSigma=
\begin{pmatrix}
\bSigma_{11} & & &\\
& \bSigma_{22} & & \\
&& \ddots & \\
& & & \bSigma_{kk}
\end{pmatrix}
, \quad \bSigma_{jj} \in \mathbb{R}^{d_j\times d_j},~~ j=1,2,\ldots, k,
\end{align*}
where the number of blocks $k$ can depend on $n$. Let $d=\max_{1\leq j\leq k}d_j$ be the maximum block size, and formulate this block structure in a parameter space denoted by $\mathcal{B}(d)=\{(\bmu,\bSigma): \bSigma=\diag(\bSigma_{11},\ldots, \bSigma_{kk})\}$. Observe that for $(\bmu,\bSigma)\in \mathcal{B}(d)$, if we split $\bmu$ into $k$ corresponding blocks as $\bmu=(\bmu^T_1,\ldots,\bmu^T_k)^T$, then $\theta=\sum_{j=1}^k \bmu_j^T\bSigma_{jj}^{-1}\bmu_j$ and it is natural to estimate $\bmu_j^T\bSigma_{jj}^{-1}\bmu_j$ within each block using the debiasing idea from Proposition \ref{gaussian:case}. This motivates us to consider the functional estimator: 
\[
\hat{\theta}_b=\sum_{j=1}^k\frac{n-d_j-2}{n}\hat{\bmu}_{[j]}^T\hat{\bSigma}^{-1}_{[j]}\hat{\bmu}_{[j]}-\frac{p}{n}, 
\]
where $\hat{\bmu}_{[j]},\hat{\bSigma}_{[j]}$ are the sample mean and sample covariance matrix from the $j$th block respectively. Suppose $p\lesssim s^2, n\lesssim  p\lesssim n^2, \limsup_{n\rightarrow \infty}\frac{d}{n}<1$, then
\begin{align*}
\inf_{\hat{\theta}}\sup_{(\bmu,\bSigma)\in \mathcal{H}(s,\tau)\cap \mathcal{B}(d)}\mathbb{E}|\hat{\theta}-\theta| \asymp \Big[\tau \wedge \frac{\sqrt{p}}{n}\Big]+\Big[\tau \wedge  \frac{\tau+\sqrt{\tau}}{\sqrt{n}}\Big].
\end{align*}
The estimator $\hat{\theta}_b$ is minimax rate-optimal when $\tau \gtrsim \frac{\sqrt{p}}{n}$, while the trivial estimator $\bzero$ is optimal when $\tau \lesssim \frac{\sqrt{p}}{n}$.
\begin{proof}
The lower bound has been essentially derived in the proof of Proposition \ref{lowerbound:dense}, as some parameter configurations considered there belong to $\mathcal{H}(s,\tau)\cap \mathcal{B}(d)$ as well. Specifically, the lower bound term $\tau \wedge  \frac{\tau+\sqrt{\tau}}{\sqrt{n}}$ continues to hold in the current case. The other term $\tau \wedge \frac{\sqrt{p}}{n}$ follows from the argument in Case (ii) of Section \ref{proof:denselower}. To obtain the matching upper bound, it is clear that in the regime where $\tau \lesssim \frac{\sqrt{p}}{n}$, the lower bound rate reduces to $\tau$ which is attained by the trivial estimator $\bzero$. When $\tau \gtrsim \frac{\sqrt{p}}{n}$, it remains to show $\hat{\theta}_b$ achieves the optimal rate $\frac{\sqrt{p}}{n}+\frac{\tau+\sqrt{\tau}}{\sqrt{n}}$. Under the block structure and condition $\limsup_{n\rightarrow \infty}\frac{b}{n}<1$, we directly apply the result of Proposition \ref{gaussian:case} within each block to obtain
\begin{align*}
\mathbb{E}|\hat{\theta}_b-\theta| &\lesssim \sqrt{\var(\hat{\theta}_b)}= \Bigg[\sum_{j=1}^k \var\Big(\frac{n-d_j-2}{n}\hat{\bmu}_{[j]}^T\hat{\bSigma}^{-1}_{[j]}\hat{\bmu}_{[j]}-\frac{d_j}{n}\Big)\Bigg]^{1/2} \\
&\lesssim \Bigg[\sum_{j=1}^k\Big(\frac{d_j}{n^2}+\frac{\bmu_j^T\bSigma_{jj}^{-1}\bmu_j+(\bmu_j^T\bSigma_{jj}^{-1}\bmu_j)^2}{n}\Big)\Bigg]^{1/2}\lesssim \frac{\sqrt{p}}{n}+\frac{\theta+\sqrt{\theta}}{\sqrt{n}},
\end{align*}
where we have used the mutual independence between the blocks, and the identities $p=\sum_{j=1}^kd_j, \theta=\sum_{j=1}^d\bmu_j^T\bSigma_{jj}^{-1}\bmu_j$.
\end{proof}
\item[(2)] Support recovery of $\bSigma^{-1}\bmu$. Let $S$ be the support of $\bSigma^{-1}\bmu$. It is direct to confirm that $\theta=\bmu_{S}^T\bSigma^{-1}_{SS}\bmu_{S}$. Recall that we assume $\bSigma^{-1}\bmu$ is an $s$-sparse vector, i.e., $|S|\leq s$. In the situations where support recovery is possible, one appealing idea of constructing functional estimator is to first identify the support $S$ and then employ the debiasing from Proposition \ref{gaussian:case} in the reduced space. Thusly motivated, we split the data into two parts $\mathcal{D}_1,\mathcal{D}_2$ with sample size $n_1=\lfloor \frac{n}{2} \rfloor ,n_2=n-\lfloor \frac{n}{2}\rfloor$, use $\mathcal{D}_1$ to estimate the support, and perform debiasing based on $\mathcal{D}_2$. Specifically, we first obtain $\hat{S}$ as
\begin{align}
\label{support:rec}
\hat{S}=\{j\in[p]: (\hat{\balpha}_{sr})_j\neq 0 \}, \quad \hat{\balpha}_{sr} \in \argmin_{\lonenorm{\bbeta} \leq \gamma} \frac{1}{2}\bbeta^T \hat{\bSigma}^{(1)}\bbeta-\bbeta^T \hat{\bmu}^{(1)}+\sum_{j=1}^p\rho_{\lambda}(\beta_j),
\end{align}
where $\rho_{\lambda}(\cdot)$ is the MCP \citep{z10} or SCAD \citep{fl01} penalty function, and $ \hat{\bmu}^{(1)}, \hat{\bSigma}^{(1)}$ are the sample mean and unbiased sample covariance matrix from $\mathcal{D}_1$ respectively. We then construct the functional estimator 
\begin{align*}
\hat{\theta}_{sr}=
\begin{cases}
\frac{n_2-|\hat{S}|-2}{n_2}(\hat{\bmu}^{(2)}_{\hat{S}})^T\big(\hat{\bSigma}^{(2)}_{\hat{S}\hat{S}}\big)^{-1}\hat{\bmu}^{(2)}_{\hat{S}}-\frac{|\hat{S}|}{n_2} & \text{if}~|\hat{S}|\leq \frac{n}{4} \\
0 & \text{otherwise}
\end{cases}
\end{align*}
Here, $\hat{\bmu}^{(2)}, \hat{\bSigma}^{(2)}$ are the sample mean and sample covariance matrix from $\mathcal{D}_2$ respectively. To ensure support recovery we impose a beta-min condition: 
\[
\mathcal{G}=\Big\{(\bmu,\bSigma): \min_j\big |(\bSigma^{-1}_{SS}\bmu_S)_j \big |\gg \sqrt{\frac{\log p}{n}} \cdot \big[\|\bSigma^{-1/2}_{SS}\|^2_{\infty}\cdot (1\vee \sqrt{|S|^2/n})+1\big]\Big\}.
\]
Consider the common scenario where $\frac{s\log p}{n}=o(1), \tau \asymp 1$, and $\exists \delta>0 ~{\rm s.t.}~ p^{-\delta}\lesssim n^{-1/2}$, then
\begin{align*}
\inf_{\hat{\theta}}\sup_{(\bmu,\bSigma)\in \mathcal{H}(s,\tau)\cap \mathcal{G}}\mathbb{E}|\hat{\theta}-\theta| \asymp \frac{1}{\sqrt{n}}.
\end{align*}
The estimator $\hat{\theta}_{sr}$ with tuning parameters $\lambda\asymp \sqrt{\frac{\log p}{n}}, \gamma \asymp \sqrt{\frac{n}{\log p}}$ attains the optimal rate. Before presenting the proof, we make a few remarks below. 
\begin{enumerate}
\item[(a)] The support recovery of nonconvex regularized M-estimators has been systematically studied in \cite{lw17}. We will apply the general theorem from \cite{lw17} to establish variable selection consistency for our problem. Notably, the theorem certifies that the nonconvex optimization \eqref{support:rec} has a unique stationary point and it recovers the support without the typical incoherence conditions required for $\ell_1$-regularized methods. 
\item[(b)] Estimating the functional in the variable selection consistency regime becomes comparatively more tractable, as the minimax estimation admits the regular parametric rate. The optimal rate is attained by performing debiasing on the estimated support. A quick inspection of Proposition \ref{gaussian:case} indicates that the same debiasing can work when the estimated support contains the true one as a subset while allowing for false positives. This weaker notion is commonly referred to as the sure screening property \citep{fl08}. It is possible to explore sure screening methods \citep{fl18} combined with the debiasing to derive rate-optimal functional estimator under a different set of conditions. We do not include a detailed treatment here for simplicity. 
\end{enumerate}
\begin{proof}
The lower bound $\frac{1}{\sqrt{n}}$ follows directly from the argument in Case (i) of Section \ref{lowerbound:theta}, as the parameter configurations used there belong to $\mathcal{H}(s,\tau)\cap \mathcal{G}$ as well. We now show $\hat{\theta}_{sr}$ achieves the same rate. Suppose $\PP(\hat{S}\neq S)\lesssim p^{-\delta}+n^{-1/2}$ for the moment. Define the event $\mathcal{L}=\{|\hat{S}|\leq n/4\}$. For a given set $R\subseteq [p]$, denote $\theta_R=\bmu_{R}^T\bSigma^{-1}_{RR}\bmu_{R}$. We have
\begin{align*}
\mathbb{E}|\hat{\theta}_{sr}-\theta| \leq \mathbb{E}(|\hat{\theta}_{sr}-\theta_{\hat{S}}|\cdot \mathbbm{1}_{\mathcal{L}})+\mathbb{E}(|\theta_{\hat{S}}-\theta|\cdot \mathbbm{1}_{\mathcal{L}})+\mathbb{E}(|\hat{\theta}_{sr}-\theta|\cdot \mathbbm{1}_{\mathcal{L}^c}):=\mathcal{J}_1+\mathcal{J}_2+\mathcal{J}_3.
\end{align*}
First note that the conditions $\frac{s\log p}{n}=o(1), \tau \asymp 1, p^{-\delta}\lesssim n^{-1/2}$ imply that $\mathcal{J}_3=\theta\cdot \PP(\mathcal{L}^c)\lesssim \PP(\hat{S}\neq S)\lesssim n^{-1/2}$. It is also clear that $\mathcal{J}_2\leq \mathbb{E}(|\theta_{\hat{S}}-\theta|\lesssim \PP(\hat{S}\neq S)\lesssim n^{-1/2}$. Regarding $J_1$,
\begin{align*}
\mathcal{J}_1&=\sum_{|R|\leq n/4} \PP(\hat{S}=R)\cdot \mathbb{E}\big|(n_2-|R|-2)n_2^{-1}(\hat{\bmu}^{(2)}_{R})^T(\hat{\bSigma}^{(2)}_{RR})^{-1}\hat{\bmu}^{(2)}_{R}-|R|n_2^{-1}-\theta_{R}\big| \\
&\lesssim n^{-1/2}\cdot \sum_{|R|\leq n/4} \PP(\hat{S}=R) \lesssim n^{-1/2},
\end{align*}
where we have used the independence of $\hat{S}$ from $\hat{\bmu}^{(2)}, \hat{\bSigma}^{(2)}$ in the first equality; the second inequality follows from the proof of Proposition \ref{gaussian:case} and the fact that $\theta_R\leq \theta \lesssim 1$. It remains to prove $\PP(\hat{S}\neq S)\lesssim p^{-\delta}+n^{-1/2}$. Like the proof of Corollary 1(b) in \cite{lw17}, we verify the conditions of Theorem 2(b) in \cite{lw17} to obtain the variable selection consistency in our problem. With some minor abuse of notation, below all the statistics are computed from $\mathcal{D}_1$. 
\begin{enumerate}
\item[(i)] Restricted strong convexity of $\frac{1}{2}\bbeta^T \hat{\bSigma}\bbeta-\bbeta^T \hat{\bmu}$. With probability greater than $1-p^{-\delta}$ it holds that
\begin{align*}
\bu^T\hat{\bSigma}\bu&\geq -|\bu^T(\hat{\bmu}-\bmu)|^2+\frac{1}{2}\bu^T\bSigma\bu-\Big| \Big(n^{-1}\sum_{i=1}^n|\bu^T(\bx_i-\bmu)|^2\Big)^{\frac{1}{2}}-\sqrt{\bu^T\bSigma \bu} \Big|^2 \\
&\geq c_1 \ltwonorm{\bu}^2-c_2\frac{\log p}{n}\lonenorm{\bu}^2, \quad \forall \bu \in \mathbb{R}^p,
\end{align*}
where the last inequality holds due to \eqref{bound:second:term} and \eqref{bound:first:term}.
\item[(ii)] Choice of the tuning parameters $\lambda, \gamma$. The orders $\lambda\asymp \sqrt{\frac{\log p}{n}}, \gamma \asymp \sqrt{\frac{n}{\log p}}$ satisfy the requirements specified in Theorem 1 of \cite{lw17}, because $\frac{s\log p}{n}=o(1), \lonenorm{\balpha}\lesssim \sqrt{s}$, and $\linfnorm{\hat{\bSigma}\balpha-\hat{\bmu}}\lesssim \sqrt{\frac{\log p}{n}}$ from Lemma \ref{rare:events} Part (i).  
\item[(iii)] Strict dual feasibility. As in the proof of Corollary 1(b) in \cite{lw17}, the strict dual feasibility will hold if (a) $\linfnorm{\hat{\bSigma}_{S^cS}\hat{\bSigma}^{-1}_{SS}\hat{\bmu}_S-\hat{\bmu}_{S^c}}\lesssim \sqrt{\frac{\log p}{n}}$; (b) $\linfnorm{\hat{\bSigma}_{SS}\balpha_S-\hat{\bmu}_S} \lesssim \sqrt{\frac{\log p}{n}}$; (c) $\linfnorm{\hat{\bSigma}_{SS}^{-1}}\lesssim \|\bSigma^{-1/2}_{SS}\|^2_{\infty}\cdot (1+ \sqrt{|S|^2/n})$. Firstly, (b) is implied by $\linfnorm{\hat{\bSigma}\balpha-\hat{\bmu}}\lesssim \sqrt{\frac{\log p}{n}}$ from (ii). To prove (c), we have 
\begin{align*}
\linfnorm{\hat{\bSigma}_{SS}^{-1}} &\leq \linfnorm{\hat{\bSigma}_{SS}^{-1}-\bSigma^{-1}_{SS}}+\linfnorm{\bSigma_{SS}^{-1}}\leq \linfnorm{\bSigma_{SS}^{-1/2}}^2\cdot (1+ \linfnorm{(\bSigma_{SS}^{-1/2}\hat{\bSigma}_{SS}\bSigma_{SS}^{-1/2})^{-1}-\bI_{|S|}}) \\
&\hspace{-1.cm}\leq \linfnorm{\bSigma_{SS}^{-1/2}}^2\cdot (1+ \sqrt{|S|}\ltwonorm{(\bSigma_{SS}^{-1/2}\hat{\bSigma}_{SS}\bSigma_{SS}^{-1/2})^{-1}-\bI_{|S|}})\lesssim \|\bSigma^{-1/2}_{SS}\|^2_{\infty}\cdot (1+ \sqrt{|S|^2/n}),
\end{align*}
where the last inequality holds with probability at least $1-p^{-\delta}$, and this is due to the spectral norm bound $\ltwonorm{(\bSigma_{SS}^{-1/2}\hat{\bSigma}_{SS}\bSigma_{SS}^{-1/2})^{-1}-\bI_{|S|}}\lesssim \sqrt{|S|/n}+\sqrt{\log p/n}$ which can be derived using a similar argument as in the proof of Lemma \ref{identity:deviation} Part (ii). Regarding (a), let $\bX_S\in \mathbb{R}^{n_1\times |S|}, \bX_{S^c}\in \mathbb{R}^{n_1\times (p-|S|)}$ be the data matrices indexed by $S, S^c$ respectively. It is straightforward to use the Woodbury formula to confirm the identity 
\[
\hat{\bSigma}_{S^cS}\hat{\bSigma}^{-1}_{SS}\hat{\bmu}_S-\hat{\bmu}_{S^c}=-[\underbrace{n_1^{-1}\bX^T_{S^c}(\bI_{n}-\bX_S(\bX_S^T\bX_S)^{-1}\bX_S^T)\bone}_{:=\mathcal{J}_1}]\cdot [\underbrace{1+\hat{\bmu}_S^T(\hat{\bSigma}_{SS})^{-1}\hat{\bmu}_S}_{:=\mathcal{J}_2}]
\]
We first prove $\linfnorm{\mathcal{J}_1}\lesssim \sqrt{\frac{\log p}{n}}$. Let $\bz_j$ be the $j$th column of $\bX_{S^c}$. Given that $\balpha=\bSigma^{-1}\bmu$ has the support $S$, it is direct to verify $\bSigma_{S^cS}\bSigma^{-1}_{SS}\bmu_S=\bmu_{S^c}$, which implies that $\bz_j\mid \bX_S\sim N(\bv_j, \sigma^2_j\bI_{n_1})$ with $\bv_j=\bX_S\bSigma^{-1}_{SS}\bSigma_{SS^c}\be_j, \sigma^2_j=\be_j^T(\bSigma_{S^cS^c}-\bSigma_{S^cS}\bSigma_{SS}^{-1}\bSigma_{SS^c})\be_j$. Then $(\mathcal{J}_1)_j=n_1^{-1}\bz^T_j(\bI_{n}-\bX_S(\bX_S^T\bX_S)^{-1}\bX_S^T)\bone\mid \bX_S \sim N(0, \sigma_j^2n_1^{-2}\bone^T(\bI_{n}-\bX_S(\bX_S^T\bX_S)^{-1}\bX_S^T)\bone)$. Therefore, the Gaussian tail inequality gives $\forall t>0$,
\begin{align*}
\PP\Big(|(\mathcal{J}_1)_j|\geq t\sqrt{\frac{\log p}{n}}\Big)\leq 2\EE \exp\Big(\frac{-t^2\log p}{2n\sigma_j^2n_1^{-2}\bone^T(\bI_n-\bX_S(\bX_S^T\bX_S)^{-1}\bX_S^T)\bone}\Big)\lesssim 2e^{-c_2t^2\log p}.
\end{align*}
Choosing sufficiently large $t>0$, this result combined with a simple union bound yields $\linfnorm{\mathcal{J}_1}\lesssim \sqrt{\frac{\log p}{n}}$ with probability greater than $1-p^{-\delta}$. Finally, we should prove $\mathcal{J}_2\lesssim 1$. Towards this goal, we write $\hat{\bmu}_S^T(\hat{\bSigma}_{SS})^{-1}\hat{\bmu}_S=\frac{n_1}{n_1-|S|-2}\cdot [\frac{n_1-|S|-2}{n_1}\hat{\bmu}_S^T(\hat{\bSigma}_{SS})^{-1}\hat{\bmu}_S-\frac{|S|}{n_1}]+\frac{|S|}{n_1-|S|-2}$. Under the scaling $\frac{s\log p}{n}=o(1)$, combining the variance calculations in the proof of Proposition \ref{gaussian:case} with Chebyshev's inequality leads to $\PP(\hat{\bmu}_S^T(\hat{\bSigma}_{SS})^{-1}\hat{\bmu}_S\gtrsim 1)\lesssim n^{-1}$. 
\end{enumerate}
\end{proof} 
\end{itemize}

\subsection{Reference Material}
\label{reference:material}
\begin{thmx} \label{hoeffding:quote}
(General Hoeffding's inequality). Let $x_1, \ldots, x_n \in \RR$ be independent, zero-mean, sub-gaussian random variables. Then for every $t \geq 0$, we have
\begin{eqnarray*}
\mathbb{P}\bigg(\Big|\sum_{i=1}^nx_i \Big|\geq t \bigg) \leq 2\exp \Big(-\frac{ct^2}{\sum_{i=1}^n\|x_i\|^2_{\psi_2}}\Big),
\end{eqnarray*}
where $c>0$ is an absolute constant, and $\|\cdot \|_{\psi_2}$ is the sub-gaussian norm defined as $\|x \|_{\psi_2}=\inf\{t>0: \mathbb{E}e^{x^2/t^2}\leq 2\}$. \\

\noindent (Bernstein's inequality). Let $x_1,\ldots, x_n \in \RR $ be independent, zero-mean, sub-exponential random variables. Then for every $t\geq 0$, we have
\begin{eqnarray*}
\mathbb{P}\bigg(\Big|\sum_{i=1}^nx_i \Big|\geq t \bigg) \leq 2\exp\Bigg[-c \min \bigg(\frac{t^2}{\sum_{i=1}^n\|x_i\|^2_{\psi_1}}, \frac{t}{\max_i \|x_i\|_{\psi_1}} \bigg) \Bigg],
\end{eqnarray*}
where $c>0$ is an absolute constant, and $\|\cdot\|_{\psi_1}$ is the sub-exponential norm defined as $\|x\|_{\psi_1}=\inf\{t>0: \mathbb{E}e^{|x|/t}\leq 2\}$.
\end{thmx}

The above two results are Theorem 2.6.2 and Theorem 2.8.1, respectively in \cite{versh18}. Please refer there to see the details.

\begin{thmx}\label{maximal:noteq}
Let $x_1,\ldots, x_p \in \mathbb{R}$ be sub-gaussian random variables, which are not necessarily independent. Then there exists an absolute constant $c>0$ such that for all $p>1$,
\begin{eqnarray*}
\mathbb{E}\max_{1\leq i\leq p} |x_i| \leq c \sqrt{\log p} \max_{1\leq i\leq p} \|x_i\|_{\psi_2}.
\end{eqnarray*}
\end{thmx}

The above result can be found in Lemma 2.4 of \cite{blm13}.

\begin{thmx}\label{matrix:deviation}
(Matrix deviation inequality). Let $\bA$ be an $m \times n$ matrix whose rows $\{\bA_i\}_{i=1}^n$are independent, isotropic and sub-gaussian random vectors in $\RR^n$. Then for any subset $\cT \subseteq \mathbb{R}^n$, we have for any $u \geq 0$, the event
\begin{eqnarray*}
\sup_{\bx\in \cT}\big |\|\bA\bx\|_2-\sqrt{m}\|\bx\|_2\big |\leq CK^2(w(\cT)+u \cdot {\rm rad}(\cT))
\end{eqnarray*}
holds with probability at least $1-2e^{-u^2}$. Here $C>0$ is an absolute constant, $K=\max_i \|\bA_i\|_{\psi_2}$, and $w(\cT), {\rm rad}(\cT)$ are defined as:
\begin{eqnarray*}
w(T)=\mathbb{E}\sup_{\bx \in \cT}\bg^T\bx, ~\bg \sim N(0, \bI_p); ~~ {\rm rad}(\cT)=\sup_{\bx\in \cT}\|\bx\|_2.
\end{eqnarray*}
\end{thmx}
Theorem \ref{matrix:deviation} appears as Exercise 9.1.8 in \cite{versh18}.

\begin{singlespace}

\end{singlespace}

\end{document}